\documentclass[12pt,dvips]{amsart}
\usepackage{amsmath}
\usepackage{amssymb}
\usepackage{amscd}
\usepackage{latexsym}
\usepackage{palatino}
\usepackage{amsthm} % GDL
\usepackage{epsfig}
\usepackage{psfrag}
\usepackage{amsfonts}
\usepackage{enumerate}
\usepackage{bbold}
\usepackage{comment}
\usepackage{nomencl}

\numberwithin{equation}{section}
\numberwithin{figure}{section}

\theoremstyle{plain}

\newtheorem{theorem}{Theorem}[section]
\newtheorem{lemma}[theorem]{Lemma}

\newtheorem{proposition}[theorem]{Proposition}

\newtheorem{corollary}[theorem]{Corollary}

\theoremstyle{definition}

\newtheorem{definition}[theorem]{Definition}

\theoremstyle{remark}

\newtheorem{example}[theorem]{Example}

\newtheorem{remark}[theorem]{Remark}

\usepackage{multicol}

\makeatletter

\def\theglossary{\@restonecoltrue\if@twocolumn\@restonecolfalse\fi
\columnseprule\z@ \columnsep 35\p@
\let\@makessectionhead\indexsec
\@xp\section\@xp*\@xp{\glossaryname}%
\let\item\@idxitem
\parindent\z@  \parskip\z@\@plus.3\p@\relax
\footnotesize}

\def\glossaryname{Notation Index}

\makeatother

\makeglossary

\input xy
 \xyoption{all}

\renewcommand{\to}{\rightarrow}
\newcommand{\into}{\hookrightarrow}

\renewcommand{\iff}{\Leftrightarrow}

\renewcommand{\Im}{\operatorname{Im}}
\newcommand{\im}{\operatorname{im}}
\newcommand{\id}{\operatorname{id}}

\renewcommand{\dim}{\operatorname{dim}}

\newcommand{\Lie}{{\operatorname{Lie}}}

\DeclareMathOperator{\Vect}{Vect}

\newcommand{\Z}{{\mathbb Z}}

\newcommand{\Q}{{\mathbb Q}}
\newcommand{\R}{{\mathbb R}}
\newcommand{\C}{{\mathbb C}}

\newcommand{\gl}{{\mathfrak gl}}
\newcommand{\g}{{\mathfrak g}}
\newcommand{\pg}{{\mathfrak p}{\mathfrak g}}

\renewcommand{\u}{{\mathfrak u}}

\renewcommand{\mod}{{/\!\!/}}

\newcommand{\mmod}{{/\!\!/\!\!/\!\!/}}
\newcommand{\PhiR}{{\Phi_{\R}}}
\newcommand{\PhiC}{{\Phi_{\C}}}
\newcommand{\PhiCzero}{{\PhiC^{-1}(0)}}

\newcommand{\ddt}{\left. \frac{d}{dt} \right|_{t=0}}

\UseComputerModernTips

\DeclareMathOperator{\grad}{grad}

\DeclareMathOperator{\stab}{stab}

\DeclareMathOperator{\Gr}{Gr}

\DeclareMathOperator{\Crit}{Crit}
\DeclareMathOperator{\trace}{trace}
\DeclareMathOperator{\Rep}{Rep}
\DeclareMathOperator{\rank}{rank}
\DeclareMathOperator{\tr}{tr}
\DeclareMathOperator{\Hom}{Hom}
\DeclareMathOperator{\out}{out}
\DeclareMathOperator{\inw}{in} 
\DeclareMathOperator{\Ad}{Ad}

\newcommand{\linearisation}{\left. \frac{d}{dt} \right|_{t=0} }

\begin{document}

\title{Morse theory of the moment map for representations of quivers} 

\author{Megumi Harada}
\address{Department of Mathematics and Statistics, McMaster University,
Hamilton, Ontario L8S 4K1  Canada }
\email{Megumi.Harada@math.mcmaster.ca}

\author{Graeme Wilkin}
\address{Department of Mathematics, University of Colorado, Boulder, 
CO 80309-0395, U.S.A.} 
\email{graeme.wilkin@colorado.edu}

\date{\today}

% 53C26 = hyperkahler
% 53D20 = moment maps and symplectic reduction

%%%%%%%%%%%%%%%%%%%%%
%  Abstract
%%%%%%%%%%%%%%%%%%%%%

\begin{abstract}
The results of this paper concern the Morse theory of the norm-square of the moment map on the space of representations of a quiver. We show that the gradient flow of this function converges, and that the Morse stratification induced by the gradient flow co-incides with the Harder-Narasimhan stratification from algebraic geometry. Moreover, the limit of the gradient flow is isomorphic to the graded object of the Harder-Narasimhan-Jordan-H\"older filtration associated to the initial conditions for the flow. With a view towards applications to Nakajima quiver varieties we construct explicit local co-ordinates around the Morse strata and (under a technical hypothesis on the stability parameter) describe the negative normal space to the critical sets. Finally, we observe that the usual Kirwan surjectivity theorems in rational cohomology and integral K-theory carry over to this non-compact setting, and that these theorems generalize to certain equivariant contexts.
\end{abstract}

\maketitle

\setcounter{tocdepth}{1}
\tableofcontents

\section{Introduction}\label{sec:intro}

The motivation of this manuscript is to develop the analytic background for an equivariant Morse-theoretic approach to the study of the topology of Nakajima quiver varieties. 
We first recall the following long-standing question. Consider a hyperhamiltonian\footnote{The action of a Lie group $G$ on a hyperk\"ahler manifold $(M, \omega_{I},
  \omega_{J}, \omega_{K})$ is \emph{hyperhamiltonian} if it is hamiltonian with respect to each
  of the symplectic structures $\omega_{I}, \omega_{J}, \omega_{K}$.} action of a Lie group $G$ on a hyperk\"ahler manifold $M$
with moment maps $\Phi_{I}, \Phi_{J}, \Phi_{K}$ from $M$ to $\g^{*}$
with respect to $\omega_I, \omega_J, \omega_K$ respectively, and consider the $G$-equivariant inclusion 
\begin{equation}\label{eq:intro-hKinclusion}
(\Phi_{I}, \Phi_{J}, \Phi_{K})^{-1}(\alpha) \into M ,
\end{equation}
where $\alpha$ is a central value in $(\g^{*})^{3}$. The question is: When is the induced ring homomorphism from the $G$-equivariant cohomology\footnote{We assume throughout the paper
  that the coefficient ring is $\Q$.} $H^{*}_{G}(M)$ of the original
hyperhamiltonian $G$-space $M$ to the ordinary cohomology $H^{*}(M
\mmod_{\alpha} G)$ of the hyperk\"ahler quotient \(M \mmod_{\alpha} G := (\Phi_{I},
\Phi_{J}, \Phi_{K})^{-1}(\alpha)/G\) a {\em surjection} of rings?

In the setting of symplectic quotients, the analogous result is the well-known
Kirwan surjectivity theorem \cite{Kir84}. This theorem was
the impetus for much subsequent work in equivariant symplectic
geometry to compute topological invariants of symplectic quotients, since it
reduces the problem of computing $H^{*}(M \mod G)$ to that of
computing an equivariant cohomology ring $H^{*}_{G}(M)$ and the kernel
of a surjective ring homomorphism $\kappa$ (see e.g. \cite{JK95, Mar00, TW03, Gol02}). Thus it is
quite natural to ask whether a similar theory can be developed for hyperk\"ahler quotients.

The proof of the above Kirwan surjectivity
theorem in \cite{Kir84} uses the Morse theory of the norm-square
of the moment map.  This method has recently been extended in \cite{dwww} to prove Kirwan surjectivity for an infinite-dimensional hyperk\"ahler quotient: the moduli space of semistable rank $2$ Higgs bundles. Hence, one of the goals of the present manuscript is to develop tools
toward the further development of the Morse theory methods of \cite{dwww}, which would then aid the study of hyperk\"ahler quotients such as Nakajima quiver varieties. 
At present, we restrict our attention to the setting of moduli spaces of
representations of quivers and Nakajima quiver varieties for two
simple reasons. Firstly, the theory for this case
has much in common with that of the moduli spaces of Higgs bundles and
the moduli spaces of vector bundles which arise in gauge theory. This
includes the presence of algebro-geometric tools such
as a Harder-Narasimhan stratification on the space of representations
of the quiver which is compatible with the Morse theory of the moment
map.  Secondly, the study of quiver varieties is an 
active subject with many connections to other areas of mathematics,
including geometric representation theory, gauge
theory, and mirror symmetry (see e.g. \cite{Nakajima01, CraBoe06, HauRod08} and
references therein).

Although our results are motivated by, and
analogues of, similar results from gauge theory, it is not
straightforward to translate these to the setting of quivers.
In this paper we have developed new
ideas and constructions in order to carry through the analogous
program, in particular we are
able to deal with the general setting of a moment map associated to any
representation space of any quiver, making no assumptions
on the underlying graph or dimension vector.

A more detailed description of the main results of this paper is as follows. Let \(Q = ({\mathcal I}, E)\) be a {\em quiver}, i.e. an oriented graph with vertices ${\mathcal I}$ and edges $E$. We always assume $Q$ is finite. For an edge (also called an ``arrow'') \(a \in E\), let \(\inw(a), \out(a) \in {\mathcal I}\) denote the head and tail of the arrow respectively. Also let ${\mathbf v} = (v_{\ell})_{\ell \in {\mathcal I}} \in \Z_{\geq 0}^{{\mathcal I}}$ be a \emph{dimension vector}. From this data we may build the affine {\em space of representations of the quiver}, 
\[
\Rep(Q, {\mathbf v}) = \bigoplus_{a \in E} \Hom(V_{\out(a)}, V_{\inw(a)}) 
\]
where each $V_{\ell}$ for \(\ell \in {\mathcal I}\) is a hermitian
inner product space, with dimension $\dim(V_{\ell}) = v_{\ell}.$ Then
$\Rep(Q,{\bf v})$ is naturally equipped with the componentwise
conjugation action of the compact Lie group \(G := \prod_{\ell \in
  {\mathcal I}} U(V_{\ell}),\) where $U(V_{\ell})$ is the unitary
group associated to $V_{\ell}$ for each vertex \(\ell \in {\mathcal
  I},\) and the {\em moduli space of representations of the quiver
  $Q$} (at a central parameter $\alpha$) is the K\"ahler (or GIT)
quotient $\Rep(Q, {\bf v}) \mod_{\alpha} G$.  Similarly, the cotangent
bundle \(T^{*}\Rep(Q, {\mathbf v})\) is naturally a quaternionic
affine space and is hyperhamiltonian with respect to the induced
action of $G$. The hyperk\"ahler quotient $T^{*}\Rep(Q, {\bf v})
\mmod_{(\alpha,0)} G$ of $T^{*}\Rep(Q, {\bf v})$ by $G$ at
a central value $(\alpha,0) \in \g^{*} \oplus \g_{\C}^{*}$ is then
called a {\em Nakajima quiver variety}. Our results concerning spaces
of representations of quivers also may be applied to their cotangent
bundles, as we explain in Section~\ref{sec:applications}, and hence
may be used in the study of Nakajima quiver varieties.

In the gauge-theoretic study of the moduli of holomorphic bundles,
there are notions of stability and semi-stability coming
from geometric invariant theory, from which one obtains the Harder-Narasimhan stratification of the space of holomorphic
structures on a vector bundle. In the setting of quivers, Reineke
\cite{Rei03} showed that analogous notions also exist for the space
$\Rep(Q,{\bf v})$, and in particular
there exists an analogous Harder-Narasimhan stratification of $\Rep(Q, {\bf
  v})$ (associated to a choice of stability parameter $\alpha$). On
the other hand, just as in Kirwan's original work, we may also
consider the Morse-type stratification obtained from the
negative gradient flow of the norm-square of
the moment map $\|\Phi - \alpha\|^{2}$ on $\Rep(Q, {\bf v})$.
That such a stratification even makes sense is non-trivial, since
$\Rep(Q, {\bf v})$ is non-compact and $\|\Phi - \alpha\|^2$ not
proper, so it is not immediate that the gradient flow even converges.
Our first main theorem (Theorem~\ref{thm:gradflowconvergence}) 
is that the flow does converge, and hence the Morse
stratification exists. In fact, this result is valid for any linear
action on a hermitian vector space, not just this specific case of
quivers.

\begin{theorem}\label{thm:intro-convergence}
Let $V$ be a hermitian vector space and suppose that a compact connected Lie group $G$ acts linearly on $V$ via an injective homomorphism \(G \to U(V).\) Let \(\Phi: V \to \g^{*} \cong \g\) denote a moment map for this action. For \(\alpha \in \g,\) define \(f := \|\Phi - \alpha\|^{2}: V \to \R\) and denote by $\gamma(x,t)$ the negative gradient flow on $V$ with respect to $f$. Then for any initial condition \(x_{0} \in V,\) the gradient flow \(\gamma(x_{0}, t): \R_{\geq 0} \to V\) exists for all \(t \in \R_{\geq 0}\), and converges to a critical point $x_{\infty}$ of $f$. 
\end{theorem}

In the case of $\Rep(Q,{\bf v})$, we also obtain explicit
descriptions of the connected components of the critical sets of
$\|\Phi - \alpha\|^2$ in terms
of smaller-rank quiver varieties, 
thus providing us with an
avenue to inductively compute cohomology rings or
Poincar\'e polynomials (Proposition \ref{prop:cohomology-stratum}).

Our next main result (Theorem~\ref{theorem:equivstratification}) is
that the Morse stratification on $\Rep(Q, {\bf v})$ obtained by
Theorem~\ref{thm:intro-convergence} and the Harder-Narasimhan
stratification mentioned above are in fact equivalent (when specified
with respect to the same central parameter $\alpha$). Thus there is a
tight relationship between the algebraic geometry and the Morse
theory; in particular, in any given situation, we may use whichever
viewpoint is more convenient. 

\begin{theorem} 
Let \(Q = ({\mathcal I}, E)\) be a quiver and \({\mathbf v} \in
\Z^{\mathcal I}_{\geq 0}\) a dimension vector. Let $\Rep(Q,
{\mathbf v})$ be the associated representation space and \(\Phi:
\Rep(Q, {\mathbf v}) \to \g^* \cong \g \cong \prod_{\ell \in {\mathcal
I}} \u(V_\ell)\) be the moment map for the standard Hamiltonian action
of \(G = \prod_{\ell \in {\mathcal I}} U(V_\ell)\) on $\Rep(Q, {\mathbf
v})$ and \(\alpha \in (\alpha_\ell)_{\ell \in {\mathcal I}} \in
(i\R)^{\mathcal I}.\)  Then the algebraic stratification of $\Rep(Q, {\bf v})$ by
Harder-Narasimhan type with respect to $\alpha$ 
coincides with the
analytic stratification of $\Rep(Q, {\bf v})$ by the negative 
gradient flow of $f = \| \Phi - \alpha \|^2$. 
\end{theorem}

Our third main theorem (Theorem~\ref{thm:convergencetogradedobject})
is an algebraic description of the limit of the negative gradient flow
with respect to $\|\Phi - \alpha\|^{2}$. There is a refinement of the
Harder-Narasimhan filtration of a representation of a quiver called the {\em
  Harder-Narasimhan-Jordan-H\"older (H-N-J-H) filtration}. We show that the limit of the flow is
isomorphic to the associated graded object of this H-N-J-H filtration
of the initial condition. This is a quiver analogue of results for
holomorphic bundles (see \cite{Das92} and \cite{DasWen04}) and Higgs bundles (see
\cite{Wil06}).

\begin{theorem}\label{theorem:intro-graded}
Let \(Q = ({\mathcal I}, E), {\bf v} \in \Z^{\mathcal I}_{\geq 0},
\Rep(Q, {\bf v}),\) and $\alpha$ be as above. 
Let $A_0 \in \Rep(Q, {\bf v})$, and let $A_\infty$ be the limit of $A_0$ under the negative gradient flow of $\| \Phi - \alpha \|^2$. Then $A_\infty$ is isomorphic to the graded object of the H-N-J-H filtration of $A_0$.
\end{theorem}

Our last main result (stated precisely in
Proposition~\ref{prop:local-coord-on-stratum}) is an explicit
construction of local coordinates
near any representation $A$ of a given Harder-Narasimhan type ${\bf v^*}$. This provides a useful tool for local computations in neighborhoods of strata, for example in standard Morse-theoretic arguments that build a manifold by inductively gluing Morse strata. 
Using these local coordinates, we conclude in
Proposition~\ref{prop:kahler-normal-bundle} that the Harder-Narasimhan
strata of $\Rep(Q, {\bf v})$ have well-behaved tubular neighborhoods, which may also be
identified with the disk bundle of their normal bundle in $\Rep(Q,
{\bf v})$; it is also straightforward to compute their codimensions
(Proposition~\ref{prop:codimension}).

Finally, in the last section, we present some applications of
our results. As mentioned at the beginning of the introduction, one of
our main motivations is in the development of equivariant
Morse-theoretic tools to study Nakajima quiver varieties. Accordingly,
in Section~\ref{subsec:hK} we first review the construction of these
varieties, which involves both a real-valued moment map
$\Phi_\R$ and a holomorphic (complex-valued) moment map $\PhiC$. Our
approach is to study the space $\PhiCzero$ (which is typically singular) in terms
of the Morse theory of $\|\Phi_\R - \alpha\|^2$, and in
Section~\ref{subsec:hK} we explain how the results in this
manuscript may be applied to $\PhiCzero$. We also prove that although
the level set $\PhiC^{-1}(0)$ may be singular, the level set near the
critical sets of $\|\PhiR-\alpha\|^2$ can be described locally in
terms of the linearized data $\ker d\PhiC$
(Theorem~\ref{theorem:PhiC-level-linear}). In
Section~\ref{subsec:Kirwan}, we prove a Kirwan surjectivity result in
both rational cohomology and integral topological $K$-theory for
K\"ahler quotients of affine space by a linear group action, and in
particular deduce results for moduli spaces of representations of
quivers. In Section \ref{sec:reineke} we compare our Morse-theoretic formulae for the Poincar\'e polynomials of moduli spaces of representations of a quiver with those of Reineke in \cite{Rei03}, and in Section~\ref{subsec:eqvt-Morse}, we observe that
the Morse theory developed in this manuscript immediately generalizes
to certain equivariant settings, 
which yields as immediate corollaries equivariant Kirwan surjectivity
theorems in both rational cohomology and integral topological
$K$-theory for moduli spaces of representations of quivers; here the
equivariance is with respect to any closed subgroup of $U(\Rep(Q, {\bf
  v}))$ that commutes with the group $G = \prod_{\ell} U(V_{\ell})$.

\medskip
{\bf Acknowledgments.}  
We thank Georgios Daskalopoulos, Lisa Jeffrey, Jonathan Weitsman, and
Richard Wentworth for their encouragement, and Tam\'as Hausel and Hiroshi Konno for their interest in our work. We are grateful to Reyer Sjamaar for pointing out to us his work in \cite{Sja98}.
Both authors thank the hospitality of the Banff Research Centre, which
hosted a Research-in-Teams workshop on our behalf, as well as the
American Institute of Mathematics, where part of this work was
conducted. The first author was supported in part by an NSERC
Discovery Grant, an NSERC University Faculty Award, and an Ontario
Ministry of Research and Innovation Early Researcher Award.

\section{Preliminaries}\label{sec:preliminaries}

In this section, we set up the notation and sign conventions to be
used throughout. We have chosen our
conventions so that our moment map formulae agree with those of
Nakajima in \cite{Nak94}.

\subsection{Lie group actions, moment maps, and
  K\"ahler quotients}\label{subsec:hk-intro}

Let $G$ be a compact connected Lie group.
Let $\g$ denote its Lie algebra and $\g^{*}$ its dual, and let
$\langle \cdot, \cdot \rangle$ denote a fixed $G$-invariant inner product on $\g$. We will always identify $\g^{*}$ with $\g$ using this inner product; by abuse of notation, we also denote by $\langle \cdot, \cdot \rangle$ the natural pairing between $\g$ and its dual.  Suppose that $G$ acts on the left on a manifold $M$. Then
we define the {\em infinitesmal action} 
\(\rho: M \times \g \to TM\) by 
\begin{equation}\label{def:rho}
\rho: (x,u) \in M \times \g \mapsto \ddt (\exp tu) \cdot x \in T_xM, 
\end{equation}
where $\{\exp tu\}$ denotes the $1$-parameter subgroup of $G$
corresponding to \(u \in \g\) and \(g \cdot x\) the group
action. We also denote by $\rho(u)$ the vector field on $M$
generated by $u$, specified by \(\rho(u)(x) := \rho(x,u).\) 
Similarly, for a fixed \(x \in M,\) we denote by $\rho_{x}: \g \to
T_{x}M$ the restriction of the infinitesmal action to the point $x$,
i.e. \(\rho_{x}(u) := \rho(x,u)\) for \(u \in \g.\) If $M$ is a
Riemannian manifold, we use $\rho^{*}$ to denote the adjoint
\(\rho_{x}^{*}: T_{x}M \to \g\) with respect to the Riemannian
metric on $M$ and the fixed inner product on $\g$.

In the special case in which \(M = V\) is a vector space, there is a natural identification between the tangent bundle $TM$ and $V \times V$. In this situation, we denote by \(\delta \rho(u)\) the restriction to $T_{x}V \cong V$ of the derivative of \(\pi_{2} \circ \rho(u): V \to V\) where \(\rho(u): V \to TV\) is as above and \(\pi_{2}: TV \cong V \times V \to V\) is the projection to the second factor; the basepoint \(x \in V\) of \(\delta \rho(u): T_{x}V \cong V \to V\) is understood by context. Similarly let \(\delta \rho(X): \g \to V\) denote the linear map defined by \(\delta \rho(X)(v) := \delta \rho(v)(X).\)

Now let $(M,\omega)$ be a symplectic manifold and suppose $G$ acts
preserving $\omega$. Recall that \(\Phi: M \to \g^*\) is a \emph{moment
map}  for this $G$-action if $\Phi$ is $G$-equivariant with respect to
the given $G$-action on $M$ and the coadjoint action on $\g^*$, and in addition, 
for all \(x \in M, X \in T_xM, u \in \g,\) 
\begin{equation}\label{eq:hamilton}
\langle d\Phi_x(X), u \rangle = - \omega_x(\rho_x(u), X),
\end{equation}
where we have identified the tangent space at \(\Phi(x) \in \g^*\) with
$\g^*$. By the identification \(\g^{*} \cong \g\) using the $G$-invariant inner product, we may also consider $\Phi$ to be a $\g$-valued map; by abuse of notation we also denote by $\Phi$ the $G$-equivariant composition \(\Phi: M \to \g^{*} \stackrel{\cong}{\to} \g.\)  
In particular, by 
differentiating the condition that $\Phi$ is $G$-equivariant, we obtain the relation
\begin{align*}
\linearisation \exp(tu) \Phi(x) \exp(-tu) & = \linearisation \Phi \left(
  \exp(tu) \cdot x \right) \\
\iff \, \,  [u, \Phi(x)] & = d\Phi_x \left( \rho_x(u) \right), 
\end{align*}
for all \(u \in \g, x \in M.\) Therefore
\begin{equation}\label{eq:mom-map-Lie}
\langle [u, \Phi(x)],  v  \rangle_{\g} = - \omega_x(\rho_x(v), \rho_x(u)), 
\end{equation}
for all \(u, v \in \g, x \in M.\) 

If $M$ is additionally K\"ahler, the relationships
between the complex structure $I$, the metric $g$, and the symplectic
form $\omega$ on $M$ are
\begin{equation}\label{eq:Kahler-relations}
\omega_x(X,Y) = g_x(IX, Y) = - g_x(X,IY)
\end{equation}
for all \(x \in M, X, Y \in T_xM.\) 
The above equations imply that the metric $g$ is $I$-invariant. We say a $G$-action on a K\"ahler manifold is \emph{Hamiltonian} if it preserves the K\"ahler structure and is Hamiltonian with respect to $\omega$.

\subsection{Moduli spaces of representations of quivers and quiver varieties}\label{subsec:intro-quivers}

In this section, we recall the construction of the
moduli spaces of representations of quivers. These spaces are constructed from the
combinatorial data of a finite oriented graph $Q$ (a quiver) and a
dimension vector ${\bf v}$ (which specifies the underlying vector
space of the representation), and are K\"ahler quotients of
the affine space of representations $\Rep(Q, {\bf v})$.

We refer the reader to \cite{Nak94} for details on what
follows. Let \(Q = ({\mathcal I}, E)\) be a finite oriented graph with vertices \({\mathcal I}\) and oriented edges \(a \in E,\) where we denote by \(\out(a) \in {\mathcal I}\) the outgoing vertex of the arrow $a$, and by $\inw(a)$ the incoming vertex. Also choose a finite-dimensional hermitian vector space \(V_{\ell}\) for each vertex \(\ell \in {\mathcal I},\) with \(\dim_{\C}(V_{\ell}) = v_{\ell}.\) Assembling this data 
gives us the {\em dimension vector} \({\bf v}  = (v_\ell)_{\ell \in {\mathcal I}} \in (\Z_{\geq 0})^{\mathcal I}.\) 
The {\em space of representations of the quiver $Q$ with dimension vector ${\bf v}$} is 
\begin{equation*}
\Rep(Q,{\mathbf v}) := \bigoplus_{a \in E} \Hom(V_{\out(a)}, V_{\inw(a)}).
\end{equation*}
Here $\Hom(-,-)$ denotes the hermitian vector space of $\C$-linear homomorphisms. 
We also denote by 
\begin{equation}\label{eq:def-VectQv}
\Vect(Q, {\bf v}) := \bigoplus_{\ell \in {\mathcal I}} V_{\ell}
\end{equation}
the underlying vector space of the representation, and let \(\rank(Q,
{\bf v}) := \dim(\Vect(Q, {\bf v}))\) denote its total dimension. 

The notion of a subrepresentation is also straightforward. Let \(A
\in \Rep(Q, {\bf v}), A' \in \Rep(Q, {\bf v'})\) be representations
with corresponding $\{ V_{\ell}\}_{\ell \in {\mathcal I}},
\{V'_{\ell}\}_{\ell \in {\mathcal I}}$ respectively. We say $A'$ is a
{\em subrepresentation} of $A$, denoted \(A' \subseteq A,\) if
\(V'_{\ell} \subseteq V_{\ell}\) for all \(\ell \in {\mathcal I},\)
the $V'_{\ell}$ are invariant under $A$, i.e. \(A_{a}(V'_{\out(a)})
\subseteq V'_{\inw(a)}\) for all \(a \in E,\) and $A'$ is the
restriction of $A$, i.e. \(A'_{a} = A_{a} \vert_{V'_{\out(a)}}\) for
all \(a \in E\). For each \(\ell \in {\mathcal I},\) let $U(V_{\ell})$ denote the unitary group associated to $V_{\ell}$. 
The group \(G = \prod_{\ell \in {\mathcal I}} U(V_{\ell})\) acts on
$\Rep(Q,{\bf v})$ by conjugation, i.e. 
\[
(g_{\ell})_{\ell \in {\mathcal I}} \cdot (A_a)_{a \in E} = 
\left( g_{\inw(a)}A_ag^*_{\out(a)}\right)_{a \in E}. 
\]
Hence the infinitesmal action of an element \(u = (u_{\ell})_{\ell \in {\mathcal I}} \in
\prod_{\ell \in {\mathcal I}} \u(V_{\ell})\) is given by
\[
\rho(A, u) = \left( u_{\inw(a)}A_a - A_a u_{\out(a)}\right)_{a \in E}. 
\]
Moreover, the K\"ahler form $\omega$ on $\Rep(Q, {\bf v})$ is given as follows: given two tangent vectors $\delta A_{1} = \left( (\delta A_{1})_{a}\right)_{a \in E}, \delta A_{2} = \left( (\delta A_{2})_{a}\right)_{a \in E}$  at a point in $\Rep(Q, {\bf v})$, 
\begin{equation}\label{eq:Kahler-form-onRepQv}
\omega(\delta A_{1}, \delta A_{2}) = \sum_{a \in E} \Im(\tr(\delta A_{1})_{a}^{*} (\delta A_{2})_{a}),
\end{equation}
where $\Im$ denotes the imaginary part of an element in $\C$.

We now 
explicitly compute the moment map $\Phi$
for the $G$-action on \(\Rep(Q,{\bf v}).\) 
Denote by $\Phi_\ell$ the $\ell$-th component, and  
identify \(\u(V_{\ell}) \cong \u(V_{\ell})^*\) using
the invariant pairing 
\[
\langle u, v \rangle := \tr(u^* v).
\]
With these conventions, the 
natural action of $U(V)$ on $\Hom(V',V)$ and $\Hom(V,V')$ 
given by, for 
\(A \in \Hom(V',V), B \in \Hom(V,V'),\) 
\[
g \cdot A = gA, \quad g \cdot B = Bg^*
\]
(where the right hand side of the equations is ordinary matrix
multiplication), has moment maps 
\[
\Phi(A) =  \frac{i}{2} AA^* \in \u(V), \quad 
\Phi(B) = - \frac{i}{2} B^*B \in \u(V)
\]
respectively. We conclude that for \(A = 
(A_a)_{a \in E} \in
\Rep(Q,{\bf v}),\) 
\[
\Phi_{\ell}(A) = \frac{1}{2}i \left(  \sum_{a: \inw(a)=\ell} 
A_a A_a^* + 
\sum_{a': \out(a')=\ell}  - A_{a'}^* A_{a'}  \right).
\]
Hence we have  
\begin{equation}\label{eq:formula-Kahler-mom-map-RepQv}
\Phi(A) = \frac{1}{2} i \sum_{a \in E} [A_{a}, A_{a}^{*}]
\end{equation}
where $A_{a}A_{a}^{*}$ is understood to be valued in $\u(V_{\inw(a)})$ and $A_{a}^{*}A_{a}$ in $\u(V_{\out(a)})$. Henceforth we often simplify the notation further and write~\eqref{eq:formula-Kahler-mom-map-RepQv} as
\begin{equation}\label{eq:simplified-Kahler-mom-map-eq}
\Phi(A) = \frac{1}{2} i [ A, A^{*}],
\end{equation}
where the summation over the arrows \(a \in E\) is understood.

\begin{remark}\label{remark:doublevertices} 
There is a central \(S^1
\subseteq G = \prod_{\ell \in {\mathcal I}} U(V_{\ell})\), given by
the diagonal embedding into the scalar matrices in each $U(V_{\ell})$,
which acts trivially on $\Rep(Q,{\bf v})$. For the purposes of
taking quotients, it is sometimes more convenient to consider the action of
the quotient group \(PG := G/S^1,\) which then acts
effectively on $\Rep(Q, {\mathbf v})$. From this point of view, the moment map naturally takes
values in the subspace \((\pg)^* \subseteq \g^*,\) where the inclusion is induced by the quotient
\(\g \to \pg := \g/\Lie(S^1),\) so we may consider \(\Phi_{\R}\)
as a function from $M$ to \((\pg)^*.\) 
On the other hand, the moment map (being a commutator) 
takes values in the subspace of traceless matrices in \(\g \subseteq \u(\Vect(Q, {\mathbf v})),\) 
the orthogonal complement of $\Lie(S^{1})$ in $\g$, which may be identified 
with $\pg^{*}$. For our Morse-theoretic purposes, 
there is no substantial distinction between $G$ and $PG$, although the difference does become relevant when computing equivariant cohomology. 
\end{remark} 

We now construct the relevant K\"ahler quotients. 
Let \[\alpha = (\alpha_{\ell})_{\ell \in {\mathcal I}} \in (i\R)^{\mathcal I}.\] 
This uniquely specifies an element in the center $Z(\g)$ of $\g$, namely 
\[
(\alpha_{\ell} \id_{V_\ell} )_{\ell \in {\mathcal I}} \in Z(\g),
\] 
which by
abuse of notation we also denote by \(\alpha.\) We will often refer to
such an $\alpha$ as a {\em central parameter}. We always assume that \(\tr(\alpha) = 0,\) i.e. \(\alpha \in \pg^{*}\) (see Remark~\ref{remark:doublevertices}). Then the Hamiltonian quotient is 
\begin{equation}\label{eq:def-moduli-space-of-rep}
X_{\alpha}(Q, {\mathbf v}) := \Phi^{-1}(\alpha)/G .
\end{equation}  
We also call $X_\alpha(Q, {\bf v})$ the {\em representation
  variety} of the quiver $Q$ corresponding to $\alpha$. See \cite{Kin94} and Section \ref{sec:HN} of this paper for more details about the relationship between $X_\alpha(Q, {\bf v})$ and the GIT quotient of $\Rep(Q, {\bf v})$.

\begin{remark}\label{remark:framed-vs-unframed} 

The construction above also applies to the framed representation varieties studied by Nakajima in \cite{Nak94}. The data for these consists of two dimension vectors \({\mathbf v},
  {\mathbf w} \in \Z_{\geq 0}^{\mathcal I}\) specifying hermitian
  vector spaces \(V_\ell, W_\ell\) with \(\dim_{\C}(V_\ell) = v_\ell,
  \dim_{\C}(W_\ell) = w_\ell\) for \(\ell \in {\mathcal I}.\) We
  may define \[ \Rep(Q,{\mathbf v}, {\mathbf w}) := \left(
    \bigoplus_{a \in E} \Hom(V_{\out(a)}, V_{\inw(a)}) \right) \oplus
  \left( \bigoplus_{\ell \in {\mathcal I}} \Hom(V_\ell, W_\ell)
  \right) \] and the associated K\"ahler quotient \(X_\alpha(Q, {\bf
    v}, {\bf w}) := \Phi^{-1}(\alpha)/G\) for the analogous moment map
  $\Phi$ on $\Rep(Q, {\bf v}, {\bf w})$. Hence \(\Rep(Q, {\mathbf
    v})\) and $X_\alpha(Q, {\bf v})$ correspond to the special case when
  \({\mathbf w} = 0\).  In the literature, the case \({\mathbf w} = 0\)
  is often called the ``unframed'' case, while \({\mathbf w} \neq 0\)
  is the ``framed'' case. In \cite[p. 261]{CraBoe01}, Crawley-Boevey points out that a framed representation variety
  $X_{\alpha}(Q, {\mathbf v}, {\mathbf w})$ can also be realized as an
  unframed representation variety for a different quiver, and so for the purposes of this paper it is sufficient to restrict attention to the unframed case.

\end{remark}

\section{Morse theory with $\|\Phi - \alpha\|^2$} \label{sec:gradientflow}

Let $V$ be a hermitian vector space and suppose a compact connected
Lie group $G$ acts linearly on $V$ via an injective homomorphism \(G
\to U(V).\) Let $\Phi$ denote the corresponding moment map. For
$\alpha \in \mathfrak{g}$, let $f := \|\Phi - \alpha\|^2$. Denote
by $\gamma(x,t)$ the negative gradient flow on $V$ with respect to
$\grad(f)$, i.e. $\gamma(x,t)$ satisfies
\begin{equation}\label{eq:gradient-flow-def}
\gamma(x,0) = x, \quad \gamma'(x,t) = - \grad(f)_{\gamma(x,t)}.
\end{equation}
In this section we prove that the gradient flow of $f$ exists for all
time $t$ and converges to a critical point of $f$. This
general situation of a linear action on a vector space contains the
main case of interest in this paper, namely, when $V$ is the space
$\Rep(Q, {\bf v})$.

\begin{theorem}\label{thm:gradflowconvergence}
Let $V$ be a hermitian vector space and suppose that a compact connected Lie group $G$ acts linearly on $V$ via an injective homomorphism \(G \to U(V).\) Let \(\Phi: V \to \g^{*} \cong \g\) denote a moment map for this action. For \(\alpha \in \g,\) define \(f := \|\Phi - \alpha\|^{2}: V \to \R\) and denote by $\gamma(x,t)$ the negative gradient flow on $V$ with respect to $f$. Then for any initial condition \(x_{0} \in V,\) the gradient flow \(\gamma(x_{0}, t): \R \to V\) exists for all time \(t \in \R\) and converges to a critical point $x_{\infty}$ of $f$. 
\end{theorem}

\begin{remark} 
Although in the construction of $X_{\alpha}(Q,{\bf v})$ we always restrict to {\em
central} parameters \(\alpha \in Z(\g),\) this hypothesis
is unecessary for Theorem~\ref{thm:gradflowconvergence}. We restrict
to $\alpha \in Z(\g)$ to ensure that the function $f := \|\Phi -
\alpha\|^2$ is $G$-invariant and hence the symplectic
quotient $f^{-1}(0)/G$, the study of which is the main motivation of
this paper, makes sense. 
\end{remark}

In the special case when $V = \Rep(Q, {\bf v})$, we also explicitly
describe in Proposition~\ref{prop:momentmapdetermined} the critical
sets of $f$ in terms of lower-rank representation spaces, and we
define a Morse-theoretic stratification of $\Rep(Q, {\mathbf v})$.

\subsection{Convergence of the gradient flow}

First recall that the gradient of the norm-square of any moment map
$\Phi$ with respect to a K\"ahler symplectic structure $\omega$,
compatible metric $g$, and complex structure $I$, is given as
follows. For any \(x \in V, X \in T_x V,\)
\begin{align*}
df_x(X) & = 2 \langle (d\Phi)_{x}(X),  \Phi(x) - \alpha \rangle \\
& = -2 (\omega)_{x}(\rho(x,  \Phi(x) - \alpha), X) \\
& = -2 g_{x}(I \rho(x,  \Phi(x) - \alpha), X) \\
& = g_{x}(-2 I \rho(x, \Phi(x) - \alpha), X).
\end{align*}
Hence the negative gradient vector field for $f = \|\Phi - \alpha\|^2$ is given by 
\begin{equation}\label{eq:negative-gradient}
- \grad(f)(x) = 2I\rho(x, (\Phi(x) - \alpha)).
\end{equation}
The proof of Theorem \ref{thm:gradflowconvergence} is contained in
Lemmas~\ref{lemma:exists-alltime} to~\ref{lemma:lojasiewicz} below. The first step is to show that the negative gradient flow is defined for all time \(t
\in \R.\) Note that since \(G \subseteq U(V)\) acts
complex-linearly on $V$, its complexification $G_{\C} \subseteq GL(V)$
also acts naturally on $V$.

\begin{lemma}\label{lemma:exists-alltime}
Let $V, G, \Phi, \alpha$ and $\gamma(x,t)$ be as in the statement of
Theorem~\ref{thm:gradflowconvergence}. 
Then for any initial condition \(x_{0} \in V,
\gamma(x_{0}, t)\) exists for all \(t \geq 0.\) Moreover, $\gamma(x_0, t)$ is contained in a compact subset of $V$.
\end{lemma}

\begin{proof}
Local existence for ODEs shows that for any \(x_0 \in V\), the gradient
flow $\gamma(x_0, t)$ exists on $(-\varepsilon_{x_0},
\varepsilon_{x_0})$ for some $\varepsilon_{x_{0}} > 0$ that depends
continuously on $x_0$ (see for example Lemmas 1.6.1 and 1.6.2 in
\cite{Jos05}). By construction, the function $f$ is decreasing along the negative gradient flow, therefore $\| \Phi(x) - \alpha \|$ is bounded along the flow
and there exists a constant
$C$ such that $\| \Phi(\gamma(x_0,t)) \| \leq C$. Lemma 4.10 of \cite{Sja98} shows that the set of points 
$$
K_C := \left\{
  x \in G_\C \cdot x_0 \subset V \, : \, \| \Phi(x) \| \leq C
  \right\}$$ 
  is a bounded subset of $V$, so its closure is compact. Equation
  \eqref{eq:negative-gradient} shows that the finite-time gradient
  flow is contained in a $G_\C$-orbit (see for example Section 4 of
  \cite{Kir84}) so we conclude that $\gamma(x_0, t) \in K_C$ for any
  value of $t$ for which $\gamma(x_{0},t)$ is defined. Since the
  closure of $K_C$ is
  compact and $\varepsilon_{x} > 0$ is a continuous function on $V$,
  there exists $\varepsilon > 0$ such that $\varepsilon_{x} \geq
  \varepsilon > 0$ on $\overline{K_C}$. Therefore we can iteratively extend the
  flow so that it exists for all positive time (see for example the proof of
  \cite[Theorem 1.6.2]{Jos05}).
\end{proof}

Next we show that the flow converges along a subsequence to a critical point of $f$. 

\begin{lemma}\label{lem:convergencealongsubsequence}
  Let $V, G, \Phi, \alpha$ and $\gamma(x,t)$ be as in the statement of
  Theorem~\ref{thm:gradflowconvergence}. 
 Then for any initial condition \(x_{0} \in V,\)
  there exists a sequence \(\{t_{n}\}_{n=0}^{\infty}\) with \(\lim_{n
    \to \infty} t_{n} = \infty\) and a critical point \(x_{\infty} \in
  V\) of $f$ such that
\[
\lim_{n \to \infty} \gamma(x_{0}, t_{n}) = x_{\infty}.
\]
\end{lemma}

\begin{proof}
Lemma~\ref{lemma:exists-alltime} shows that the negative
gradient flow \(\{\gamma(x_{0}, t)\}_{t \geq 0}\) is contained in a
compact set $\overline{K_C}$. Hence there exists a sequence $\{ t_n \}_{n=0}^{\infty}$
such that $\lim_{n \rightarrow \infty} \gamma(x_0, t_n) = x_\infty$
for some $x_\infty \in \overline{K_C}$. To see that $x_\infty$ is a critical point
of $f$, firstly note that since $f(\gamma(x_0, t_n))$ is bounded below
and nonincreasing as a function of $n$, then
\begin{equation}
\lim_{n \rightarrow \infty} \frac{df}{dt}(\gamma(x_0, t_n)) =  0.
\end{equation}
Moreover, equation \eqref{eq:negative-gradient} shows that the gradient vector field $\grad f$ is continuous on $V$, so 
\begin{multline}
\| \grad f(x_\infty) \|^2 = \lim_{n \rightarrow \infty} \| \grad f(\gamma(x_0, t_n)) \|^2 \\ = \lim_{n \rightarrow \infty} df(\grad f(\gamma(x_0,t_n))) = - \lim_{n \rightarrow \infty} \frac{df}{dt} (\gamma(x_0, t_n)) = 0.
\end{multline}
Therefore $x_\infty$ is a critical point of $f$.
\end{proof}

Even if the negative gradient flow $\gamma(x_{0}, t)$ converges along
a subsequence $\{t_{n}\}_{n=0}^{\infty}$, it is still possible that
the flow $\{\gamma(x_{0}, t)\}_{t \geq 0}$ itself does not
converge; for instance, $\gamma(x_{0}, t)$ may spiral around a
critical point $x_{\infty}$ (cf. \cite[Example 3.1]{Ler04}). The key estimate that shows that the flow does not spiral around the critical set is the following gradient inequality, which shows that the length of the gradient flow curve is finite. In the case at hand $f$ is a polynomial and so the inequality is simple to prove. More generally (when $f$ is analytic) this inequality is originally due to Lojasiewicz in \cite{Loj84}, see also \cite{Ler04} for an exposition of the case when $f$ is the norm-square of a moment map on a finite-dimensional manifold.

\begin{lemma}\label{lemma:lojasiewicz}
Let $(V, G, \Phi, \alpha)$ be as in the statement of Theorem~\ref{thm:gradflowconvergence}. Then 
for every critical point $x_\infty$ of $f$ there exists $\delta > 0$, $C > 0$, and $0 < \theta < 1$ such that
\begin{equation}\label{eqn:lojasiewicz}
\| \grad f(x) \| \geq C \left| f(x) - f(x_\infty) \right|^{1-\theta}
\end{equation} 
for any $x \in V$ with 
$\| x - x_\infty \| < \delta$.
\end{lemma}

A standard procedure (see for example \cite{Simon83}) then shows that the gradient flow converges.

\begin{lemma}\label{lem:limitexistsalongflow}
  Let $(V, G, \Phi, \alpha)$ be as in the statement of Theorem~\ref{thm:gradflowconvergence}, and let $\gamma(x, t): V \times \R \to V$ denote the negative gradient flow of \(f = \|\Phi - \alpha\|^{2}.\) Then for any initial condition \(x_{0} \in V,\) there exists a critical point \(x_{\infty} \in V\) of $f$ such that
$$\lim_{t \rightarrow \infty} \gamma(x_0, t) = x_\infty.$$
\end{lemma}

\begin{proof}[Proof of Lemma \ref{lem:limitexistsalongflow}]
The idea is
to show that once the gradient flow gets close to a critical point
then it either converges to a nearby critical point or it flows down
to a lower critical point. Lemma \ref{lem:convergencealongsubsequence} shows that the gradient flow converges to a critical point along a subsequence, and so the rest of the proof follows that given in \cite{Ler04}. The proof given in \cite{Ler04} assumes
that $\| \Phi - \alpha \|^2$ is proper;  however we do not need this
condition here, since we have the result of Lemma
\ref{lem:convergencealongsubsequence}.
\end{proof} 

As a consequence of these results we may now prove the main theorem of this section.

\begin{proof}[Proof of Theorem~\ref{thm:gradflowconvergence}]
Lemma~\ref{lemma:exists-alltime} shows that the negative gradient flow $\gamma(x_{0}, t)$ exists for all time \(t \in \R\) and for any initial condition \(x_{0} \in V.\) Lemma~\ref{lem:limitexistsalongflow} shows that $\{ \gamma(x_{0}, t) \}$ converges to a limit point $x_{\infty} \in V$. This limit point $x_{\infty}$ agrees with the limit point of the subsequence of Lemma~\ref{lem:convergencealongsubsequence}, which shows in addition that $x_{\infty}$ is a critical point of $f$. The theorem follows.  
\end{proof}

In the course of the proof of gradient flow convergence we also obtain the following estimate, which is an important part of the proof of Theorem \ref{theorem:equivstratification}.
\begin{lemma}\label{lem:close-convergence}
Let $x_\infty$ be a critical point of $f$. Then for any $\varepsilon > 0$ there exists $\delta > 0$ such that for any $x$ satisfying
\begin{enumerate}

\item $\| x - x_\infty \| < \delta$, so that the inequality \eqref{eqn:lojasiewicz} holds (with respect to the critical point $x_\infty$)

\item $f\left( \lim_{t \rightarrow \infty} \gamma(x,t) \right) = f(x_\infty)$,
\end{enumerate}
then we have
\begin{equation}\label{eqn:critical-distance-estimate}
\left\| x_\infty - \lim_{t \rightarrow \infty} \gamma(x, t) \right\| < \varepsilon .
\end{equation}
\end{lemma}

The proof is standard (cf. \cite{Simon83}), and therefore omitted.

\subsection{Critical points of $\| \Phi - \alpha \|^2$}\label{subsec:crit-pts-description}

In this section we analyze properties of the components of the
critical set $\Crit(\|\Phi - \alpha\|^2)$ in the quiver case. 
Let \(Q = ({\mathcal I}, E)\) be a quiver as in
Section~\ref{subsec:intro-quivers}, and $\Rep(Q,{\mathbf v})$ the
space of representations of $Q$ for a choice of dimension vector
\({\mathbf v}.\) Let \(A = (A_a)_{a \in E} \in \Rep(Q, {\mathbf v})\)
be a representation.  For a given $A$,  we denote by $\beta := \Phi(A) -
\alpha$ the shifted moment map value at $A$ where  
$\alpha$ is a central parameter as in Section~\ref{subsec:intro-quivers}. 
(Although $\beta$ depends on $A$, we suppress it from the notation for
simplicity.)

We first observe that 
\(A \in \Rep(Q,{\mathbf v})\) is a
critical point of \(f = \|\Phi -\alpha\|^2\)
if and only if 
\begin{equation}\label{eq:crit-pt-eqn-quiver} 
\left(  \beta_{\inw(a)}A_a -
A_a \beta_{\out(a)} \right) = 0
\end{equation}
for all \(a \in E.\) 
Secondly, since \(i \beta \in \prod_{\ell \in
{\mathcal I}} i\u(V_{\ell})\) is Hermitian, it can be diagonalized
with purely real eigenvalues. Since the action of
$\g$ on $V_{\ell}$ is by left multiplication, we have an
eigenvalue decomposition 
\[
V_{\ell} = \bigoplus_{\lambda} V_{\ell,\lambda} ,
\]
where the sum is over distinct eigenvalues of $i \beta$. Let
\(V_{\lambda}\) denote the $\lambda$-eigenspace of \(i \beta\) in
\(\Vect(Q, {\mathbf v})\),
and let ${\bf v_\lambda} = (\dim_{\C}(V_{\ell, \lambda}))_{\ell \in
  {\mathcal I}} \in \Z^{\mathcal I}_{\geq 0}$ denote the associated
dimension vector.
Thirdly, \eqref{eq:crit-pt-eqn-quiver} implies that
$A_a$ preserves $V_{\lambda}$ for all eigenvalues $\lambda$ and all edges $a \in E$. In
particular, the restrictions \(A_{\lambda} := A
|_{V_{\lambda}}\) are subrepresentations of $A$.
Hence we get a decomposition of the representation 
\begin{equation}\label{eq:A-decomp-by-beta}
A = \bigoplus_{\lambda} A_{\lambda},
\end{equation}
where again the sum is over distinct eigenvalues of $i\beta$. The following definitions of degree and slope-stability are originally due to King in \cite{Kin94}.

\begin{definition}\label{def:degree-slope}
Let \(Q = ({\mathcal I}, E)\) be a quiver,
with associated hermitian vector spaces \(\{V_{\ell}\}_{\ell \in {\mathcal
I}}\) and dimension vector \({\mathbf v} = (v_{\ell})_{\ell \in
  {\mathcal I}} \in \Z_{\geq 0}^{\mathcal I}.\) 
Let $\alpha$ be a central parameter. 
We define the {\em $\alpha$-degree} of $(Q,{\mathbf v})$ by 
\begin{equation}\label{eq:alpha-degree-def}
\deg_{\alpha}(Q, {\mathbf v}) := \sum_{\ell \in {\mathcal I}} i
\alpha_{\ell} v_{\ell}.
\end{equation}
We also define the {\em rank} of $(Q,{\mathbf v})$ as 
\begin{equation}\label{eq:rank-def} 
\rank(Q,{\mathbf v}) := \sum_{\ell \in {\mathcal I}} v_{\ell} =
\dim_\C(\Vect(Q, {\bf v})) \in \Z. 
\end{equation}
Finally, we define the {\em
$\alpha$-slope} $\mu_\alpha(Q, {\bf v})$
of $(Q,{\mathbf v})$ by 
\begin{equation}\label{eq:alpha-slope-def}
\mu_{\alpha}(Q,{\mathbf v}) := \frac{\deg_{\alpha}(Q,{\mathbf
v})}{\rank(Q,{\mathbf v})}.
\end{equation}
\end{definition}

At a critical point $A$ of \(f = \|\Phi - \alpha\|^{2},\) the
$\alpha$-slope of the subrepresentation $A_{\lambda}$ 
turns out to be related to the eigenvalues $\lambda$ of $i \beta$ on
$V_\lambda$. 

\begin{lemma}\label{lem:eigenvaluesdetermined}
  Let $Q = ({\mathcal I}, E)$ be a quiver with specified dimension
  vector \({\bf v} \in \Z^{\mathcal I}_{\geq 0},\)
  $\Rep(Q, {\mathbf v})$ its associated representation space, and
  \(\Phi: \Rep(Q, {\mathbf v}) \to \g^{*} \cong \g \cong \prod_{\ell
    \in {\mathcal I}} \u(V_{\ell})\) a moment map for the standard
  Hamiltonian action of $G = \prod_{\ell \in {\mathcal I}}
  U(V_{\ell})$ on $\Rep(Q, {\mathbf v})$. Suppose $A$ is a critical
  point of \(f = \|\Phi - \alpha\|^{2},\) and further suppose
  $\lambda$ is an eigenvalue of \(i \beta = i(\Phi(A) - \alpha),\) with
  associated subrepresentation $A_{\lambda}$. Let ${\bf v}_{\lambda}$
  be the dimension vector of $A_\lambda$. Then
\[
\lambda = - \mu_{\alpha}(Q, {\bf v}_{\lambda}).
\]
\end{lemma}

\begin{proof} 
Equation~\eqref{eq:crit-pt-eqn-quiver} implies that 
$A$ preserves the eigenspace $V_{\lambda}$ for each $\lambda$. Hence for each edge \(a \in E,\) $A_a$ decomposes as a sum \(A_a = \oplus_\lambda A_{a,\lambda}\) according
to~\eqref{eq:A-decomp-by-beta}, and in turn $i \Phi(A)$ may also be
written as a sum 
\[
i \Phi(A) = \bigoplus_{\lambda} i \Phi_{\lambda}(A_{\lambda}) =
\bigoplus_\lambda \sum_{a \in E} - \frac{1}{2}[A_{a,\lambda},
A^*_{a,\lambda}] \in \u(V_{\lambda}),
\]
where each summand in the last expression is considered as an element in
 the appropriate $\u(V_{\ell, \lambda})$ and \(\Phi_{\lambda}:=\Phi
 |_{\Rep(Q,{\mathbf v}_{\lambda})}.\) It is evident
 that \(\trace(\Phi_{\lambda}(A_{\lambda})) = 0\) for each $\lambda$
 since it is a sum of commutators. On the other hand, by definition,
 $V_{\lambda}$ is the $\lambda$-eigenspace of $i\beta =
 i(\Phi(A)-\alpha)$, so we have
\[
i (\Phi(A)-\alpha) |_{V_{\lambda}} = \lambda \id_{V_{\lambda}}.
\]
Taking the trace of the above equation, we obtain
\[
-i\trace( \alpha |_{V_{\lambda}}) = \lambda \rank(V_{\lambda}).
\]
By definition, the $\alpha$-degree of the representation $A_{\lambda}$ is \( \deg_{\alpha}(Q,
{\mathbf v}_{\lambda}) = i\trace( \alpha |_{V_{\lambda}}) \), so
\[
\lambda = - \frac{\deg_{\alpha}(Q, {\mathbf
v}_{\lambda})}{\rank(V_{\lambda})} = - \mu_{\alpha}(Q,{\mathbf
v}_{\lambda}). \qedhere
\]
\end{proof}

In fact, the converse also holds.

\begin{proposition}\label{prop:momentmapdetermined}
  Let $Q = ({\mathcal I}, E)$, ${\mathbf v} \in \Z^{\mathcal I}_{\geq 0}$, $G = \prod_{\ell \in
    {\mathcal I}} U(V_{\ell})$, and $\Phi: \Rep(Q, {\mathbf v}) \to
  \g^{*} \cong \g$ be as in Lemma~\ref{lem:eigenvaluesdetermined}.
  Then \(A \in \Rep(Q, {\mathbf v})\) is a critical point of \(f =
  \|\Phi - \alpha\|^{2}\) if and only if $A$ splits into orthogonal
  subrepresentations \(A = \oplus_{\lambda} A_{\lambda}\) as
  in~\eqref{eq:A-decomp-by-beta}, where each $A_{\lambda}$ satisfies
\[
i\left(\Phi_{\lambda}(A_{\lambda}) - \alpha |_{V_{\lambda}} \right)  = - \mu_{\alpha}(Q, {\mathbf v}_{\lambda}) \cdot \id_{V_{\lambda}}. 
\]
\end{proposition}

\begin{proof}
  The proof of Lemma \ref{lem:eigenvaluesdetermined} shows that if $A$
  is a critical point of $f = \| \Phi - \alpha \|^2$ with associated
  splitting \eqref{eq:A-decomp-by-beta}, 
then for each eigenvalue $\lambda$
  we have $ \Phi_\lambda(A_{\lambda}) = \left. \alpha
  \right|_{V_\lambda} - i \lambda \cdot \id_{V_\lambda}$. In the
  other direction, we see that if \(i\beta |_{V_{\lambda}} =
  i\Phi_{\lambda}(A_{\lambda}) - i\alpha |_{V_{\lambda}}\) is a scalar
  multiple of the identity on each $V_{\lambda}$,
  then~\eqref{eq:crit-pt-eqn-quiver} holds for all \(a \in E\)
since $(Q, A)$ splits as in \eqref{eq:A-decomp-by-beta}. Therefore the negative gradient vector field vanishes and $A$ is a critical point of $\| \Phi - \alpha \|^2$.
\end{proof}

Therefore the $\alpha$-slopes of the subrepresentations in the splitting \(A = \oplus_{\lambda}
A_{\lambda}\) of a critical representation encode crucial information
about the critical point. This leads to the following definition.

\begin{definition}\label{def:criticaltype}
Let $A \in \Rep(Q, {\bf v})$ and let 
\begin{equation}\label{eq:A-splitting}
A = \bigoplus_{s=1}^m A_s, \quad V_\ell = \bigoplus_{s=1}^m V_{\ell, s} 
\end{equation}
be a splitting of $A$ into subrepresentations. 
For each $s, 1 \leq s \leq m$, let ${\bf v}_s$ be the
associated dimension vector 
\begin{equation}
{\bf v}_s = \left(\dim_{\C}(V_{\ell, s})_{\ell \in {\mathcal I}} \right) \in
\Z_{\geq 0}^{\mathcal I}
\end{equation}
of $V_s$. The {\em critical type} associated to the
splitting~\eqref{eq:A-splitting} is defined to be the vector ${\bf v}^* = ({\bf v}_1, \ldots, {\bf v}_m)$, where the subrepresentations are ordered so that  $\mu_{\alpha}(Q, {\mathbf v}_i) \geq \mu_{\alpha}(Q, {\mathbf v}_j)$ for all $i < j$. The \emph{slope vector} associated to ${\bf v^*}$ is the vector $\nu \in \R^{\rank(Q, {\bf v})}$ given by 
$$
\nu = \left( \mu_\alpha(Q, {\bf v}_1), \ldots, \mu_\alpha(Q, {\bf v}_1), \mu_\alpha(Q, {\bf v}_2), \ldots, \mu_\alpha(Q, {\bf v}_2), \ldots, \mu_\alpha(Q, {\bf v}_m), \ldots, \mu_\alpha(Q, {\bf v}_m) \right), 
$$
where there are $\rank(Q, {\bf v}_s)$ terms equal to $\mu_\alpha(Q, {\bf v}_s)$ for each $1 \leq s \leq m$. 
\end{definition}

The following example illustrates this definition.

\begin{example}
Suppose the quiver $Q = ({\mathcal I}, E)$ is the following directed graph: 

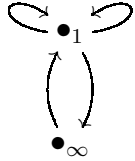
\begin{equation*}
\xygraph{
!{<0cm, 0cm>;<1.0cm, 0cm>:<0cm, 1.0cm>::}
!{(0,1.5) }*+{\bullet_1}="a"
!{(0,0) }*+{\bullet_\infty}="b"
"a" :@/^0.3cm/ "b"  "a" :@(l,lu) "a"
"b" :@/^0.3cm/ "a"  "a" :@(r,ru) "a"}
\end{equation*}
where the vertices are ${\mathcal I} = \{1, \infty\}$. Suppose the
dimension vector is ${\bf v} = (2,1) \in \Z^2$, and the central
parameter is chosen to be 
\[
\alpha = \left( \left[ \begin{array}{cc} i\lambda & 0 \\ 0 & i\lambda
    \end{array} \right], -2i\lambda \right) \in i \u(2) \times
i \u(1)
\]
for a positive real number \(\lambda > 0.\) Let \(A \in \Rep(Q, {\bf
  v})\) and suppose that $A$ admits a splitting 
\begin{equation}\label{eq:A-split} 
A = A_1 \oplus A_2
\end{equation}
into subrepresentations where the dimension vectors ${\bf v}_1, {\bf v}_2$ associated to $A_1$ and $A_2$ are 
\begin{equation}\label{eq:A-subdim}
{\bf v}_1 = (1,1), \quad {\bf v}_2 = (1,0),
\end{equation}
respectively. If this corresponds to a critical representation, then the critical type is the vector ${\bf v^*} = \left( (1,1), (1,0) \right)$. The slope of each subrepresentation is given by
\begin{equation}
\deg_\alpha(Q, {\bf v}_1) = i(i\lambda)(1) + i(-2i\lambda)(1) = -\lambda + 2 \lambda = \lambda > 0, 
\end{equation}
and since $\rank(Q, {\bf v}_1) = 2$, we conclude that $\mu_{\alpha}(Q, {\bf v}_1) = \frac{\lambda}{2} > 0$,
while the second subrepresentation has 
\[
\deg_\alpha(Q, {\bf v}_2) = i (i\lambda) (1) = - \lambda < 0, \quad
\rank(Q, {\bf v}_2) = 1.
\]
Therefore $\mu_{\alpha}(Q, {\bf v}_2) = - \lambda < 0$, and  the slope vector $\nu$ associated to ${\bf v^*}$ is 
\[
\nu = \left( \frac{\lambda}{2}, \frac{\lambda}{2}, - \lambda \right)
\in \R^3 = \R^{\rank(Q,{\bf v})}.
\]
\end{example}

  In the special case when \(A \in \Rep(Q, {\mathbf v})\) is a
  critical point of \(f = \|\Phi - \alpha\|^2,\)
  Proposition~\ref{prop:momentmapdetermined} shows that there exists a
  canonical splitting of $A$, given by the eigenspace decomposition
  associated to $i\beta = i(\Phi(A) - \alpha)$. In this case, we say
  that the {\em critical type} of $A$ is the critical type of this
  canonical splitting.

Let $\mathcal{T}$ be the set of all possible critical types for
elements of $\Rep(Q, {\bf v})$. The description of the critical
points of $\| \Phi - \alpha \|^2$ by their critical type, along with the convergence of the
gradient flow proven in Theorem \ref{thm:gradflowconvergence}, leads to the
following Morse stratification of the space $\Rep(Q, {\bf v})$.

\begin{definition} \label{def:Morsestratification}
Let $\gamma(x,t)$ denote the negative gradient flow of $f = \|\Phi -
\alpha\|^2$ on $\Rep(Q, {\mathbf v})$. 
Let $C_{\bf v^*}$ denote the set of critical points of $f$ of critical type ${\bf v^*}$. We define the {\em analytic (or Morse-theoretic) stratum of
  critical type ${\bf v^*}$} to be 
\begin{equation}\label{eq:stratification} 
S_{\bf v^*} := \left\{ A \in \Rep(Q, {\bf v}) \, \mid \, \lim_{t\to \infty} \gamma(A,t) \in
C_{{\bf v}^*}\right\}.
\end{equation}
\end{definition}

By Theorem \ref{thm:gradflowconvergence}, every point limits to some
critical point, so $\displaystyle{\bigcup_{ {\bf v}^* \in \mathcal{T}} S_{ {\bf v}^*}
  = \Rep(Q, {\bf v})}$.  We call this the {\em analytic (or
  Morse-theoretic) stratification} of the space $\Rep(Q, {\mathbf
  v})$.  In
Sections \ref{sec:HN} and \ref{sec:local} we show that this analytic 
stratification has good local properties in the sense of 
\cite[Proposition 1.19]{AtiBot83}. The negative gradient flow also allows us to describe the topology of
the analytic strata $S_{{\bf v}^*}$ in terms of that of the
critical sets.

\begin{proposition}\label{prop:def-retract-to-Cnu}
  Let \(Q = ({\mathcal I}, E), {\mathbf v} \in \Z^{\mathcal I}_{\geq 0}, G = \prod_{\ell \in
    {\mathcal I}} U(V_\ell),\) and \(\Phi: \Rep(Q, {\mathbf v}) \to
  \g^* \cong \g\) be as
  in Lemma~\ref{lem:eigenvaluesdetermined}. Let $\gamma(x,t)$ denote
  the negative gradient flow with respect to \(f = \|\Phi -
  \alpha\|^2,\) and $C_{ {\bf v}^*}$ and $S_{{\bf v}^*}$ the respective critical set and analytic
  stratum of a fixed splitting type ${\bf v}^*$. Then the flow
  $\gamma(x,t)$ restricted to $S_{{\bf v}^*}$ defines a $G$-equivariant
  deformation retraction of $S_{{\bf v}^*}$ onto $C_{{\bf v}^*}$. In particular, 
\[
H^*_G(S_{{\bf v}^*}) \cong H^*_G(C_{{\bf v}^*}).
\]
\end{proposition}

\begin{proof}
Theorem
  \ref{thm:gradflowconvergence} in this section shows that 
$\gamma(x,t)$ is defined for all $t$ and converges to a critical point
of $f$. The same argument as in \cite{Ler04}
  allows us to conclude that taking the limit of the negative gradient flow
defines a deformation retract $S_{{\bf v}^*} \rightarrow C_{{\bf v}^*}$.
Since the gradient flow equations
  \eqref{eq:negative-gradient} are $G$-equivariant, the flow is
  $G$-equivariant, as is the deformation retract $S_{{\bf v}^*} \to C_{{\bf v}^*}$. 
\end{proof}

In Proposition \ref{prop:cohomology-stratum} we provide an explicit description of $H^*_G(S_{{\bf v}^*})$ in terms of lower-rank representation varieties.

\subsection{Distance decreasing formula along the flow}\label{subsec:distance-decreasing} 

In this section we prove a distance decreasing formula for the
distance between the $G$-orbits of two solutions $A_1(t)$, $A_2(t)$ to
the gradient flow equations \eqref{eq:gradient-flow-def}, whose
initial conditions $A_1(0)$, $A_2(0)$ are related by an element of
$G_\C$ (see Remark \ref{rem:momentmapclosealongflow}). The estimates obtained in the proof of this result are an important part of 
the proof of Theorem \ref{theorem:equivstratification} in the next section. It can
also be viewed as a quiver analogue of the distance decreasing formula for Hermitian metrics on a holomorphic vector bundle
that vary according to the Yang-Mills flow given in \cite{Don85}. Since the space $\Rep(Q, {\bf v})$ is finite-dimensional then it is more convenient to do our analysis directly on the group $G_\C$ 
rather than on the space of metrics as in~\cite{Don85}. 

Throughout, let \(Q = ({\mathcal I}, E)\) be a quiver with specified
dimension vector \({\bf v} \in \Z^{\mathcal I}_{\geq 0}.\) Given $A(0)
\in \Rep(Q, {\bf v})$, let $g = g(t)$ be a curve in $G_\C$ that
satisfies the following flow equation with initial condition:
\begin{align}\label{eqn:groupflow}
\begin{split}
\frac{\partial g}{\partial t} g^{-1} & = i \left( \Phi(g(t) \cdot A(0)) - \alpha \right), \\
g(0) & = \id.
\end{split}
\end{align}
Equation \eqref{eqn:groupflow} is an ODE and $\Phi$ is a polynomial, so solutions exist locally and
are unique. We will show later that they exist for all time.

\begin{remark}\label{rem:groupflow}
By definition, $\frac{\partial g}{\partial t} g^{-1}$ is self-adjoint,
so $(g^*)^{-1} \frac{\partial g^*}{\partial t} = \frac{\partial
g}{\partial t} g^{-1}$ wherever defined. This will be used in the
sequel. 
\end{remark}

\begin{lemma}\label{lem:gradientflowgenerated}
Let \(A(0) \in \Rep(Q, {\bf v})\) as above and $g(t)$ a curve
satisfying~\eqref{eqn:groupflow}. Let $A(t) := g(t) \cdot A(0)$. Then
the curve
$A(t)$ in $\Rep(Q, {\bf v})$ satisfies the negative gradient flow equation
\eqref{eq:negative-gradient}, i.e. 
\begin{equation}
\frac{\partial A}{\partial t} = I \rho_{A} \left(\Phi(A(t)) - \alpha \right).
\end{equation}
\end{lemma}

\begin{proof}
Differentiating each component of $g(t) \cdot A(0)$ gives us
\begin{equation}
\frac{\partial A_a}{\partial t} = \frac{\partial}{\partial t} \left( g(t) A_a(0) g(t)^{-1} \right) = \frac{\partial g}{\partial t} A_a(0) g(t)^{-1} - g(t) A_a(0) g(t)^{-1} \frac{\partial g}{\partial t} g(t)^{-1}.
\end{equation}
This can be rewritten as
\begin{align*}
\frac{\partial A_a}{\partial t} & = \left[ \frac{\partial g}{\partial t} g(t)^{-1}, A_a(t) \right] \\
 & = \left[ i \left( \Phi(A) - \alpha \right) , A_a(t) \right] \\
 & = I \rho_A \left(\Phi(A) - \alpha \right). \qedhere
\end{align*}
\end{proof}

We now define a function which measures the distance
between a positive self-adjoint matrix and the identity matrix (cf. \cite{Don85}). This
is the key ingredient allowing us to analyze the distance between
$G$-orbits of different negative gradient flows in
Section~\ref{sec:HN}. 
For a positive self-adjoint $h \in G_\C$, let
\begin{equation}\label{eq:def-sigma}
\sigma(h) := \tr h + \tr h^{-1} - 2 \rank(Q, {\bf v}).
\end{equation}
Note that $\sigma(h) \geq 0$ for all $h$, with equality if and only if
$h = \id$. 

It is convenient to also define a shifted version of the moment map
$\Phi$. 

\begin{definition}
Let $h \in G_\C$ be self-adjoint. Then we define for \(A \in \Rep(Q,
{\bf v})\) 
\begin{equation}
\Phi_h (A) := i \sum_{a \in E} \left[ A_a, h A_a^* h^{-1} \right].
\end{equation}
\end{definition}

In the special case when $h = g^{-1} (g^*)^{-1}$ for some $g \in
G_\C$, a computation shows that 
\begin{equation}\label{eq:relatemetricmomentmap}
g \Phi_h(A) g^{-1} = \Phi(g \cdot A)
\end{equation}
(this is analogous to the well-known formula for the change in the curvature of the unitary connection associated to a holomorphic structure on a Hermitian vector under a change in the Hermitian metric, see for example \cite{Don85}).

Now let $A_1(0)$ and $A_2(0)$ be two initial conditions related by
some $g_0 \in G_\C$, i.e.  $A_2(0) = g_0 \cdot A_1(0)$. Let $g_1(t)$ and $g_2(t)$ be
the corresponding solutions to equation \eqref{eqn:groupflow}. Then 
the corresponding solutions to the gradient flow equations are 
$A_1(t) = g_1(t) \cdot A_1(0)$ and $A_2(t) = g_2(t) \cdot A_2(0) =
g_2(t) \cdot g_0 \cdot A_1(0)$. Let $\overline{g}(t) = g_2(t) \cdot g_0 \cdot
g_1(t)^{-1}$ be the element of $G_\C$ that connects the two flows,
i.e. $\overline{g}(t) \cdot A_1(t) = A_2(t)$, and let 
\begin{equation}\label{eqn:metric-difference}
h(t) := \overline{g}(t)^{-1} \left( \overline{g}(t)^* \right)^{-1} .
\end{equation} 

The following is the main result
of this section; it states that the discrepancy between $h(t)$ and the
identity matrix, measured by the function $\sigma$
in~\eqref{eq:def-sigma}, is non-increasing along the flow. 

\begin{theorem}\label{thm:distancedecreasing}
Let $h(t)$ be as in \eqref{eqn:metric-difference}. Then 
$\frac{\partial }{\partial t} \sigma(h(t)) \leq 0$.
\end{theorem}

\begin{proof} 
Differentiating \(\overline{g} = g_2 g_0 g_1^{-1}\) and \(h =
\overline{g}^{-1}(\overline{g}^*)^{-1}\) yields

\begin{equation}
\frac{\partial \overline{g}}{\partial t}\overline{g}^{-1} =
\frac{\partial g_2}{\partial t} g_2^{-1} - \overline{g} \left( \frac{
\partial g_1}{\partial t} g_1^{-1} \right) \overline{g}^{-1}
\end{equation}

and 
\begin{align}\label{eqn:distancedecreasingmetricflow}
\begin{split}
\frac{\partial h}{\partial t} & = -\overline{g}^{-1} \left(
\frac{\partial g_2}{\partial t} g_2^{-1} \right) (\overline{g}^*)^{-1} + \left(
\frac{\partial g_1}{\partial t} g_1^{-1} \right) \overline{g}^{-1} (\overline{g}^*)^{-1} \\
& \quad \quad - \overline{g}^{-1} \left( (g_2^*)^{-1} \frac{\partial
g_2^*}{\partial t} \right) (\overline{g}^*)^{-1} + \overline{g}^{-1} (\overline{g}^*)^{-1} \left(
(g_1^*)^{-1} \frac{\partial g_1^*}{\partial t} \right). 
\end{split}
\end{align}

The observations in Remark~\ref{rem:groupflow}, and equations ~\eqref{eqn:groupflow}
and~\eqref{eqn:distancedecreasingmetricflow}, in turn yield

\begin{align}\label{eqn:derivativemetric}
\begin{split}
\frac{\partial h}{\partial t} & = -2 \overline{g}^{-1} i \left(
 \Phi(A_2(t)) - \alpha \right) \overline{g} h(t) + i \left(
 \Phi(A_1(t)) - \alpha \right) h(t) + h(t) i \left( \Phi(A_1(t)) -
 \alpha \right) \\ & = -2 i \left( \Phi_{h(t)} (A_1(t)) - \alpha
 \right) h(t) + i \left( \Phi(A_1(t)) - \alpha \right) h(t) + h(t) i
 \left( \Phi(A_1(t)) - \alpha \right). 
\end{split}
\end{align}

Furthermore, since $G_\C$ is a product of general linear groups then we can multiply elements of $\g_\C$ to give
\begin{equation}
[u, vw]  = v[u,w] + [u, v]w ,
\end{equation}
which in particular implies 
\([u, v^{-1}]  = -v^{-1} [u, v] v^{-1}\) and 
\([u, v]  = -v [u, v^{-1} ] v\) for \(v \in \g_\C\) invertible. We then
have 
\begin{align}\label{eqn:productformula}
\begin{split}
i \Phi_h(A) - i \Phi(A) & = - \sum_{a \in E} \left(  [A_a, h A_a^* h^{-1}] - [A_a, A_a^*] \right) \\
 & = - \sum_{a \in E} \left( [A_a, h A_a^* h^{-1} ] - [A_a, h h^{-1} A_a^*] \right) \\
 & = - \sum_{a \in E}  \left[A_a, h [A_a^*, h^{-1}] \right] \\
 & = - \sum_{a \in E} \left(  h \left[ A_a, [A_a^*, h^{-1} ] \right] + [A_a, h] [A_a^*, h^{-1}] \right) \\
 & = -h \sum_{a \in E} \left( \left[ A_a, [A_a^*, h^{-1} ] \right] -
 [A_a, h^{-1}] h [A_a^*, h^{-1}] \right).
\end{split}
\end{align}
A similar computation yields
\begin{equation}\label{eqn:second-productformula}
i \left( \Phi_{h} (A) - \Phi (A) \right) = \sum_{a \in E} \left(
\left[ A_a, [A_a^*, h] \right] - [A_a^*, h] h^{-1} [A_a, h] \right)
h^{-1}. 
\end{equation}
Differentiating $\sigma(h(t))$ gives us
\begin{equation}
\frac{\partial}{\partial t} \sigma(h(t)) = \tr \frac{\partial h}{\partial t} - \tr \left( h^{-1} \frac{\partial h}{\partial t} h^{-1} \right). 
\end{equation}
Equation~\eqref{eqn:derivativemetric} then shows that 
\begin{equation}
\tr \frac{\partial h}{\partial t} = -2 i \tr \left( \left( \Phi_{h(t)} (A_1(t) ) - \Phi (A_1(t)) \right) h(t)  \right)
\end{equation}
and
\begin{equation}
-\tr \left( h^{-1} \frac{\partial h}{\partial t} h^{-1} \right)= 2 i \tr \left(  h(t)^{-1} \left( \Phi_{h(t)} (A_1(t) ) - \Phi (A_1(t)) \right) \right).
\end{equation}
Combined with equation \eqref{eqn:second-productformula} we see that
\begin{align*}
\tr \frac{\partial h}{\partial t} & = -2 \sum_{a \in E} \tr \left( \left[A_a, [A_a^*, h] \right] - [A_a^*, h] h^{-1} [A_a, h] \right) \\
 & =  2 \sum_{a \in E} \tr  \left( [A_a^*, h] \overline{g}^* \overline{g} [A_a, h] \right) \\
 & = - 2 \sum_{a \in E} \tr \left( (\overline{g} [A_a, h])^* \overline{g} [A_a, h] \right) \leq 0.
\end{align*}
Similarly, equation \eqref{eqn:productformula} shows that
\begin{align*}
-\tr h^{-1} \frac{\partial h}{\partial t} h^{-1} & = -2 \sum_{a \in E} \tr \left( \left[A_a, [A_a^*, h^{-1}] \right] - [A_a, h^{-1}] h [A_a^*, h^{-1}] \right) \\
 & = 2 \sum_{a \in E} \tr  \left( [A_a, h^{-1}] \overline{g}^{-1} (\overline{g}^*)^{-1} [A_a^*, h^{-1}] \right) \\
 & = - 2 \sum_{a \in E} \tr \left( \left( (\overline{g}^*)^{-1} [A_a^*, h^{-1}] \right)^* (\overline{g}^*)^{-1} [A_a^*, h^{-1}] \right) \leq 0. \qedhere
\end{align*}
\end{proof}

We next show that a bound on the value of $\sigma$ yields a bound on
the distance to the identity matrix. 

\begin{proposition}\label{prop:hclosetoidentity}
Let \(g \in G_\C\) and \(h = g^{-1}(g^*)^{-1}.\) 
For each $\varepsilon > 0$ there exists $\delta$ such that if
$\sigma(h) < \delta$ then 
\begin{equation}
\| h - \id \| + \| h^{-1} - \id \| < \varepsilon.
\end{equation}
\end{proposition}

\begin{proof}
Since \(h = g^{-1}(g^*)^{-1},\) $h$ is positive and unitarily
diagonalizable.  Let $h_d := u^* h u$ be the diagonalisation of $h$
by a unitary matrix $u$. Since both $\sigma$ and the inner product are
invariant under unitary conjugation, then $\sigma(h) = \sigma(h_d)$,
$\| h_d - \id \| = \| h - \id \|$ and $\| h_d^{-1} - \id \| = \|
h^{-1} - \id \|$. Hence the problem reduces to studying diagonal,
positive, self-adjoint matrices. Let $\lambda_1, \ldots, \lambda_n$ be the eigenvalues of $h$. Then 
\[
\| h_d - \id \|  = \left( \sum_{j=1}^n (\lambda_j  - 1)^2 \right)^{\frac{1}{2}}
\leq \sum_{j=1}^n \left| \lambda_j - 1 \right|,
\]
and similarly
\begin{equation}
\| h_d^{-1} - \id \| \leq \sum_{j=1}^n \left| \lambda_j^{-1} - 1
\right|. 
\end{equation}
Choose \(\delta > 0\) such that for all $j = 1, \ldots, n$, if \(\lambda_j + \lambda_j^{-1} - 2 <
\delta\) then \(|\lambda_j - 1| + |\lambda_j^{-1}-1| < \frac{1}{n}
\varepsilon.\) If \(\sigma(h) < \delta,\) from the definition of
$\sigma$ it follows that 
\(\lambda_j + \lambda_j^{-1} - 2 < \delta\) for each $j$ and we obtain 
\begin{equation}
\| h_d - \id \| + \| h_d^{-1} - \id \| \leq \sum_{j=1}^n \left| \lambda_j -1 \right| + \left| \lambda_j^{-1} - 1 \right| < \varepsilon.
\end{equation}
The result follows.
\end{proof}

The next result shows that, for any \(g \in G_\C,\) if $\sigma( g^{-1}
(g^*)^{-1} )$ is small, then there exists some $\tilde{g}$ in the
$G$-orbit of $g$ that is close to the identity. Set $h := g^{-1}
(g^*)^{-1}$. Then as in the proof above, $h$ is positive and unitarily
diagonalizable; let \(h_d = u^* h u\) for $u$ unitary and $h_d$
diagonal. Since the eigenvalues are positive and real, there exists a
well-defined positive square root, denoted $\sqrt{h_d}$. Let
\begin{equation}\label{eqn:tilde-g}
\tilde{g} := u (\sqrt{h_d})^{-1} u^*.
\end{equation}
We have the following.

\begin{lemma}\label{lemma:tildegclose-if-sigmaclose} 
Let \(g \in G_\C\) and $\tilde{g}$ as in \eqref{eqn:tilde-g}. Then \(\tilde{g} \in G
\cdot g.\) Moreover, for any \(\varepsilon > 0\) there exists a
$\delta > 0$ such that if \(\sigma(g^{-1}(g^*)^{-1}) < \delta\) then
\(\|\tilde{g} - \id\| < \varepsilon.\) 
\end{lemma} 

\begin{proof}
A calculation shows that
$\tilde{g}^{-1} (\tilde{g}^*)^{-1} = h$.  Since $\tilde{g}^{-1}
(\tilde{g}^*)^{-1} = h = g^{-1} (g^*)^{-1}$,
then we have
\begin{equation}\label{eqn:unitary-element}
\tilde{g} g^{-1} (g^*)^{-1} \tilde{g}^* = \id \, \Leftrightarrow \, (\tilde{g}
g^{-1})^{-1} = (\tilde{g} g^{-1})^* .
\end{equation}
Define 
\begin{equation}\label{eqn:g_u}
g_u := \tilde{g} g^{-1} ,
\end{equation}
and note that \eqref{eqn:unitary-element} shows that $g_u$ is unitary. Note also that $\tilde{g} = g_u g \in G \cdot g$. Therefore, it only remains to show that $\| \tilde{g} - \id
\| < \varepsilon$ if $\sigma(h)$ is small enough.

Observe $\tilde{g}$ is self-adjoint and positive, with eigenvalues
$\tilde{\lambda}_1, \ldots, \tilde{\lambda}_n$ equal to those of
$(\sqrt{h_d})^{-1}$. Therefore
\begin{equation}
\| \tilde{g} - \id \| = \| u^{-1} \tilde{g} u - \id \| = \|
\sqrt{h_d}^{-1} - \id \| = 
\left( \sum_{j=1}^n (\tilde{\lambda}_j - 1)^2 \right)^{\frac{1}{2}}
\leq \sum_{j=1}^n \left| \tilde{\lambda}_j 
- 1\right|. 
\end{equation}
Since each $\tilde{\lambda}_j$ is positive, then $\left|
\tilde{\lambda}_j - 1 \right| \leq \left| \tilde{\lambda}_j^2 - 1
\right|$ and we have
\begin{equation}
\| \tilde{g} - \id \| \leq \sum_{j=1}^n \left| \tilde{\lambda}_j^2 - 1\right|,
\end{equation}
where $\tilde{\lambda}_j^2$ is the $j^{th}$ eigenvalue of $h^{-1}$.
Now the same argument as in Proposition~\ref{prop:hclosetoidentity}
shows that for any given $\varepsilon > 0$ there exists a $\delta > 0$
such that if \(\sigma(h^{-1}) = \sigma(h) < \delta\) then $\sum_{j=1}^n
\left| \tilde{\lambda}_j^2 - 1\right| < \varepsilon$. The result
follows. 
\end{proof}

As a corollary of this result, we have a new proof of Lemma
\ref{lemma:exists-alltime}, which shows that the solutions
to~\eqref{eqn:groupflow} exist for all time. Although this result
has already been proven in Lemma \ref{lemma:exists-alltime}, we
include this proof here as it follows easily from the preceding
results, and the method of proof may be of independent interest.

\begin{corollary}
For any initial condition $A(0) \in \Rep(Q, {\bf v})$, the solution to
equation \eqref{eqn:groupflow} exists for all $t$.
\end{corollary}

\begin{proof}
Let $A(0) \in \Rep(Q, {\bf v})$, and let
$g_1(t)$ denote the associated solution to \eqref{eqn:groupflow}. Local
existence for ODEs shows that $g_1(t)$ exists for $t \in [0, T)$ for
  some \(T > 0\), so 
it remains to show that the solution can be extended past $t=T$.

For \(\varepsilon > 0,\) let \(\delta > 0\) be as in
Lemma~\ref{lemma:tildegclose-if-sigmaclose}. 
Since the solution $g_1(t)$ is continuous, there exists $t_0 \in [0,T)$ with $\sigma
\left( g_1(t_0)^* g_1(t_0) \right) < \delta$. 
Let $g_2(t)$
be the solution to \eqref{eqn:groupflow} with initial condition
$g_1(t_0) \cdot A_0$. This is just a translation of the first
solution, so $g_2(t) = g_1(t + t_0)g_1(t_0)^{-1}$, and the problem reduces to
extending the solution for $g_2(t)$ past $T-t_0$.

Since \(g_1(0), g_2(0)\) are related by an element of $G_\C$, then as
for the above analysis we may define $\overline{g}(t)$ so that
$\overline{g}(t) = g_2(t)g_1(t_0)g_1(t)^{-1}$, and by
Theorem~\ref{thm:distancedecreasing},
\(\sigma(\overline{g}(t)^{-1}(\overline{g}(t)^*)^{-1}) \leq
\sigma(g_1(t_0)^{-1}(g_1(t_0)^*)^{-1}) = \sigma(g_1(t_0)^* g_1(t_0)) <
\delta\) for \(t \in [0, T-t_0).\) Here we use that \(\sigma(h) =
  \sigma(h^{-1})\) by definition of $\sigma$ for any positive-definite
  $h$. Then Lemma~\ref{lemma:tildegclose-if-sigmaclose} shows that for
  every \(t \in [0, T - t_0),\) the $G$-orbit of $g(t)$ under left
    multiplication intersects a bounded set in $G_\C$. Since $G$
    itself is compact, we may conclude $g(t)$ remains in a compact
    set. Therefore we can extend the solution for $g(t)$ and hence
    $g_2(t)$ past $T-t_0$, as desired. 
\end{proof}

Finally, we show that estimates on $\sigma(h) =
\sigma(g^{-1}(g^*)^{-1})$ yields distance estimates between moment map
values of $g \cdot A$ and $g_u^{-1}\cdot A$. Such an estimate is
crucial in the next section. 

\begin{proposition}\label{prop:momentmapclose}
Let \(g \in G_\C, h = g^{-1}(g^*)^{-1},\) and $g_u \in G$ as in \eqref{eqn:g_u}. 
Suppose $A \in \Rep(Q,{\bf v})$. Then for every $\varepsilon > 0$
there exists $\delta > 0$ 
such that if $\sigma(h) < \delta$ then 
\begin{equation}
\| \Phi( g \cdot A) - \Phi(g_u^{-1} \cdot A) \| < \varepsilon. 
\end{equation}
\end{proposition}

\begin{proof}

Since $\Phi$ is $G$-equivariant and \(g = g_u^{-1}\tilde{g},\)
\[
\Phi( g \cdot A)  = g_u^{-1} \Phi( \tilde{g} \cdot A ) g_u \quad
\mbox{and} \quad 
\Phi( g_u^{-1} \cdot A)  = g_u^{-1} \Phi(A) g_u.
\]
Moreover, since the inner product is invariant under the conjugate
action of $G$, it is sufficient to find a bound on $\| \Phi(
\tilde{g} \cdot A) - \Phi(A) \|$, or equivalently
\(\|\tilde{g} \Phi_h(A) \tilde{g}^{-1} - \Phi(A)\|,\)
by~\eqref{eq:relatemetricmomentmap}. First, using
Proposition~\ref{prop:hclosetoidentity} and Lemma~\ref{lemma:tildegclose-if-sigmaclose}, let \(\delta > 0\) be
such that if \(\sigma(h) < \delta\) then \(\tilde{g} - \id\) and $h$
are both bounded in norm by some constant. 
We have then  
\begin{align*}
\| \tilde{g} \Phi_h(A) \tilde{g}^{-1} - \Phi_h(A) \| & = \| \tilde{g} \left[ \Phi_h(A), \tilde{g}^{-1} \right] \| \\
 & =  \| \tilde{g} \left[ \Phi_h(A), \tilde{g}^{-1} - \id \right] \| \\
 & \leq 2 \left( \| \tilde{g} - \id \| + \| \id \| \right) \| \Phi_h(A) \| \| \tilde{g}^{-1} - \id \| \\
 & \leq C_1 \| \tilde{g}^{-1} - \id \|
\end{align*}
for some constant $C_1>0$. 
A similar computation using~\eqref{eqn:productformula} yields 
\[
\| \Phi_h(A) - \Phi(A) \| \leq C_2 \| h^{-1} - \id \|
\]
for a constant $C_2 > 0$. Combining these two estimates gives
\[
\| \Phi( \tilde{g} \cdot A) - \Phi(A) \| \leq C_1 \| \tilde{g}^{-1} -
\id \| + C_2 \| h^{-1} - \id \| .
\]
Again using Proposition~\ref{prop:hclosetoidentity}
and~\ref{lemma:tildegclose-if-sigmaclose}, and the fact that \(\sigma(h)
=\sigma(h^{-1})\) for any positive-definite $h$, we conclude that (after possibly shrinking
$\delta$)  if $\sigma(h) < \delta$ then 
\begin{equation}
C_1 \| \tilde{g}^{-1} - \id \| + C_2 \| h^{-1} - \id \| < \varepsilon. \qedhere
\end{equation}
\end{proof}

\begin{remark}\label{rem:momentmapclosealongflow}
In the notation of Theorem \ref{thm:distancedecreasing}, given initial
conditions $A_1(0)$, $A_2(0) = g_0 \cdot A_1$, solutions $A_1(t),
A_2(t)$ to the gradient flow with initial conditions $A_1(0)$,
$A_2(0)$ respectively, and $g(t) \in G_\C$ the group action that
connects the two flows, let $g_u(t) \in G$ be the unitary element
associated to $g(t)$ from \eqref{eqn:g_u}.
Then Theorem~\ref{thm:distancedecreasing} and
Proposition~\ref{prop:momentmapclose} together imply that for any
$\varepsilon > 0$ there exists $\delta > 0$ such that 
if
$\sigma (g_0^* g_0) < \delta$, then 
$$\|\Phi(A_2(t)) - \Phi(g_u(t) \cdot A_1(t)) \| < \varepsilon$$ for all
$t$. In other words, the $G$-orbits of the two solutions remain
close.  
\end{remark}

\section{The Harder-Narasimhan stratification}\label{sec:HN}

In this section, we relate the Morse-theoretic stratification of $\Rep(Q, {\mathbf
  v})$ obtained in Section~\ref{sec:gradientflow}
to the Harder-Narasimhan
stratification (recalled below) of $\Rep(Q, {\mathbf v})$ with respect to the same
central parameter $\alpha$. This latter stratification is defined in terms of
slope stability conditions similar to the case of holomorphic bundles, and the content of this section is to exhibit the
relationship between the algebraic-geometric description of this
stratification (via stability) and the analytic description (via
gradient flow). We restrict the discussion to the unframed case
$\Rep(Q, {\mathbf v})$, which (see
Remarks~\ref{remark:framed-vs-unframed}
and~\ref{remark:cotangent-double-arrows}) covers all other cases of
interest.

\subsection{Slope and stability for representations of quivers}\label{subsec:slope-stability-def}

To set the notation, we briefly recall
some preliminary definitions. 
Let \(Q = ({\mathcal I}, E)\) be a quiver, 
${\mathbf v} \in \Z^{\mathcal I}_{\geq 0}$ a dimension vector, 
and \(\alpha \in (i\R)^{\mathcal I}\) a central 
parameter. 
In order to define
the slope stability condition with respect to the parameter $\alpha$,
it will be necessary to compare the $\alpha$-slope (as in
Definition \ref{def:degree-slope}) of a
representation \(A \in \Rep(Q,{\mathbf v})\) with that of its
invariant subspaces. We make this notion more precise below
(see \cite[Definition 2.1]{Rei03}). 

\begin{definition}\label{def:alphastable}
A representation $A \in \Rep(Q, {\bf v})$ is called {\em $\alpha$-stable}
(resp. {\em $\alpha$-semistable}) if for every proper
subrepresentation $A' \in \Rep(Q, {\bf v'})$ of $A$ 
we have
\begin{equation}
\mu_{\alpha}(Q, {\bf v'}) < \mu_{\alpha}(Q, {\bf v}), \quad \left(\text{resp.  } \mu_{\alpha}(Q, {\bf v'}) \leq \mu_{\alpha}(Q, {\bf v}) \right).
\end{equation}
We denote by $\Rep(Q, {\mathbf v})^{\alpha-st}$ (resp. $\Rep(Q, {\mathbf v})^{\alpha-ss}$) the subset in $\Rep(Q, {\mathbf v})$ of $\alpha$-stable (resp. $\alpha$-semistable) representations. A representation is called {\em $\alpha$-polystable} if it is the direct sum of $\alpha$-stable representations of the same $\alpha$-slope. 
\end{definition}

We now briefly recall the notion of stability arising from geometric
invariant theory (cf. Section 2 of \cite{Kin94}). Assume now that the parameter \(\alpha = (\alpha_\ell)_{\ell \in {\mathcal I}}\)
associated to the quiver $Q$ is integral, i.e. satisfies \(i \alpha_{\ell} \in \Z\)
for all \(\ell \in {\mathcal I}.\) Let \(\chi_{\alpha}: G_{\C} =
\prod_{\ell \in {\mathcal I}} GL(V_{\ell}) \to \C\) be the character
of $G_{\C}$ given by 
\[
\chi_\alpha(g) := \prod_{\ell \in {\mathcal I}}
\det(g_{\ell})^{i\alpha_{\ell}}.
\]
Using $\chi_{\alpha}$, we define a lift of the action of $G_{\C}$ from
$\Rep(Q, {\bf v})$ to the trivial complex line bundle $L := \Rep(Q, {\bf v})
\times \C$ by 
\begin{equation}
g \cdot \left( \left( A_a \right)_{a \in E}, \xi \right) := \left(
\left( g_{\inw(a)} A_a g_{\out(a)}^{-1} \right)_{a \in E}, \chi_\alpha(g) \xi \right).
\end{equation}

As noted by King \cite[Section 2]{Kin94}, since the diagonal
one-parameter subgroup $\{ (t \cdot \id, \ldots, t \cdot \id)
\, : \, t \in \C^* \} \subseteq G_\C$ acts trivially on $\Rep(Q, {\bf
  v})$, the following definition of stability differs slightly
from the usual definition (where a stable point has finite stabiliser
in $G_\C$). In the rest of the paper we use the following variant of
King's definition.

\begin{definition}\label{def:chistable}
Let $A = \left( A_a \right)_{a \in E} \in \Rep(Q, {\bf v })$. Then we
say $A$ is \emph{$\chi_{\alpha}$-semistable} if for any $\xi \in \C
\setminus \{0\},$ the $G_{\C}$-orbit closure $\overline{G_{\C} \cdot
(x, \xi)}$ in $L$ 
is disjoint 
from the zero section of $L$. A representation $A$ is \emph{$\chi_\alpha$-stable} if the orbit $G_{\C} \cdot (x, \xi)$ is closed, and the dimension of the orbit in $\Rep(Q, {\bf v})$ satisfies $\dim_\C G_\C \cdot x = \dim_\C G_\C - 1$. A representation is \emph{$\chi_\alpha$-polystable} if it is both $\chi_\alpha$-semistable and the direct sum of $\chi_\alpha$-stable representations.
\end{definition}

\begin{remark} The definition of $\chi$-stability given by King
  \cite[Definition 2.1]{Kin94} is equivalent to that of
  $\chi_\alpha$-polystability given above. \end{remark}

The main result of \cite{Kin94} is that when $\alpha$ is integral, the 
$\chi_{\alpha}$-(semi)stability condition above is equivalent to the
$\alpha$-(semi)stability conditions of
Definition~\ref{def:alphastable}. 
In this paper we also analyze Hamiltonian
quotients of $\Rep(Q,{\mathbf v})$, so we now recall the relationship
between the above $\alpha$-stability conditions and moment map level sets.  The
following lemma appears in \cite[Section 6]{Kin94} and also in \cite{Nak99}, but  we 
provide here a different proof, which comes from our results in
Section~\ref{sec:gradientflow} on the gradient
flow of $\| \Phi_{\R} - \alpha \|^2$ on the space $\Rep(Q, {\mathbf v})$.

\begin{lemma} (\cite[Corollary 3.22]{Nak99})\label{lem:KempfNess}
Let \(Q = ({\mathcal I}, E)\) be a quiver with specified dimension
vector \({\bf v} \in \Z^{\mathcal I}_{\geq 0}.\) Let $\alpha \in
(i\Z)^{\mathcal I}$ be an
integral central parameter. 
If a representation $A \in \Rep(Q, {\bf v})$ is $\alpha$-polystable then there exists $g \in G_{\C}$
such that $\Phi(g \cdot A) = \alpha$. If a representation
$A \in \Rep(Q, {\bf v})$ is $\alpha$-semistable then the orbit closure
$\overline{G_\C \cdot A}$ has non-trivial intersection with
$\Phi^{-1}(\alpha)$.
\end{lemma}

\begin{proof}
  The results of the previous section show that the negative gradient flow
  $\gamma(A, t)$ of $\| \Phi - \alpha \|^2$ is contained in a compact set and generated by the action of a path $g(t) \in G_\C$, for any initial
  condition $A \in \Rep(Q, {\bf v})$. In this
  situation, if $A$ is $\chi_\alpha$-semistable then (since $\overline{
    G_\C \cdot (x, \xi)}$ is disjoint from the zero section $\Rep(Q,
  {\bf v}) \times \{ 0 \}$ for $\xi \neq
  0$) we see that $\xi(t) = \chi_\alpha(g(t))\xi$ is bounded away from zero. 
 For $(A, \xi) \in L$, define
\begin{equation}\label{eqn:Donaldson-functional}
F(A, \xi) = \frac{1}{2} \| A \|^2 + \frac{1}{2} \log \| \xi \|^2.
\end{equation}
(This $F$ is the analogue of the
  Donaldson functional in this situation.) 
A calculation shows that $\frac{\partial F}{\partial t} = - \| \Phi -
\alpha \|^2$, so the lemma follows from the fact proven above: that if $A$ is
$\chi_\alpha$ semistable, then  $\xi(t)$ is bounded away from
zero along the gradient flow, and hence $F$ is bounded below.

If $A$ is $\alpha$-polystable then the $G_\C$-orbit of $A$ is closed,
and so the above result shows that there exists $g \in
G_\C$ such that $g \cdot A$ is a critical point of $F$, i.e. $\Phi(g
\cdot A) - \alpha = 0$. \end{proof}

Having established the connection with GIT, we now turn our attention
to the definition of the Harder-Narasimhan stratification of
$\Rep(Q,{\mathbf v})$ with respect to $\alpha$. The construction uses
the definition of slope stability in
Definition~\ref{def:alphastable}. 

\begin{definition} 
Let \(Q = ({\mathcal I}, E)\) be a
quiver. Suppose \(A \in \Rep(Q,{\mathbf v})\) is a representation of
$Q$ with associated hermitian vector spaces $\{V_{\ell}\}_{\ell \in
{\mathcal I}}$, and similarly \(A' \in \Rep(Q,{\mathbf v'})\) with
\(\{V'_{\ell}\}_{\ell \in {\mathcal I}}.\) We say a collection of
linear homomorphisms \(\psi_{\ell}: V_{\ell} \to V'_{\ell}\) is a {\em
homomorphism of representations of quivers} if $\psi_{\ell}$
intertwines the actions of $A$ and $A'$, i.e. for all \(a \in E,\) 
\[
\psi_{\inw(a)} A_a = A'_a \psi_{\out(a)}.
\]
\end{definition}

We may also define a quotient representation in the standard
manner. 

\begin{definition}\label{def:quotient-rep}

Let \(A \in \Rep(Q,{\bf v})\) be a representation of a quiver 
\(Q = ({\mathcal I}, E)\) 
with associated hermitian vector spaces $\{V_{\ell}\}_{\ell \in
{\mathcal I}}$, and let \(A' \in \Rep(Q,{\bf v'})\) be a
subrepresentation of $A$ with $\{V'_\ell\}_{\ell \in {\mathcal I}}$. Then the {\em quotient representation}
\(\overline{A} = A/A' \in \Rep(Q, {\bf v} - {\bf v'})\) is defined to be the
collection of linear maps on the quotient vector spaces
\(\{V_{\ell}/V'_{\ell}\}_{\ell \in {\mathcal I}}\) induced by the
$A_a$, i.e.
\[
\bar{A}_a: V_{\out(a)}/V'_{\out(a)} \to V_{\inw(a)}/V'_{\inw(a)}.
\]
This is well-defined since $A$ preserves the $V'_{\ell}$.
\end{definition}

Using these definitions we can make sense
of an exact sequence of representations of quivers. We will use the
usual notation 
\[
\xymatrix{
0 \ar[r] & A \ar[r] & B \ar[r] & C \ar[r] & 0 
}
\]
for \(A \in \Rep(Q,{\bf v'}), B \in \Rep(Q,{\bf v}), C \in \Rep(Q,{\bf
v''})\) where ${\bf v''} = {\bf v} - {\bf v'}$, \(A \to B\) is an inclusion
with image a subrepresentation of $B$ and \(C \cong B/A.\) The
following lemma appears in \cite{Rei03}, and gives a quiver
analogue of 
well-known results for the case of holomorphic bundles.

\begin{lemma}(\cite[Lemma 2.2]{Rei03})
  \label{lem:exactsequenceinequality} 
Let $Q=({\mathcal I}, E)$ be a quiver. Suppose that 
\[
\xymatrix{
0 \ar[r] & A \ar[r] & B \ar[r] & C \ar[r] & 0 
}
\]
is a short exact sequence of representations of the quiver $Q$,
with $A \in \Rep(Q,{\bf v'})$, $B \in \Rep(Q,{\bf v})$, and $C \in \Rep(Q,{\bf
v''})$. Then 
\[
\mu_{\alpha}(Q, {\bf v'}) \leq \mu_{\alpha}(Q, {\bf v}) \Leftrightarrow
\mu_{\alpha}(Q, {\bf v}) \leq \mu_{\alpha}(Q, {\bf v''})
\]
and 
\[\mu_{\alpha}(Q, {\bf v'}) \geq
\mu_{\alpha}(Q, {\bf v}) \Leftrightarrow \mu_{\alpha}(Q, {\bf v}) \geq \mu_{\alpha}(Q, {\bf
v''}),
\]
with
\[
\mu_{\alpha}(Q, {\bf v'}) = \mu_{\alpha}(Q, {\bf v}) \Leftrightarrow
\mu_{\alpha}(Q,{\bf v}) = \mu_{\alpha}(Q, {\bf v''}).
\]
\end{lemma}

Similarly, the following Proposition is contained in 
\cite[Proposition 2.5]{Rei03}. (See also \cite[V.1.13, V.7.17]{Kob87}
for a discussion in the case of holomorphic bundles.)

\begin{proposition}(\cite[Proposition 2.5]{Rei03})\label{prop:maximalsemistable}
Let $Q = ({\mathcal I}, E)$ be a quiver, ${\mathbf v} \in
\Z^{\mathcal I}_{\geq 0}$ a dimension vector, and \(A \in \Rep(Q, {\mathbf
  v}).\) 
Let \(\alpha = (\alpha_{\ell})_{\ell \in {\mathcal I}}\) be a 
central parameter. Then there exists a unique sub-representation $A' \in \Rep(Q, {\bf v'})$ such that
\begin{enumerate}
\item \label{item:slopecondition} $\mu_{\alpha}(Q, {\bf \tilde{v}}) \leq \mu_{\alpha}(Q, {\bf v'})$ for all proper sub-representations $\tilde{A} \in \Rep(Q, {\bf \tilde{v}})$ of $A$, and 
\item \label{item:rankcondition} if $\tilde{A}$ is a proper
  subrepresentation of $A$ with $\mu_{\alpha}(Q, {\bf \tilde{v}}) =
  \mu_{\alpha}(Q, {\bf v})$, then either $\tilde{A} = A'$, or $\rank(Q, {\bf \tilde{v}}) < \rank(Q,
  {\bf v'})$.
\end{enumerate}
Such an $A' \in \Rep(Q, {\bf v'})$ is called the \emph{maximal $\alpha$-semistable
subrepresentation} of $A \in \Rep(Q, {\bf v})$.
\end{proposition}

With Proposition~\ref{prop:maximalsemistable} in hand, the proof of
the following is also standard (see also \cite[V.1.13, V.7.15]{Kob87}
for the case of holomorphic bundles).

\begin{theorem}(cf. \cite[Prop.2.5]{Rei03})
Let \(Q = ({\mathcal I}, E)\) be a quiver, ${\mathbf v} \in
\Z^{\mathcal I}_{\geq 0}$ a dimension vector, and 
$A \in \Rep(Q, {\bf v})$. 
Let \(\alpha = (\alpha_{\ell})_{\ell \in
{\mathcal I}}\) be a central parameter. 
Then there exists a canonical filtration of $A$ by subrepresentations
\begin{equation}\label{eq:HNfiltration}
0 = A_0 \subsetneq A_1 \subsetneq A_2 \ldots \subsetneq A_L = A,
\end{equation}
where $\displaystyle{A_j \in \Rep \left(Q, \sum_{k=1}^j {\bf v}_k \right)}$.
and each $A_j/A_{j-1}$ is the maximal $\alpha_j$-semistable
subrepresentation of $A/A_{j-1}$ for all $1 \leq j \leq L$, where 
\begin{equation}
\alpha_j : = \left. \alpha \right|_{\Vect(Q, {\bf v}_j)} - \frac{1}{\rank(Q, {\bf v}_j)} \left. \tr \alpha \right|_{\Vect(Q, {\bf v}_j)} \cdot \id
\end{equation}
is the trace-free stability parameter associated to ${\bf v}_j$. This filtration is referred to as the {\em Harder-Narasimhan filtration}.
\end{theorem}

We will call the length $L$ of this sequence~\eqref{eq:HNfiltration}
the {\em Harder-Narasimhan $\alpha$-length} of $A$. In particular, an
$\alpha$-semistable representation has Harder-Narasimhan
$\alpha$-length $1$. We will often abbreviate ``Harder-Narasimhan" as ``H-N''.  From the H-N filtration we may read off the following parameters. 

\begin{definition}\label{def:HNtype}(cf. Definition 2.6 of \cite{Rei03})
The {\em Harder-Narasimhan (H-N) type} of the filtration \eqref{eq:HNfiltration} is the vector ${\bf v^*} = ({\bf v}_1, \ldots, {\bf v}_L)$. The \emph{slope vector} associated to ${\bf v^*}$ is the ordered $\rank(Q, {\bf v})$-tuple
\begin{equation}
\left( \mu_{\alpha}(Q,{\bf v}_1), \ldots, \mu_{\alpha}(Q,{\bf v}_1), \ldots, \mu_{\alpha}(Q,{\bf v}_L), \ldots, \mu_{\alpha}(Q, {\bf v}_L) \right)
\in \R^{\rank(Q,{\bf v})},
\end{equation}
where there are $\rank(Q,{\bf v_j})$ terms equal to
$\mu_{\alpha}(Q,{\bf v}_j)$ for all $1 \leq j \leq L$. 
\end{definition}

\begin{definition}\label{def:HNstratification}
  Given a fixed parameter $\alpha$, let $\Rep(Q, {\bf v})_{{\bf v}^*} \subseteq \Rep(Q,{\bf v})$ denote the subset of
  representations of H-N type ${\bf v^*}$ with
  respect to the stability parameter $\alpha$. We call this the {\em Harder-Narasimhan (H-N)
    stratum of $\Rep(Q, {\mathbf v})$ of H-N type ${\bf v}^*$}. The {\em Harder-Narasimhan stratification} is the decomposition of $\Rep(Q, {\bf v})$ indexed by the (finite) set of all types:
\begin{equation}
\Rep(Q, {\bf v}) = \bigcup_{{\bf v}^*} \Rep(Q, {\bf v})_{{\bf v}^*}.
\end{equation}
\end{definition}

In \cite[Definition 3.6]{Rei03}, Reineke describes a partial ordering
on the set of H-N types analogous to the H-N partial ordering in
\cite{AtiBot83}, and in \cite[Proposition 3.7]{Rei03} he shows that the
closures of the subsets behave well with respect to this partial
order, i.e.
\begin{equation}\label{eq:HNclosures-partialorder} 
\overline{\Rep(Q, {\bf v})}_{{\bf v}^*} \subseteq \bigcup_{{\bf w}^* \geq {\bf v^*}} \Rep(Q, {\bf v})_{{\bf w}^*}. 
\end{equation}

Moreover, this decomposition is invariant under the natural
action of $G_{\C}$ on $\Rep(Q,{\bf v})$, since the definition of the
Harder-Narasimhan type of a representation is an isomorphism
invariant. 

\subsection{Comparison of stratifications}\label{subsec:equivstratification}

The main result of this section is that the analytic Morse 
stratification of $\Rep(Q, {\mathbf v})$ obtained in
Section~\ref{sec:gradientflow} is the same as the H-N stratification from Section~\ref{subsec:slope-stability-def}.

\begin{theorem}\label{theorem:equivstratification}
  Let $Q = ({\mathcal I}, E)$ be a quiver, 
 \({\bf v} \in \Z^{\mathcal I}_{\geq 0}\) a dimension vector, 
  $\Rep(Q, {\mathbf v})$ its associated representation space, and
  \(\Phi: \Rep(Q, {\mathbf v}) \to \g^{*} \cong \g \cong \prod_{\ell
    \in {\mathcal I}} \u(V_{\ell})\) a moment map for the standard
  Hamiltonian action of $G = \prod_{\ell \in {\mathcal I}}
  U(V_{\ell})$ on $\Rep(Q, {\mathbf v})$. 
Then the algebraic stratification of $\Rep(Q, {\bf v})$ by
Harder-Narasimhan type (as in Definition \ref{def:HNstratification})
coincides with the
analytic Morse stratification of $\Rep(Q, {\bf v})$ by the negative 
gradient flow of $f = \| \Phi - \alpha \|^2$ (as in Definition
\ref{def:Morsestratification}).
\end{theorem}

The key steps in the proof of this theorem are: Lemma \ref{lem:cangetclose}, which shows that the closure of every $G^\C$-orbit in $\Rep(Q, {\bf v})_{\bf v^*}$ intersects the critical set $C_{\bf v^*}$, Lemma \ref{lem:neighbourhood-subset}, which shows that the two stratifications co-incide on a neighbourhood of the critical set, and Lemma \ref{lem:intersection-invariant}, which extends this result to any $G^\C$-orbit in $\Rep(Q, {\bf v})_{\bf v^*}$ that intersects this neighbourhood.

Firstly, we prove some background results needed for Lemmas \ref{lem:neighbourhood-subset} and \ref{lem:intersection-invariant}.

\begin{definition}\label{def:function-type}
Given a Harder-Narasimhan type ${\bf v^*}$ with corresponding slope vector $\nu$, let $\Lambda_{\bf v^*}$ and $\Lambda_\nu$ both denote the $\rank(Q, {\bf v}) \times \rank(Q, {\bf v})$ matrix with diagonal entries corresponding to the elements of $ \nu$. Define $f({\bf v^*}) : = \| \Lambda_{\bf v^*} \|^2$ and $f(\nu) := \| \Lambda_{\bf v^*} \|^2$. 
\end{definition}

\begin{lemma}\label{lem:H-N-infimum}
Let $A \in \Rep(Q, {\bf v})_{\bf v^*}$. Then $f(A) \geq f({\bf v^*})$.
\end{lemma}

\begin{proof}
Let $L$ be the length of the Harder-Narasimhan filtration. With respect to this filtration, the representation has the following form for each $a \in E$.
\begin{equation} \label{eqn:H-N-decomposition}
A_a = \left(
    \begin{matrix} A_a^1 & \eta_a^{1,2} & \eta_a^{1,3} & \cdots &
      \eta_a^{1,L} \\ 0 & A_a^2 & \eta_a^{2,3} & \cdots & \eta_a^{2,L}
      \\ \vdots & \ddots & \ddots & \ddots & \vdots \\ \vdots & \ddots
      & \ddots & \ddots & \eta_a^{L-1,L} \\ 0 & \cdots & \cdots & 0 &
      A_a^L \end{matrix} \right). 
\end{equation} 
The moment map $\Phi(A) - \alpha = i \sum_{a \in E} \left[ A_a, A_a^*\right] - \alpha$ can then be expressed in terms of the filtration, with block-diagonal terms 
\begin{equation}
  \beta_j = i \sum_{a \in E} \left( [A_a^j, (A_a^j)^* ] + \sum_{k > j}
    \eta_a^{j,k} ( \eta_a^{j,k} )^* - \sum_{k < j} (\eta_a^{k,j})^*
    \eta_a^{k,j} \right) - \alpha_j. \end{equation} 
Therefore we have
\begin{equation} \label{eqn:moment-map-terms}
- i \tr \beta_j = \sum_{a \in E} \left( \sum_{k > j}  \tr \eta_a^{j,k} ( \eta_a^{j,k} )^* - \sum_{k < j} \tr (\eta_a^{k,j})^* \eta_a^{k,j} \right) + \deg_\alpha(Q, {\bf v_j}).
\end{equation} 
Taking the sum over $j$ from $1$ to $\ell$, a computation
shows that for all $\ell \leq L$ we have 
\begin{equation} \label{eqn:degree-lower-bound}
-i \sum_{j=1}^\ell \tr \beta_j \geq \sum_{j=1}^\ell \deg_\alpha(Q, {\bf v_j}). 
 \end{equation} 
Induction on $j$ in \eqref{eqn:moment-map-terms} shows that equality for all $\ell$ in \eqref{eqn:degree-lower-bound} occurs if and only if $\eta_a^{j,k} = 0$ for all $a \in E$ and $j < k$. 

Let $\tilde{\nu} = ( \tilde{\nu}_1, \ldots, \tilde{\nu}_{\rank(Q, {\bf v})})$, where $\tilde{\nu}_k = -i \frac{1}{\rank(Q, {\bf v}_{j})} \tr \beta_j$ if $\sum_{\ell = 1}^{j-1} \rank(Q, {\bf v}_{\ell}) < k \leq \sum_{\ell=1}^j \rank(Q, {\bf v}_{\ell})$. Then the results of \cite[Section 12]{AtiBot83} for the norm-square function $\| \cdot \|^2 : \mathfrak{u}(\Vect(Q, {\bf v})) \rightarrow \R$ (which is a convex invariant function) show that
\begin{equation} 
f(A) = \| \Phi(A) - \alpha \|^2  \geq f(\tilde{\nu}) \geq f(\nu),
\end{equation}
where the last inequality follows from \eqref{eqn:moment-map-terms}.
\end{proof}

\begin{lemma}\label{lem:cangetclose}
Given a Harder-Narasimhan type ${\bf v^*}$ and any $A \in \Rep(Q, {\bf v})_{\bf v^*}$, there exists $A_\infty \in C_{\bf v^*}$ such that for all $\varepsilon > 0$ there exists $g \in G_\C$ such that $\| g \cdot A - A_\infty \| < \varepsilon$.
\end{lemma}

\begin{proof}

As for the previous proof, decompose each $A_a$ in terms of the Harder-Narasimhan filtration as in \eqref{eqn:H-N-decomposition}. With respect to this filtration, denote the vector spaces for the successive quotient representations by $\Vect(Q, {\bf v_j})$ for $j = 1, \ldots, L$, and let
\begin{equation}\label{eqn:quotient-parameter}
\alpha_j  = \left. \alpha \right|_{\Vect(Q, {\bf v}_j)} - \frac{1}{\rank(Q, {\bf v}_j)} \left. \tr \alpha \right|_{\Vect(Q, {\bf v}_j)} \cdot \id
\end{equation}
denote the trace-free stability parameters associated to the subspaces $\Vect(Q, {\bf v}_j)$. Also let $\Lambda_{\nu_j} = \mu_\alpha(Q, {\bf v}_j) \cdot \id$ be the associated diagonal matrices, as defined in Definition \ref{def:function-type}. Note that $\alpha_j - \left. \alpha \right|_{\Vect(Q, {\bf v}_j)} = - \frac{1}{\rank(Q, {\bf v}_j)} \left. \tr \alpha \right|_{\Vect(Q, {\bf v}_j)} \cdot \id = i \mu_\alpha(Q, {\bf v}_j) \cdot \id = i \Lambda_{\nu_j}$. Since each $A^j$ is $\alpha$-semistable, then Lemma \ref{lem:KempfNess} applied to each subrepresentation shows that for any $\delta > 0$ there exists 
\begin{equation*}
\tilde{g} = \left( \begin{matrix} \tilde{g}_1 & 0 & \cdots & 0 \\ 0 & \tilde{g}_2 & \cdots & 0 \\ \vdots & \vdots & \ddots & \vdots \\ 0 & 0 & \cdots & \tilde{g}_L \end{matrix} \right) \in G_\C,
\end{equation*}
such that 
\begin{equation*}
\sum_{j = 1}^L \left\| \Phi( \tilde{g_j} \cdot A^j ) - \alpha - i\Lambda_{\nu_j} \right\|^2 = \sum_{j=1}^L \left\| \Phi( \tilde{g}_j \cdot A^j) - \alpha_j \right\|^2 < \frac{1}{2} \delta .
\end{equation*}
(Recall that $\alpha_j$ is the trace-free stability parameter for the representation $A^j \in \Rep(Q, {\bf v_j})$ used in Lemma \ref{lem:KempfNess}.) In particular, as a result of Lemma \ref{lem:KempfNess} and the description of the critical sets in Proposition \ref{prop:momentmapdetermined}, the block diagonal part $A^{gr}$ of this representation $\tilde{g} \cdot A$ (the graded object of the Harder-Narasimhan filtration) is close to a critical point of $f$, i.e. there exists $A_\infty \in C_{\bf v^*}$ such that $\| A^{gr} - A_\infty \| < \frac{1}{2} \varepsilon$. Note that (up to $G$-equivalence) $A_\infty$ is determined by $A$, since it is determined by  Lemma \ref{lem:KempfNess} and the graded object of the H-N filtration of $A$. Given any $t \in \R^*$, apply a $G_\C$-transformation of the form 
\begin{equation*}
\tilde{h}_t = \left( \begin{matrix} t^L & 0 & \cdots & 0 \\ 0 & t^{L-1} & \cdots & 0 \\ \vdots & \vdots & \ddots & \vdots \\ 0 & 0 & \cdots & t \end{matrix} \right)
\end{equation*}
 to $\tilde{g} \cdot A$ to obtain
\begin{equation}
\tilde{h}_t \cdot \tilde{g} \cdot A = \left( \begin{matrix} \tilde{g}_1 A_a^1 \tilde{g}_1^{-1} & t \, \tilde{g}_1 \eta_a^{1,2} \tilde{g}_2^{-1} & t^2 \, \tilde{g}_1 \eta_a^{1,3} \tilde{g}_3^{-1} & \cdots & t^{L-1} \, \tilde{g}_1 \eta_a^{1,L} \tilde{g}_L^{-1} \\ 0 & \tilde{g}_2 A_a^2 \tilde{g}_2^{-1}  & t \, \tilde{g}_2 \eta_a^{2,3} \tilde{g}_3^{-1} & \cdots & t^{L-2} \, \tilde{g}_2 \eta_a^{2,L} \tilde{g}_L^{-1} \\ \vdots & \ddots & \ddots & \ddots & \vdots \\ \vdots & \ddots & \ddots & \ddots & t \, \tilde{g}_{L-1} \eta_a^{L-1,L} \tilde{g}_L^{-1} \\ 0   & \cdots & \cdots & 0 & \tilde{g}_L A_a^L \tilde{g}_L^{-1} \end{matrix} \right).
\end{equation}
For $t > 0$ small enough, this has the effect of scaling the extension classes $\eta_a^{j,k}$ so that  $t^{k-j} \sum_a \sum_{j,k} \| \tilde{g}_j \eta_a^{j,k} \tilde{g}_k^{-1} \| < \frac{1}{2} \varepsilon$. Combining this with the previous estimate for $\| A^{gr} - A_\infty \|$ shows that $\| \tilde{h}_t \cdot \tilde{g} \cdot A - A_\infty \| < \varepsilon$, as required.
\end{proof}

The next lemma is a restatement of Lemma \ref{lem:close-convergence} in terms of representations of quivers.

\begin{lemma}\label{lem:nearby-limit}
Let $A_\infty \in C_{\bf v^*}$ for some Harder-Narasimhan type ${\bf v^*}$. Then for any $\varepsilon_2 > 0$ there exists $\varepsilon_1 > 0$ such that for any $A$ satisfying $\| A - A_\infty \| < \varepsilon_1$ and $f\left(\lim_{t \rightarrow \infty} \gamma(A,t) \right) = f(A_\infty)$, then
$$
\left\| \lim_{t \rightarrow \infty} \gamma(A,t) - A_\infty \right\| < \frac{1}{2} \varepsilon_2  .
$$
\end{lemma}

\begin{lemma}\label{lem:critical-distance}
Let $A_1 \in C_{\bf v^*}$ and $A_2 \in C_{\bf w^*}$ with ${\bf v^*} \neq {\bf w^*}$. Then there exists $C>0$ (depending only on ${\bf v^*}$ and ${\bf w^*}$) such that $\| \Phi(A_1) - \Phi(A_2) \| \geq C$, and
\begin{equation}
\left\| A_1 - A_2 \right\| \geq \frac{C}{2(\| A_1 \| + \| A_2 \|)}. 
\end{equation}
\end{lemma}

\begin{proof}
Recall that the value of the moment map at a critical point is determined by Proposition \ref{prop:momentmapdetermined}. Since the respective splitting types of $A_1$ and $A_2$ are different, then there is a lower bound $C$ on the magnitude of the difference between the values of $\Phi(A_1)- \alpha$ and $\Phi(A_2) - \alpha$, depending only on ${\bf v^*}$ and ${\bf w^*}$. We also have the following estimate on this difference
\begin{align*}
C & \leq \left\| \sum_{a \in E} \left[ (A_1)_a, (A_1)_a^* \right] - \left[ (A_2)_a, (A_2)_a^* \right] \right\| \\
 & = \left\| \sum_{a \in E} (A_1)_a (A_1)_a^* - (A_1)_a^* (A_1)_a - (A_2)_a (A_2)_a^* + (A_2)_a^* (A_2)_a \right\| \\
 & \leq \left\|  \sum_{a \in E} \left( (A_1)_a - (A_2)_a \right) (A_1)_a^*+ (A_2)_a \left( (A_1)_a^* - (A_2)_a^* \right) \right\| \\
 & \quad + \left\| \sum_{a \in E} \left( (A_2)_a^* - (A_1)_a^* \right) (A_1)_a + (A_2)_a^* \left( (A_2)_a - (A_1)_a \right) \right\| \\
& \leq 2 \left( \max_{a \in E} \| (A_1)_a \| + \max_{a \in E} \| (A_2)_a \| \right) \| A_1 - A_2 \| \\
 & \leq 2 \left( \| A_1 \| + \| A_2 || \right) \| A_1 - A_2 \| .
\end{align*}
Therefore, $\| A_1 - A_2 \| \geq \frac{C}{2(\| A_1 \| + \| A_2 \|)}$, as required.
\end{proof}

The next lemma is the key result needed to prove Theorem \ref{theorem:equivstratification}. It says that each Morse stratum $S_{\bf v^*}$ co-incides with the Harder-Narasimhan stratum $\Rep(Q, {\bf v})_{\bf v^*}$ on a neighbourhood of the critical set $C_{\bf v^*}$. The estimates from Lemma \ref{lem:nearby-limit} and Lemma \ref{lem:critical-distance} are in turn an important part of the proof of Lemma \ref{lem:neighbourhood-subset}.

\begin{lemma}\label{lem:neighbourhood-subset}
There exists a neighbourhood $V_{\bf v^*}$ of $C_{\bf v^*}$ such that $V_{\bf v^*} \cap \Rep(Q, {\bf v})_{\bf v^*} \subset S_{\bf v^*}$. 
\end{lemma}

\begin{proof}
The problem reduces to showing that there is a neighbourhood of every critical point on which the lemma is true, then one can take the union of these neighbourhoods. Therefore it is sufficient to show that for every $A_\infty \in C_{\bf v^*}$ there exists $\varepsilon > 0$ such that $\| A - A_\infty \| < \varepsilon$ and $A \in \Rep(Q, {\bf v})_{\bf v^*}$ implies that $A \in S_{\bf v^*}$.

Fix $A_\infty \in C_{\bf v^*}$. By Lemma \ref{lem:critical-distance}, choose $\delta > 0$ such that $\| \tilde{B} - A_\infty \| \geq \delta$ for all critical points $\tilde{B}$ with Harder-Narasimhan type ${\bf w^*} \neq {\bf v^*}$. Since, by Proposition \ref{prop:momentmapdetermined}, there are a finite number of critical values of $f$, and the function $f$ is continuous, then there exists $\varepsilon > 0$ such that $\| A - A_\infty \| < \varepsilon$ implies that $\lim_{t \rightarrow \infty} f(\gamma(A, t)) \leq f(C_{\bf v^*})$. Lemma \ref{lem:H-N-infimum} then implies that the hypotheses of Lemma \ref{lem:nearby-limit} are satisfied, and so (after shrinking $\varepsilon$ if necessary) if $\| A- A_\infty \| < \varepsilon$ and $A \in \Rep(Q, {\bf v})_{\bf v^*}$, then
$$
\left\|  \lim_{t \rightarrow \infty} \gamma(A, t) - A_\infty \right\| < \frac{1}{2} \delta .
$$
Therefore $\lim_{t \rightarrow \infty} \gamma(A, t) \in C_{\bf v^*}$, and so $A \in S_{\bf v^*}$. 
\end{proof}

\begin{lemma}\label{lem:intersection-invariant}
$\Rep(Q, {\bf v})_{\bf v^*} \cap S_{\bf v^*}$ is $G_\C$-invariant.
\end{lemma}

\begin{proof}
Given $A \in \Rep(Q, {\bf v})_{\bf v^*} \cap S_{\bf v^*}$, let
\begin{equation*}
\mathcal{G}_A = \left\{ g \in G_\C \, \mid \, g \cdot A \in \Rep(Q, {\bf v})_{\bf v^*} \cap S_{\bf v^*} \right\} .
\end{equation*}
Since $\mathcal{G}_A$ is non-empty and $G_\C$ is connected, then the result will follow if we can show that $\mathcal{G}_A$ is open and closed in $G_\C$. The continuity of the $G_\C$ action, together with Theorem \ref{thm:distancedecreasing}, Lemma \ref{lem:cangetclose}, and Lemma \ref{lem:neighbourhood-subset}, shows that $\mathcal{G}_A$ is open. 

To see that $\mathcal{G}_A$ is closed, let $\{ g_k \}$ be a sequence of points in $\mathcal{G}_A$ converging to some $g_\infty \in G_\C$, and let $A_k := g_k \cdot A$, $A_\infty := g_\infty \cdot A$. Since $\Rep(Q, {\bf v})_{\bf v^*}$ is $G_\C$-invariant, then $A_\infty \in \Rep(Q, {\bf v})_{\bf v^*}$, and it only remains to show that $A_\infty \in S_{\bf v^*}$. 

Let $\gamma(A_\infty, t)$ be the gradient flow of $f$ with initial conditions $A_\infty$, let $\gamma(A_\infty, \infty)$ denote the limiting critical point, and ${\bf w^*}$ the Harder-Narasimhan type of $\gamma(A_\infty, \infty)$. Suppose for contradiction that ${\bf w^*} \neq {\bf v^*}$. Lemma \ref{lem:critical-distance} shows that there is a constant $C$ such that any critical point $B$ with Harder-Narasimhan type ${\bf v^*} \neq {\bf w^*}$ must satisfy $\| \Phi(B) - \Phi \left(\gamma(A_\infty,\infty) \right) \| \geq C$. Now choose $\delta$ to obtain a bound of $\frac{1}{3} C$ in Proposition \ref{prop:momentmapclose}, and choose $k$ such that $\sigma \left(g_\infty g_k^{-1} (g_\infty g_k^{-1})^* \right)  < \delta$. Let $\tilde{g}(t)$ denote the element of $G_\C$ connecting $\gamma(A_k, t)$ and $\gamma(A_\infty, t)$ along the flow (note that $\tilde{g}(0) = g_\infty g_k^{-1}$), and recall that the distance decreasing formula of Theorem \ref{thm:distancedecreasing} shows that $\sigma(\tilde{g}(t) \tilde{g}(t)^*) < \delta$ along the flow. Convergence of the flow and continuity of $\Phi$ shows that there exists $t$ such that both $\| \Phi \left( \gamma(A_\infty, t) \right) - \Phi \left( \gamma(A_\infty, \infty) \right)) \| < \frac{1}{3} C$ and $\| \Phi \left( \gamma(A_k,t) \right) - \Phi \left( \gamma(A_k, \infty) \right) \| < \frac{1}{3} C$ (note that since the norm is $G$-invariant and $\Phi$ is $G$-equivariant then these estimates are $G$-invariant). Proposition \ref{prop:momentmapclose} then shows that $\| \Phi( g_u \cdot \gamma(A_k, t)) - \Phi( \gamma(A_\infty, t)) \| < \frac{1}{3} C$ for some $g_u \in G$, and combining all of these estimates gives
\begin{equation*}
\| \Phi(g_u \cdot \gamma(A_k, \infty)) - \Phi( \gamma(A_\infty, \infty)) \| < C,
\end{equation*} 
which contradicts the choice  of $C$, since $A_k \in S_{\bf v^*}$ and $A_\infty \in S_{\bf w^*}$ by assumption, and the critical sets are $G$-invariant. Therefore $A_\infty \in S_{\bf v^*}$, which completes the proof that $\mathcal{G}_A$ is closed in $G_\C$. 
\end{proof}

With these results in hand we are now ready to prove the main theorem of this section, that the Harder-Narasimhan stratification co-incides with the Morse stratification.

\begin{proof}[Proof of Theorem \ref{theorem:equivstratification}]
First recall that the stratum $\Rep(Q, {\bf v})_{\bf v^*}$ is $G_\C$-invariant, since the Harder-Narasimhan type is an isomorphism invariant. Lemma \ref{lem:cangetclose} and Lemma \ref{lem:neighbourhood-subset} together show that each $G_\C$-orbit in $\Rep(Q, {\bf v})_{\bf v^*}$ has non-trivial intersection with $S_{\bf v^*}$. Then Lemma \ref{lem:intersection-invariant} shows that $\Rep(Q, {\bf v})_{\bf v^*} \subseteq S_{\bf v^*}$. Since $\left\{ \Rep(Q, {\bf v})_{\bf v^*} \right\}$ and $\left\{ S_{\bf v^*} \right\}$ both define stratifications of $\Rep(Q, {\bf v})$, then $\Rep(Q, {\bf v})_{\bf v^*} = S_{\bf v^*}$.
\end{proof}

As a corollary of Theorem \ref{theorem:equivstratification} we can describe the splitting of the representation at a critical point in terms of the Harder-Narasimhan filtration.

\begin{corollary}\label{cor:critical-infimum}
Let $A_\infty \in C_{\bf v^*}$. Then $f(A_\infty) = \inf \{ f(A) \, : \, A \in \Rep(Q, {\bf v})_{\bf v^*} \}$ and the splitting of $A_\infty$ into orthogonal subrepresentations as in \eqref{eq:A-decomp-by-beta} corresponds to the graded object of the Harder-Narasimhan filtration of $A_\infty$.
\end{corollary}

\begin{proof}
Theorem \ref{theorem:equivstratification} and Lemma \ref{lem:H-N-infimum} together show that 
$$
f(A_\infty) = \inf \{ f(A) \, : \, A \in \Rep(Q, {\bf v})_{\bf v^*} \}.
$$
Proposition \ref{prop:momentmapdetermined} shows that each subrepresentation $A_\lambda$ satisfies
\begin{equation*}
\Phi_{\lambda}(A_{\lambda}) = \left. \alpha \right|_{V_{\lambda}}  + i \mu_{\alpha}(Q, {\mathbf v}_{\lambda}) \cdot \id_{V_{\lambda}}, 
\end{equation*}
and so by \cite[Theorem 6.1]{Kin94} each $A_\lambda$ is $\beta_\lambda$-semistable with respect to the trace-free stability parameter $\beta_\lambda = \left. \alpha \right|_{V_{\lambda}}  + i \mu_{\alpha}(Q, {\mathbf v}_{\lambda}) \cdot \id_{V_{\lambda}}$ on $\Rep(Q, {\bf v_\lambda})$. Since the H-N type and the critical type of $A_\infty$ co-incide by Theorem \ref{theorem:equivstratification}, then the uniqueness of the Harder-Narasimhan filtration shows that the decomposition \eqref{eq:A-decomp-by-beta} determines the Harder-Narasimhan filtration.
\end{proof}

We next show that the sub-representations in the
Harder-Narasimhan filtration converge along the gradient flow of $\|
\Phi(A) - \alpha \|^2$. Given a representation $A \in \Rep(Q, {\bf
v})$ with H-N filtration
$$0 \subset A_1 \subset \cdots \subset A_L = A,$$ define $\pi^{(i)}$
to be the orthogonal projection onto the subspace of $\Vect(Q, {\bf
v})$ associated to $A_i \in \Rep(Q, {\bf v_i})$. The induced
representation on the image of $\pi^{(i)}$ is then $A_i = A \circ
\pi^{(i)}$.  
Using these projections we may also denote the H-N
filtration of $A$ by $\{ \pi^{(i)} \}_{i=1}^\ell$. For a solution
$\gamma(A_0, t) = g(t) \cdot A_0$ to the gradient flow equation
\eqref{eq:gradient-flow-def}, let the corresponding projection
be $\pi_t^{(i)}$, i.e. the orthogonal projection onto the vector space
$g(t) \pi_0^{(i)} \Vect(Q, {\bf v})$.

\begin{proposition}\label{prop:filtrationlimit}
Let $\{ \pi_t^{(i)} \}$ be the H-N filtration of a solution $A(t)$ to the gradient flow equations \eqref{eq:gradient-flow-def}, and let $\{ \pi_\infty^{(i)} \}$ be the H-N filtration of the limit $A_\infty$. Then there exists a subsequence $t_j$ such that $\pi_{t_j}^{(i)} \rightarrow \pi_\infty^{(i)}$ for all $i$.  
\end{proposition}

\begin{proof}

For each $i$, $\pi_t^{(i)}$ (being projections) are uniformly bounded
operators, so there exists a
subsequence $\{ t_j \}$ such that $\lim_{j \rightarrow \infty}
\pi_{t_j}^{(i)} \rightarrow \tilde{\pi}_\infty^{(i)}$ for some
$\tilde{\pi}_\infty^{(i)}$. Hence the goal is to show that
$\tilde{\pi}_\infty^{(i)} = \pi_\infty^{(i)}$ for all $i$.

Firstly note that on each vector space $V_j$ the projections
$\pi_t^{(i)}$ and $\tilde{\pi}_\infty^{(i)}$ have the same rank (since
$\pi_t^{(i)}$ is the orthogonal projection onto the space $g(t)
\pi_0^{(i)} \Vect(Q, {\bf v})$, and $\pi_{t_j}^{(i)} \rightarrow
\tilde{\pi}_\infty^{(i)}$ as projections). Theorem
\ref{theorem:equivstratification} shows that the type of the
Harder-Narasimhan filtration is preserved in the limit of the gradient
flow, and so the ranks of the maximal semistable sub-representations
are also preserved. Therefore $\rank(\pi_\infty^{(i)}) =
\rank(\pi_{t_j}^{(i)}) = \rank(\tilde{\pi}_\infty^{(i)})$ on each
vector space $V_j$.  Since the ranks are the same on each vector space
in $\Vect(Q, {\bf v})$, by the definition of $\alpha$-degree,
$\deg_\alpha(\pi_\infty^{(i)}) =
\deg_\alpha(\tilde{\pi}_\infty^{(i)})$. Thus the degree-rank ratios
of the sub-representations corresponding to the projections
$\pi_\infty^{(i)}$ and $\tilde{\pi}_\infty^{(i)}$ are the same.

For the case $i=1$, the fact that the maximal
$\alpha$-semistable sub-representation is unique (Proposition
\ref{prop:maximalsemistable}) implies that $\pi_\infty^{(1)} =
\tilde{\pi}_\infty^{(1)}$. Now we proceed by induction: Fix $k$
and assume that $\tilde{\pi}_\infty^{(i)} = \pi_\infty^{(i)}$ for all
$i < k$. Let $\tilde{A}_\infty^{(i)}$ be the sub-representation of $A$
corresponding to the projection $\tilde{\pi}_\infty^{(i)}$, and let
$A_\infty^{(i)}$ be the sub-representation of $A$ corresponding to the
projection $\pi_\infty^{(i)}$.  Then $\tilde{A}_\infty^{(k)} /
\tilde{A}_\infty^{(k-1)}$ has the same $\alpha$-degree and rank as
$A_\infty^{(k)} / A_\infty^{(k-1)}$, which is the maximal semistable
sub-representation of $A_\infty / A_\infty^{(k-1)}$. Again, uniqueness
of the maximal semistable sub-representation implies
$\tilde{A}_\infty^{(k)} / \tilde{A}_\infty^{(k-1)} = A_\infty^{(k)} /
A_\infty^{(k-1)}$. Together with the inductive hypothesis this
gives us $\tilde{\pi}_\infty^{(i)} = \pi_\infty^{(i)}$ for all $i \leq
k$.
\end{proof}

\section{An algebraic description of the limit of the flow}\label{sec:graded}

The results of Sections~\ref{sec:gradientflow} and~\ref{sec:HN} show that the
negative gradient flow of $f = \| \Phi - \alpha \|^2$ with initial
condition \(A \in \Rep(Q, {\bf v})\) converges to a critical point
$A_{\infty}$ of $f$, and that the Harder-Narasimhan type of this limit
point $A_{\infty}$ is the same as the Harder-Narasimhan type of the
initial condition $A$. In this section we provide a more precise
description of the limit of the flow in terms of the
Harder-Narasimhan-Jordan-H\"older filtration (defined below) of the initial condition
\(A \in \Rep(Q, {\bf v}).\) Our main result, Theorem~\ref{thm:convergencetogradedobject},
should be viewed as a quiver analogue of the theorem of 
Daskalopoulos and Wentworth for the case of holomorphic bundles over K\"ahler surfaces
\cite{DasWen04}. 

We first recall the definition of the
Harder-Narasimhan-Jordan-H\"older filtration, which is a refinement of
the Harder-Narasimhan filtration using stable subrepresentations. As
before, let \(Q = ({\mathcal I}, E)\) be a finite quiver, \({\bf v}
\in \Z^{\mathcal I}_{\geq 0}\) a choice of dimension vector, and
\(\alpha \in (i\R)^{\mathcal I}\) a stability parameter.

\begin{definition}
Let $Q=({\mathcal I}, E)$ be a quiver with specified dimension vector
\({\bf v} \in \Z^{\mathcal I}_{\geq 0},\) and let $\alpha$ be a
stability parameter. 
Let $A \in \Rep(Q, {\bf v})^{\alpha-ss}$ be an $\alpha$-semistable representation. A filtration of $A$ with induced subrepresentations of $A$ 
\begin{equation}
0 = A_0 \subset A_1 \subset \cdots \subset A_m = A,
\end{equation}
where $A_j \in \Rep(Q, {\bf v_j})$, is called a \emph{Jordan-H\"older filtration} if 
\begin{itemize}
\item for each \(k, 1 \leq k \leq m,\) the induced quotient representation 
$$
A_k / A_{k-1} \in \Rep \left( Q, {\bf v_k - v_{k-1}} \right)
$$
is $\alpha$-stable, and 
\item $\mu_\alpha \left( Q, {\bf v_k - v_{k-1}} \right) = \mu_\alpha(Q, {\bf v})$ for each $k$.
\end{itemize} 
\end{definition}

Given a Jordan-H\"older filtration of $A$ as above, we define the {\em associated graded object} of the filtration as the direct sum 
\[
\Gr^{JH}(A; Q, {\bf v}) := \bigoplus_{k=1}^{m} A_{k}/A_{k-1},
\]
which by construction is also a representation of $Q$ with dimension vector ${\bf v}$. 
Existence of the Jordan-H\"older filtration follows for general reasons (see \cite{Seshadri65} and \cite[pp521-522]{Kin94}), and the associated graded objects are uniquely determined up to isomorphism in $\Rep(Q, {\bf v})$ by the isomorphism type of $A$.

Given any representation \(A \in \Rep(Q, {\bf v})\), we may now
combine the H-N filtration of $A$ by $\alpha$-semistable
subrepresentations with a Jordan-H\"older filtration of each
$\alpha$-semistable piece in the Harder-Narasimhan filtration, thus obtaining a
double filtration called a {\em Harder-Narasimhan-Jordan-H\"older
(H-N-J-H) filtration} of $A$. This is the quiver analogue of the
Harder-Narasimhan-Seshadri filtration for holomorphic bundles (see
e.g. \cite[Proposition 2.6]{DasWen04}). 

We first set some notation for double filtrations. We will say a
collection $\{A_{j,k}\}$ is a \emph{double filtration} of $A$ when
\begin{equation}\label{eq:first-filtration}
0 = A_{0,0} \subseteq A_{1,0} \subseteq \cdots \subseteq A_{L,0} = A
\end{equation}
is a filtration of $A$ by subrepresentations, and furthermore, for each $j$ with \(1 \leq j \leq L,\) we have 
\begin{equation}\label{eq:second-filtrations}
A_{j,0} \subseteq A_{j,1} \subseteq A_{j,2} \subseteq \cdots \subseteq A_{j,m_{j}} = A_{j+1,0}
\end{equation}
a sequence of intermediate subrepresentations. Notating by
\(\tilde{A}_{j,k}\) the quotient \(A_{j,k}/A_{j,0},\) the
sequence~\eqref{eq:second-filtrations} then immediately gives rise to
an induced filtration (again by subrepresentations)
\begin{equation}\label{eq:filtered-quotient} 
0 = \tilde{A}_{j,0} \subseteq \tilde{A}_{j,1} \subseteq \cdots
\subseteq \tilde{A}_{j+1,0}
\end{equation}
of the quotient representation \(\tilde{A}_{j,m_{j}} = A_{j, m_{j}}/A_{j,0} = A_{j+1,0}/A_{j,0}.\)

\begin{proposition}(cf. \cite[Proposition 2.6]{DasWen04})
Let \(Q = ({\mathcal I}, E)\) be a quiver with specified dimension
vector \({\bf v} \in \Z^{\mathcal I}_{\geq 0},\)
and \(\alpha \in Z(\g)\) a central parameter. Let $A
\in \Rep(Q, {\bf v})$. Then there exists a double filtration
$\{A_{j,k}\}$ of $A$ such that the
filtration~\eqref{eq:first-filtration} is the Harder-Narasimhan
filtration of $A$, and for each $j, 1 \leq j \leq L,$ the
filtration~\eqref{eq:second-filtrations} is a Jordan-H\"older
filtration of the quotient $\tilde{A}_{j+1, 0} \in \Rep(Q, {\bf
v_{j+1, 0} - v_{j, 0}})$.  Moreover, the isomorphism class in $\Rep(Q,
{\bf v})$ of the associated graded object
\begin{equation}
\Gr^{HNJH}(A, Q, {\bf v}) := \bigoplus_{j=1}^L \bigoplus_{k=1}^{m_j} A_{j, k} / A_{j, k-1}
\end{equation}
is uniquely determined  by the isomorphism class of $A \in \Rep(Q, {\bf v})$.
\end{proposition}

We may now state the main theorem of this section.  The point is that
it is precisely the graded object of the H-N-J-H filtration of the
initial condition which determines the isomorphism type of the limit
under the negative gradient flow.

\begin{theorem}\label{thm:convergencetogradedobject}
Let $Q = ({\mathcal I}, E)$ be a quiver with specified dimension
vector \({\bf v} \in
\Z^{\mathcal I}_{\geq 0},\) $\Rep(Q, {\bf v})$ its
associated representation space, and \(\Phi: \Rep(Q, {\bf v}) \to
\g^{*} \cong \g \cong \prod_{\ell \in {\mathcal I}} \u(V_{\ell})\) a
moment map for the standard Hamiltonian action of $G = \prod_{\ell \in
{\mathcal I}} U(V_{\ell})$ on $\Rep(Q, {\bf v})$.  Let $A_0 \in
\Rep(Q, {\bf v})$, and let $A_\infty = \lim_{t \rightarrow \infty}
\gamma(A_0, t)$ be its limit under the negative gradient flow of $\|
\Phi - \alpha \|^2$. Then $A_\infty$ is isomorphic, as a representation
of the quiver $Q$, to the associated graded object of the H-N-J-H
filtration of the initial condition $A_0$.
\end{theorem}

Firstly, recall that Proposition \ref{prop:filtrationlimit} already shows that
the vector spaces that define the Harder-Narasimhan filtration
converge in the limit of the gradient flow. To prove that the gradient
flow converges to the graded object of the H-N-J-H filtration we need
to show that the corresponding sub-representations also 
converge. The next proposition is a key step in the argument; it shows
that $\alpha$-stability and $\alpha$-semistability conditions on two
representations, together with knowledge of the relationship between
their $\alpha$-slopes, can place strong restrictions on homomorphisms
between them. Again, this is a quiver analogue of a similar statement
for holomorphic bundles (\cite[Proposition V.7.11]{Kob87}), a
well-known corollary of which is that a semistable bundle with
negative degree has no holomorphic sections. (For quivers, see \cite[Lemma 2.3]{Rei03} for the case where $\mu_{\alpha}(Q, {\bf v}_{1}) > \mu_{\alpha}(Q, {\bf v}_{2})$.)

\begin{proposition}\label{prop:stableisomorphism}
Let \(Q = ({\mathcal I}, E)\) be a finite quiver, \({\bf v_{1}}, {\bf v_{2}} \in \Z^{\mathcal I}_{\geq 0}\) be dimension vectors, and \(\alpha\) a central parameter. Suppose \(A_{1} \in \Rep(Q, {\bf v}_{1})\) with associated vector spaces $\{V_{\ell}\}$, \(A_{2} \in \Rep(Q, {\bf v}_{2})\) with associated vector spaces $\{V'_{\ell}\}$, and \(f = \{f_{\ell}: V_{\ell} \to V'_{\ell}\}_{\ell \in {\mathcal I}}\) is a homomorphism of quivers from $A$ to $A'$. 

\begin{itemize} 
\item If \(\mu_{\alpha}(Q, {\bf v}_{1}) = \mu_{\alpha}(Q, {\bf v}_{2}),\) $A_{1}$ is $\alpha$-stable, and $A_{2}$ is $\alpha$-semistable, then $f$ is either zero or injective. 
\item If \(\mu_{\alpha}(Q, {\bf v}_{1}) = \mu_{\alpha}(Q, {\bf v}_{2}),\) $A_{1}$ is $\alpha$-semistable, and $A_{2}$ is $\alpha$-stable, then $f$ is either zero or surjective. 
\item If \(\mu_{\alpha}(Q, {\bf v}_{1}) = \mu_{\alpha}(Q, {\bf v}_{2})\) and both $A_{1}$ and $A_{2}$ are $\alpha$-stable, then $f$ is either zero or an isomorphism. 
\item If \(\mu_{\alpha}(Q, {\bf v}_{1}) > \mu_{\alpha}(Q, {\bf v}_{2})\) and both $A_{1}$ and $A_{2}$ are $\alpha$-semistable, then $f$ is zero. 

\end{itemize} 
\end{proposition}

\begin{proof}
Let $K = \ker(f)$ and $L = \im(f)$, where by $\ker(f)$ we mean the direct sum $\oplus_{\ell \in {\mathcal I}} \ker(f_{\ell})$, and similarly for $\im(f)$.  If $L = \{0\}$ then there is nothing to prove, so we assume that
$L \neq \{0\}$ and hence also $K \neq \Vect(Q, {\bf v_1})$. 
Since $f$ is a homomorphism of representations of quivers, it is straightforward to see that $K$ is a subrepresentation of $A_{1}$, and $L$ is a subrepresentation of $A_{2}$. Let \(A_{1} \vert_{K}\) and \(A_{2} \vert_{L}\) denote the restrictions of $A_{1}$ and $A_{2}$ to $K$ and $L$ respectively, with associated dimension vectors \({\bf v}_{K}, {\bf v}_{L}.\) Then we have an exact sequence of representations of quivers 
\begin{equation}\label{eq:A1-A2}
0 \rightarrow A_{1} \vert_{K} \rightarrow A_{1} \rightarrow A_{2} \vert_{L} \rightarrow 0
\end{equation}
where the first map is by inclusion and the second induced by $f$. Now assume \(\mu_{\alpha}(Q, {\bf v}_{1}) = \mu_{\alpha}(Q, {\bf v}_{2}),\) $A_{1}$ is $\alpha$-stable, and $A_{2}$ is $\alpha$-semistable. 
Since by assumption \(A_{1} \vert_{K}\) is not equal to $A_{1}$, then if $K \neq 0$ it is a proper subrepresentation of $A_{1}$ and $\alpha$-stability implies \(\mu_{\alpha}(Q, {\bf v}_{K}) < \mu_{\alpha}(Q, {\bf v}_{1}).\) From this we get 
\begin{align}
\begin{split}
\mu_{\alpha}(Q, {\bf v}_{1}) & < \mu_{\alpha}(Q, {\bf v}_{L}) \quad \mbox{by Lemma~\ref{lem:exactsequenceinequality} applied to~\eqref{eq:A1-A2}} \\
& \leq \mu_{\alpha}(Q, {\bf v}_{2}) \quad \mbox{by $\alpha$-semistability of $A_{2}$} \\
& = \mu_{\alpha}(Q, {\bf v}_{1}) \quad \mbox{by assumption},
\end{split}
\end{align}
which is a contradiction. Hence \(K = \{0\}\) and $f$ is injective. This proves the first claim. 

A similar argument shows that if \(\mu_{\alpha}(Q, {\bf v}_{1}) = \mu_{\alpha}(Q, {\bf v}_{2}), A_{1}\) is $\alpha$-semistable, and $A_{2}$ is $\alpha$-stable, then $f$ is either $0$ or surjective. This proves the second claim, and hence also the third. 

Finally suppose that \(\mu_{\alpha}(Q, {\bf v}_{1}) > \mu_{\alpha}(Q, {\bf v}_{2})\) and both $A_{1}, A_{2}$ are $\alpha$-semistable. Suppose for a contradiction that \(L \neq \{0\}\) and \(K \neq \{0\}.\) Then by $\alpha$-semistability we have \(\mu_{\alpha}(Q, {\bf v}_{L}) \leq \mu_{\alpha}(Q, {\bf v}_{2})\) and \(\mu_{\alpha}(Q, {\bf v}_{K}) \leq \mu_{\alpha}(Q, {\bf v}_{1}),\) so \(\mu_{\alpha}(Q, {\bf v}_{1}) \leq \mu_{\alpha}(Q, {\bf v}_{L})\) by Lemma~\ref{lem:exactsequenceinequality} applied to~\eqref{eq:A1-A2}, which gives a contradiction. On the other hand if \(K=\{0\}\) then $A_{1}$ is isomorphic to $A_{2} \vert_{L}$ and \(\mu(Q, {\bf v}_{1}) = \mu(Q, {\bf v}_{L}) \leq \mu(Q, {\bf v}_{2}),\) again a contradiction. Hence \(L = \{0\}\) and the last claim is proved. 
\end{proof}

Let $\gamma(A_0, t)_a$ denote the component of $\gamma(A_0, t) \in
\Rep(Q, {\bf v})$ along the edge $a \in E$. Since the finite-time
gradient flow lies in a $G_\C$ orbit, for $t, T \in \R$ we can
define the isomorphism $g_{t, T} : \Rep(Q, {\bf v}) \rightarrow
\Rep(Q, {\bf v})$ such that $g_{t, T} \, \gamma(A_0, t)_a g_{t,
T}^{-1} = \gamma(A_0, T)_a$ for each $a \in E$. Then we have
%\begin{align}
%\begin{split}
\begin{equation}
g_{t,T} \, \gamma(A_0, t)_a g_{t,T}^{-1}  = \gamma (A_0, T)_a \quad \forall a \in E \, \, 
 \Leftrightarrow \, \, g_{t, T} \, \gamma (A_0, t)_a  = \gamma (A_0, T)_a g_{t, T}.
\end{equation}
%\end{split}
%\end{align}
Similarly, if $\tilde{A}_t$ is a sub-representation of $\gamma(A_0,
t)$ with associated projection $\tilde{\pi}_t$, then the induced
isomorphism $f_{t,T} = g_{t,T} \circ \tilde{\pi}_t$ also satisfies
\begin{equation}\label{eqn:fholomorphic}
f_{t, T} \, \gamma (A_0, t)_a \, \tilde{\pi}_t = \gamma (A_0, T)_a f_{t, T} \, \tilde{\pi}_t.
\end{equation}

\begin{lemma}\label{lem:holomorphiclimit}
For a sub-representation $\tilde{A}_0 \subset A_0$, let
$\tilde{f}_{0,t} = f_{0,t} / \| f_{0,t} \|$, where $f_{0,t}$ is the
map defined above for the gradient flow with initial condition
$A_0$. Then there exists a sequence $t_j$ such that $\lim_{j
\rightarrow \infty} \tilde{f}_{0, t_j} = \tilde{f}_{0,\infty}$ for
some non-zero map $\tilde{f}_{0, \infty}$ satisfying $\tilde{f}_{0,
\infty} \, \gamma(A_0, 0) \, \tilde{\pi}_0 = \gamma(A_0, \infty) \,
f_{0, \infty} \, \tilde{\pi}_0$.
\end{lemma}

\begin{proof}
Since $\| \tilde{f}_{0,t} \| = 1$ for all $t$ then there exists a
sequence $t_j$ and a non-zero map $\tilde{f}_{0,\infty}$ such that
$\tilde{f}_{0, t_j} \rightarrow \tilde{f}_{0, \infty}$, and so it only
remains to show that $\tilde{f}_{0, \infty} \gamma(A_0, 0)
\tilde{\pi}_0 = \gamma(A_0, \infty) f_{0, \infty} \tilde{\pi}_0$. To
see this, first note that
\begin{multline}
\gamma(A_0, \infty) \tilde{f}_{0,t} \tilde{\pi}_0 - \tilde{f}_{0,t} \gamma(A_0, 0) \tilde{\pi}_0 \\ = \gamma(A_0, t) \tilde{f}_{0, t} \tilde{\pi}_0 - \tilde{f}_{0, t} \gamma(A_0, 0) \tilde{\pi}_0 + \left(\gamma(A_0, \infty) - \gamma(A_0, t) \right) \tilde{f}_{0,t} \tilde{\pi}_0.
\end{multline}
Equation \eqref{eqn:fholomorphic} and Theorem \ref{thm:gradflowconvergence} show that the right-hand side converges to zero, so we have
\begin{equation}\label{eqn:limitholomorphic}
\lim_{t \rightarrow \infty} \left( \gamma(A_0, \infty) \tilde{f}_{0,t} \tilde{\pi}_0 - \tilde{f}_{0,t} \gamma(A_0, 0) \tilde{\pi}_0 \right) = 0.
\end{equation}

Now consider the equation
\begin{multline*}
\gamma(A_0, \infty) \tilde{f}_{0, \infty} \tilde{\pi}_0 - \tilde{f}_{0, \infty} \gamma(A_0, 0) \tilde{\pi}_0 \\ = \gamma(A_0, \infty) (\tilde{f}_{0, \infty} - \tilde{f}_{0, t_j}) \tilde{\pi}_0  - (\tilde{f}_{0, \infty} - \tilde{f}_{0, t_j}) \gamma(A_0, 0) \tilde{\pi}_0 + \gamma(A_0, \infty) \tilde{f}_{0, t_j} \tilde{\pi}_0 - \tilde{f}_{0, t_j} \gamma(A_0, 0)  \tilde{\pi}_0.
\end{multline*}
The convergence $\tilde{f}_{0, t_j} \rightarrow \tilde{f}_{0,
  \infty}$, together with equation \eqref{eqn:limitholomorphic}, shows
that all of the terms on the right-hand side converge to zero, hence 
\begin{equation}
\gamma(A_0, \infty) \tilde{f}_{0, \infty} \tilde{\pi}_0 = \tilde{f}_{0, \infty} \gamma(A_0, 0) \tilde{\pi}_0
\end{equation}
as required.
\end{proof}

\begin{proof}[Proof of Theorem \ref{thm:convergencetogradedobject}]
Let $A_0 \in \Rep(Q, {\bf v})$ be the initial condition for the
gradient flow of $\| \Phi(A) - \alpha \|^2$, and consider the
sub-representation $\tilde{A}_0 \in \Rep(Q, {\bf \tilde{v}})$
corresponding to the first term in the H-N-J-H filtration of $A_0 \in
\Rep(Q, {\bf v})$. Lemma \ref{lem:holomorphiclimit} shows that there
is a non-zero map $\tilde{f}_{0, \infty} : \Vect(Q, {\bf
\tilde{v}}) \rightarrow \Vect(Q, {\bf \tilde{v}})$ such that
$\tilde{A}_\infty \circ \tilde{f}_{0, \infty}  = \tilde{f}_{0, \infty}
\circ \tilde{A}_0$, where $\tilde{A}_\infty = \lim_{t \rightarrow
\infty} \tilde{A}_0$ and $\tilde{f}_{0, \infty}$ satisfies the
conditions of Proposition \ref{prop:stableisomorphism}. The proof of
Proposition \ref{prop:filtrationlimit} then shows that
$\frac{\deg_\alpha(Q, {\bf \tilde{v}_\infty})}{\rank(Q, {\bf
\tilde{v}_\infty})} = \frac{\deg_\alpha(Q, {\bf \tilde{
v}_0})}{\rank(Q, {\bf \tilde{v}_0})}$.

Applying Proposition \ref{prop:stableisomorphism} shows that $\tilde{f}_{0, \infty}$ is injective (since it is non-zero), and therefore an isomorphism onto the image $\tilde{f}_{0, \infty} \circ \Vect(Q, {\bf \tilde{v}})$. Now repeat the process on the quotient representation $A_0 / \tilde{A}_0$.  Doing this for each term in the H-N-J-H filtration of $A_0$ shows that (along a subsequence) the flow converges to the graded object of the H-N-J-H filtration of $A_0$. Theorem \ref{thm:gradflowconvergence} shows that the limit exists along the flow independently of the subsequence chosen, which completes the proof of Theorem \ref{thm:convergencetogradedobject}.
\end{proof}

\section{Fibre bundle structure of strata}\label{sec:local}

The results in Section~\ref{sec:HN} and~\ref{sec:graded} establish the
relationship between the negative gradient flow on $\Rep(Q, {\bf
v})$ and the Harder-Narasimhan stratification of $\Rep(Q, {\bf
v})$. In particular, since Theorem~\ref{theorem:equivstratification}
shows that the analytic and algebraic stratifications are equivalent,
we may without ambiguity refer to ``the'' stratum $\Rep(Q, {\bf
v})_{{\bf v}^*} = S_{{\bf v}^*}$ associated to a H-N type ${\bf v}^*$. 

In this section, we turn our attention back to the analytical
description of the stratum $\Rep(Q, {\bf v})_{{\bf v}^*}$, and in particular construct an
explicit system of local coordinates around each point in $\Rep(Q, {\bf v})_{{\bf v}^*}$. The
key idea in our construction is to first build a coordinate system
near a critical point \(A \in C_{{\bf v}^*} \subseteq \Rep(Q, {\bf v})_{{\bf v}^*}\), since the criticality
of $A$ gives us a convenient way of parametrizing a neighborhood of
$A$. These local coordinates near $C_{{\bf v}^*}$ may then be translated by
$G_\C$ to other points in $\Rep(Q, {\bf v})_{{\bf v}^*}$, since (as seen in
Lemma~\ref{lem:cangetclose}) any point in $\Rep(Q, {\bf v})_{{\bf v}^*}$ can be brought
arbitrarily close to a point in $C_{{\bf v}^*}$ via the action of 
\(G_\C\). The main results in this direction are
Propositions~\ref{prop:critical-local-co-ordinates} and
~\ref{prop:local-coord-on-stratum}.  Moreover, as a
consequence of these explicit local descriptions, we also obtain a
formula for the codimension of Harder-Narasimhan strata, and an
explicit description of the stratum as a fiber bundle in
Proposition~\ref{prop:strata-fibre-bundle}. They also lead to an
identification of a tubular neighborhood $U_{{\bf v}^*}$ of $\Rep(Q, {\bf v})_{{\bf v}^*}$ with the
(disk bundle of the) normal bundle to $\Rep(Q, {\bf v})_{{\bf v}^*}$.

We note that Reineke also provides a description of each
Harder-Narasimhan stratum as a fiber bundle \cite[Proposition
3.4]{Rei03}. However, with our analytical approach we are able to do
somewhat more, i.e. we obtain explicit local coordinates near all
points in the stratum. We expect these
descriptions to be useful in the Morse-theoretic analysis of the
hyperk\"ahler Kirwan map for Nakajima quiver varieties. Indeed, in
Section~\ref{sec:applications} we take a step in this direction by using these local coordinates to give
a description of the singularities of $\PhiC^{-1}(0)$, the zero level
set of the holomorphic moment map, which is an important intermediate
space used in the construction of a hyperk\"ahler quotient.

The first step is to use the complex group action to provide specific local
coordinates around any point in $\Rep(Q, {\bf v})$. Let ${\bf v^*}$ be a H-N type, and let \(\Rep(Q, {\bf
  v})_{*,{\bf v^*}}\) denote the subset of the H-N stratum $\Rep(Q, {\bf
  v})_{\bf v^*}$ which preserves a fixed filtration, denoted $*$, of type
${\bf v}^*$. Equip $\g_\C$ with the direct sum metric $\g_\C \cong \g
\oplus i \g$, and let $(\rho_A^\C)^*: T_A \Rep(Q, {\bf v}) \to \g_\C$
denote the adjoint of the infinitesmal action of $\g_\C$.

\begin{lemma}\label{lem:localdiffeo}
Let \(Q = ({\mathcal I}, E)\) be a quiver and 
\({\bf v} \in \Z^{\mathcal I}_{\geq 0}\) a dimension vector. Let \(\mu \in
\R^{\rank(Q, {\bf v})}\) be a H-N type, ${\bf v}^*$ a filtration type, and $\Rep(Q, {\bf v})_{*,{\bf v^*}}$
as above. 
Suppose \(A \in \Rep(Q, {\bf v})_{*, {\bf v^*}}\). Then the function 
\begin{align*}
\begin{split}
\psi_A: (\ker \rho_A^\C)^{\perp} \times \ker(\rho_A^\C)^* & \rightarrow \Rep(Q, {\bf v}) \\
(u, \delta A)  & \mapsto \exp(u) \cdot (A + \delta A) \\
\end{split}
\end{align*}
is a local diffeomorphism at the point \((0,0) \in
\ker(\rho_A^\C)^{\perp} \times \ker(\rho_A^\C)^*\). 
\end{lemma}

\begin{proof}
The derivative of $\psi_A$ at $(u, \delta A) = (0,0)$ is given by
\begin{equation}
(d \psi_A)_{(0,0)} (\delta u, \delta a) = \rho_A^\C (\delta u) + \delta a. 
\end{equation}
By the definition of the domain of $\psi_A$, it is straightforward to see that
$(d \psi_A)_{(0,0)}$ is injective. Moreover, since $T_A \Rep(Q, {\bf v})
\cong \im \rho_A^\C \oplus \ker (\rho_A^\C)^*$, then $(d
\psi_A)_{(0,0)}$ is also surjective. 
The result follows. 
\end{proof}

With respect to the filtration $*$, the space $\ker (\rho_A^\C)^*
\subseteq T_A \Rep(Q, {\bf v}) \cong \Rep(Q, {\bf v})$ splits into two
subspaces
\begin{equation}
\ker (\rho_A^\C)^* \cong \left( \ker (\rho_A^\C)^* \cap T_A\Rep(Q,
  {\bf v})_{*,{\bf v^*}} \right) \oplus \left( \ker (\rho_A^\C)^* \cap (T_A
  \Rep(Q, {\bf v})_{*,{\bf v^*}})^\perp \right). 
\end{equation}
Therefore, the previous claim shows that there exists $\varepsilon_1
\geq 0$
such that any representation $B \in \Rep(Q, {\bf v})$ satisfying $\| B - A \| <
\varepsilon_1$ can be written uniquely as
\begin{equation}
B = \exp(u) \cdot (A + \delta a + \sigma)
\end{equation}
where $\delta a \in T_A\Rep(Q, {\bf v})_{*,{\bf v^*}} \cap \ker
(\rho_A^\C)^*$ and $\sigma \in (T_A \Rep(Q, {\bf v})_{*,{\bf v^*}})^\perp
\cap \ker (\rho_A^\C)^*$. The next proposition shows that if $A$ is
critical for \(f = \|\Phi - \alpha\|^2\) and $B \in \Rep(Q, {\bf
v})_{{\bf v}^*}$, then $\sigma = 0$. It is a key step toward describing local
coordinates at points on the stratum $\Rep(Q, {\bf v})_{{\bf v}^*}$.

\begin{proposition}\label{prop:critical-local-co-ordinates}
Let $Q=({\mathcal I}, E)$, ${\bf v} \in \Z^{\mathcal I}_{\geq 0}$,
$\mu$, and $\Rep(Q, {\bf v})_{*,\mu}$ be as in
Lemma~\ref{lem:localdiffeo}. 
Let ${\bf v}^*$ be a non-minimal H-N type, suppose that \(A \in \Rep(Q, {\bf v})_{*,{\bf v^*}}\) is a critical point of \(\|\Phi - \alpha\|^2\), and suppose \(B \in \Rep(Q, {\bf v})_{{\bf v}^*}\)
satisfies \(\|B - A\| < \varepsilon_1\) as above. Then there exist
unique elements \(u \in (\ker \rho_A^\C)^{\perp}\) and \(\delta a \in
\ker(\rho_A^{\C})^* \cap T_A\Rep(Q, {\bf v})_{*, {\bf v^*}}\) such that 
\[
B = \exp(u) \cdot (A + \delta a).
\]
\end{proposition}

As in the case of the Yang-Mills functional \cite{Das92}, the proof of
Proposition~\ref{prop:critical-local-co-ordinates} is by
contradiction. Namely, we show that if $\sigma \neq 0$ then there
exists $g \in G^\C$ such that $\| \Phi(g \cdot B) - \alpha \|^2 < \|
\Phi(A) - \alpha \|^2$, which implies that $g \cdot B$ has a different
Harder-Narasimhan type to $A$; this is because a critical point is an infimum of
$\| \Phi(A) - \alpha \|^2$ on each stratum (by Corollary \ref{cor:critical-infimum}). 
We have the following. 

\begin{lemma}\label{lem:perturbcriticalpoint}
Let $Q=({\mathcal I}, E)$, ${\bf v} \in \Z^{\mathcal I}_{\geq 0}$,
 ${\bf v}^*$, and $\Rep(Q, {\bf v})_{*,{\bf v^*}}$ be as in
Lemma~\ref{lem:localdiffeo}. 
Suppose $A \in \Rep(Q, {\bf v})_{*,{\bf v^*}}$ is a critical point of $\| \Phi
- \alpha \|^2$. Then there exists $\varepsilon_2 > 0$ and $\eta > 0$
such that if $\sigma \in (T_A \Rep(Q, {\bf v})_{*,{\bf v^*}})^\perp \cap \ker
(\rho_A^\C)^*$ and $\| \sigma \| = \varepsilon_2$, then $\| \Phi(A +
\sigma) - \alpha \| < \| \Phi(A) - \alpha \| - \eta$. Moreover, the
constants $\eta$ and $\varepsilon_2$ can be chosen uniformly over the set $C_{\bf v^*}
\cap \Rep(Q, {\bf v})_{*,{\bf v^*}}$.
\end{lemma}

\begin{proof}
If $A$ is critical, then 
\begin{equation}\label{eqn:estimate-for-functional}
\| \Phi(A + \sigma) - \alpha \|^2 = \| \Phi(A) - \alpha \|^2 + \left<
H_{\| \Phi(A) - \alpha \|^2} (\sigma), \sigma \right> + O(\sigma^3).
\end{equation}
Since $\sigma \in \ker (\rho_A^\C)^*$, then $I \sigma \in \ker \rho_A^*$ and a calculation using the results in Section \ref{subsec:hk-intro} shows that
\begin{align}
\begin{split}
H_{\| \Phi(A) - \alpha \|^2} (\sigma) & = -2 I \rho_A \rho_A^* I \sigma  - 2 I \delta \rho_A (\Phi(A) -
\alpha) (\sigma) \\
& = - 2 I \delta \rho_A (\Phi(A) - \alpha) (\sigma)  \\
& = 2 \sum_{a \in \mathcal{E}} \left[ - i (\Phi(A) - \alpha), \sigma_a \right].
\end{split}
\end{align}
Combining Lemma \ref{lem:eigenvaluesdetermined} and Corollary \ref{cor:critical-infimum} shows that when written with respect
to the fixed filtration $*$, the element $-i(\Phi(A) -
\alpha)$ is diagonal with entries given by the Harder-Narasimhan slope-vector of
$A$, in descending order. Since $\sigma$ is lower-triangular with
respect to the Harder-Narasimhan filtration of $A$, then each term $\left[ - i (\Phi(A) -
\alpha), \sigma_a \right]$ has the form
\begin{equation*}
\left( \begin{matrix} 0 & 0 & 0 & \cdots & 0 \\ \tau_{2,1} \sigma_a^{2,1} & 0 & 0 & \cdots & 0 \\ \tau_{3,1} \sigma_a^{3,1} & \tau_{3,2} \sigma_a^{3,2} & 0 & \cdots & 0 \\ \vdots & \vdots & \vdots & \ddots & \vdots \\ \tau_{L, 1} \sigma_a^{L,1} & \tau_{L, 2} \sigma_a^{L, 2} & \cdots & \tau_{L, L-1} \sigma_a^{L, L-1} & 0 \end{matrix} \right) ,
\end{equation*} 
where each $\tau_{i,j}$ is negative. Therefore
\begin{align*}
\left< H_{\| \Phi(A) - \alpha \|^2} (\sigma), \sigma \right> & = \left< 2 \sum_{a \in \mathcal{E}} \left[ - i (\Phi(A) - \alpha, \sigma_a \right], \sigma_a \right> \\
 & < 0.
\end{align*} 
Hence \eqref{eqn:estimate-for-functional} shows that there exists
$\eta > 0, \varepsilon_2>0$ such that $\| \Phi(A + \sigma) - \alpha \|^2 < \| \Phi(A) -
\alpha \|^2 - \eta$ if $\| \sigma \| = \varepsilon_2$. Since the operator $- I \delta \rho_A (\Phi(A)
- \alpha) : T_A \Rep(Q, {\bf v}) \rightarrow T_A \Rep(Q, {\bf v})$
only depends on the value of $\Phi(A) - \alpha$, this estimate is
uniform over the set $C_\mu \cap \Rep(Q, {\bf v})_{*,{\bf v^*}}$.

\end{proof}

We may now prove the Proposition.

\begin{proof}[Proof of Proposition \ref{prop:critical-local-co-ordinates}]
Let $g_t = \exp(- t \Lambda_{\bf v^*})$, where $\Lambda_{\bf v^*}$ is the diagonal
(with respect to the splitting given by the critical representation
$A$) matrix from Definition \ref{def:function-type}. Since $g_t$ is constant on each summand of the
splitting of $\Vect(Q,{\bf v})$ defined by $A$, we have $g_t \cdot A =
A$. Hence for $\sigma \in \ker(\rho_A^\C)^* \cap T_A(\Rep(Q, {\bf
v})_{*,{\bf v^*}})^\perp$ we have $g_t \cdot (A + \sigma) = A + g_t \cdot
\sigma$. If $\| \sigma \| < \frac{1}{2} \varepsilon_2$, then since
$\sigma$ is lower-triangular with respect to the filtration $*$ and
since ${\bf v^*}$ is non-minimal (hence has filtration length at least $2$),
we have $\| g_t \cdot (A + \sigma) - A \| = \| g_t \cdot \sigma \|
\rightarrow \infty$ as $t \rightarrow \infty$. Therefore there exists
$t > 0$ such that $\| g_t \cdot (A + \sigma) - A \| = \varepsilon_2$; for
such a $t$, by Lemma \ref{lem:perturbcriticalpoint} we have $\| \Phi(
g_t \cdot(A + \sigma) ) - \alpha \| < \| \Phi(A) - \alpha \| - \eta$.

Now note that since $\Phi : \Rep(Q, {\bf v}) \rightarrow
\mathfrak{g}^*$ is continuous, given the value of $\eta$ in Lemma
\ref{lem:perturbcriticalpoint}, there exists $\varepsilon_3 > 0$ (depending on $A$ and $\sigma$) such
that if $\| \delta a \| < \varepsilon_3$ then $\| \Phi(A + \delta a +
\sigma) - \Phi(A + \sigma) \| < \frac{1}{2} \eta$.  As above, for any
\(t \geq 0\) we have $\| g_t \cdot \delta a \| \leq \| \delta a \|$,
so if $\| \delta a \| < \varepsilon_3$, then $\| g_t \cdot (A +
\delta a) - A \| \leq \| \delta a \| < \varepsilon_3$, which implies 
$\| \Phi( g_t \cdot (A + \delta a + \sigma)) - \Phi( g_t \cdot (A + \sigma) ) \| < \frac{1}{2}
\eta$. Therefore
\begin{equation}
\| \Phi(g_t \cdot (A + \delta a + \sigma)) - \alpha \| < \| \Phi(A) - \alpha \| - \frac{1}{2} \eta,
\end{equation}
which contradicts Lemma \ref{lem:H-N-infimum}. This shows that for any \(\sigma \in
\ker(\rho_A^{\C})^* \cap T_A \Rep(Q, {\bf v})_{*,{\bf v^*}}\) with $0 < \| \sigma \| <
\frac{1}{2} \varepsilon_2$, the element $\exp(u) \cdot (A + \delta a +
\sigma) \notin \Rep(Q, {\bf v})_{{\bf v}^*}$. The result now follows from
Lemma~\ref{lem:localdiffeo}. 
\end{proof}

This Proposition shows that near a critical point $A$ of $\|\Phi -
\alpha\|^2$ of H-N type ${\bf v^*}$, the stratum $\Rep(Q, {\bf v})_{\bf v^*}$ has a
local manifold structure with dimension 
$$
\dim \Rep(Q, {\bf v})_{\bf v^*} = \dim (\ker \rho_A^\C)^{\perp} + \dim(\ker(\rho_A^\C)^* \cap T_A\Rep(Q, {\bf v})_{*,{\bf v^*}}).
$$
In particular, locally near $A$ the dimension of the
normal bundle to $\Rep(Q, {\bf v})_{\bf v^*}$ is given by $\dim(\ker(\rho_A^*\C)^*
\cap T_A\Rep(Q, {\bf v})_{*,{\bf v^*}}^{\perp})$. Our next series of
computations shows that in fact these dimensions are independent of
$A$, which implies in particular that the H-N strata are locally
closed manifolds with well-defined
(constant-rank) normal bundles in $\Rep(Q, {\bf v})$. 
Along the way, we also compare our formula for the dimension of the
normal bundle with 
that given in Kirwan's manuscript \cite{Kir84}.

Let $\Rep(Q, {\bf v})_*^{LT}$ denote the subspace
of representations in $\Rep(Q, {\bf v})$ that are lower-triangular
with respect to
a given fixed filtration $*$, and let $(\mathfrak{g}_\C)_*^{LT}$ denote
the subspace of $\mathfrak{g}_\C$ for which the
associated endomorphisms of $\Vect(Q, {\bf v})$ are lower-triangular
with respect to the same filtration.
We begin by observing that the subspace $(\g_\C)_*^{LT}$ injects via
the infinitesmal action to $T_A\Rep(Q, {\bf v})$ for any
representation which is split with respect to that same filtration.

\begin{lemma}\label{lemma:lower-triangular-kernel}
Let $Q = ({\mathcal I}, E)$ be a quiver and
${\bf v} \in \Z^{\mathcal I}_{\geq 0}$ a dimension vector. 
Let $*$ denote a fixed filtration of $\Vect(Q, {\bf v})$ of filtration type
${\bf v}^*$, and let \(A \in \Rep(Q, {\bf v})\) be a representation of H-N
type ${\bf v}^*$, which is split with respect to the fixed filtration. Then
\[
\dim \ker(\rho_A^\C \mid_{(\g_\C)_*^{LT}}) = 0.
\]
\end{lemma}

\begin{proof}
Let $\Vect(Q, {\bf v}) \cong \Vect(Q, {\bf v_1}) \oplus \cdots \oplus
\Vect(Q, {\bf v_L})$ be the given splitting of the vector spaces,
ordered by decreasing
$\alpha$-slope. Then an element $u \in (\mathfrak{g}_\C)_*^{LT}$
consists of a sum of maps 
$u_{jk} : \Vect(Q, {\bf v_j}) \rightarrow \Vect(Q, {\bf v_k})$ where
$j < k$, so the $\alpha$-slope of the domain is always greater than
the $\alpha$-slope of the image. If $u \in \ker \rho_A^\C$ then
$\left[ u, A_a \right] = 0$ for all edges $a \in E$, and so each
component $u_{jk}$ of $u$ is a map between semistable representations
where the $\alpha$-slope is strictly decreasing. Therefore Proposition
\ref{prop:stableisomorphism} shows that $u_{jk} = 0$ for all pairs $j
< k$, and hence $\left. \dim \ker \rho_A^\C
\right|_{\mathfrak{g}_\C^{LT}} = 0$. 
\end{proof}

The next result is a formula for the rank of the normal bundle to
the stratum locally near a critical point $A$, which in particular
shows that it is independent of the choice of $A$. 
This formula is also contained in a different form in
Proposition 3.4 of \cite{Rei03} and Lemma 4.20 of \cite{Kir84}; 
however, one of the advantages of our formula is that it can also be
used in the hyperk\"ahler
case, and in particular on the singular space
$\PhiC^{-1}(0)$ (see Section \ref{sec:applications} for more details).

\begin{proposition}\label{prop:codimension}
Let \(Q = ({\mathcal I}, E)\) be a quiver and 
${\bf v} \in \Z^{\mathcal I}_{\geq 0}$ a dimension vector. 
Let ${\bf v^*}$ be a H-N type and $*$ a fixed
filtration of type ${\bf v}^*$. Suppose \(A \in \Rep(Q,
{\bf v})_{*,{\bf v^*}}\) is a critical  point of  $\|\Phi - \alpha\|^2$. 
The (complex) codimension of the H-N stratum $\Rep(Q, {\bf v})_{{\bf v}^*}$ in
$\Rep(Q, {\bf v})$, locally near $A$, is 
\begin{equation}\label{eqn:codimension-definition}
d(Q, {\bf v}, {\bf v}^*) := \dim_\C \Rep(Q, {\bf v})_*^{LT} - \dim_\C (\mathfrak{g}_\C)_*^{LT}.
\end{equation}
In particular, the codimension is independent of the choice of $A$ in
$C_{{\bf v}^*}$. 
\end{proposition}

\begin{proof}
From Lemma~\ref{lem:localdiffeo} and
Proposition~\ref{prop:critical-local-co-ordinates} we see that at a
critical representation $A$ in $\Rep(Q, {\bf v})_{{\bf v}^*}$, the dimension
of the subspace in $T_A\Rep(Q, {\bf v})$ orthogonal to the stratum is
$\dim(\ker(\rho_A^C)^* \cap (T_A\Rep(Q, {\bf v})_{*,{\bf v^*}})^{\perp})$,
where $*$ denotes the H-N filtration associated to $A$. It is
straightforward from the definition that 
\[
\Rep(Q, {\bf v})_*^{LT} = (T_A\Rep(Q,{\bf v})_{*,{\bf v^*}})^{\perp}
\]
under the standard identification of $T_A\Rep(Q, {\bf v})$ with
$\Rep(Q, {\bf v})$, so the dimension of the normal space is
\(\dim \left( \ker(\rho_A^\C)^* \cap \Rep(Q, {\bf v})_*^{LT} \right)\). 
We compute 
\begin{align*}
\dim \left( \ker (\rho_A^\C)^* \cap \Rep(Q, {\bf v})_*^{LT} \right) & = \dim \Rep(Q, {\bf v})_*^{LT} - \dim \left. \im (\rho_A^\C)^* \right|_{\Rep(Q, {\bf v})_*^{LT}} \\
 & = \dim \Rep(Q, {\bf v})_*^{LT} - \dim (\mathfrak{g}_\C)_*^{LT} + \dim \ker \left. \rho_A^\C \right|_{(\mathfrak{g}_\C)_*^{LT}} \\
 & = \dim \Rep(Q, {\bf v})_*^{LT} - \dim (\mathfrak{g}_\C)_*^{LT},
 \end{align*}
where the second equality uses that $(\rho_A^\C)^*(\Rep(Q,{\bf
 v})_*^{LT} \subseteq (\g_\C)_*^{LT}$ and the last equality uses
 Lemma~\ref{lemma:lower-triangular-kernel}. Finally, we observe that
 the last quantity \(\dim \Rep(Q, {\bf v})_*^{LT} -
 \dim(\g_\C)_*^{LT}\) is the same for any choice of filtration $*$ of
  type ${\bf v}^*$, so $d(Q, {\bf v}, {\bf v}^*)$ is indeed independent of the
 choice of critical point $A$. 
 \end{proof}

\begin{remark}
The dimension $d(Q, {\bf v}, {\bf v^*})$ also has a cohomological interpretation generalising Remark 2.4 in \cite{wilkin09} (which concerns the case where the representation splits into two subrepresentations). This interpretation is in turn based on \cite[Section 2.1]{Ringel76}. Given a representation $A$ of H-N type ${\bf v^*}$, let $A^{gr}$ denote the graded object of the H-N filtration associated to $A$. Then we have the following complex associated to the infinitesimal action of $A^{gr}$
\begin{equation*}
\xymatrix{
(\mathfrak{g}_\C)_*^{LT} \ar[r]^(0.4){\rho_{A^{gr}}^\C} & \Rep(Q, {\bf v})_*^{LT} ,
}
\end{equation*}
and $d(Q, {\bf v}, {\bf v}^*)$ corresponds to the Euler characteristic of this complex, which has an interpretation as the sum of the Euler characteristics given in \cite[Remark 2.4]{wilkin09}.
\end{remark}

In Theorem \ref{thm:gradflowconvergence} we saw that the
gradient flow of $\| \Phi - \alpha \|^2$ converges. As a consequence, 
Kirwan's Morse theory in \cite{Kir84} applies to
$\Rep(Q, {\bf v})$ and so \cite[Lemma 4.20]{Kir84} (which computes the
dimension of the negative normal bundle at the critical sets of $\|
\Phi - \alpha\|^2$) should also give
the same answer as Proposition~\ref{prop:codimension} above. To see the
translation between the two formulae, we recall that Kirwan's formula
is equivalent to 
\begin{equation}
d(Q, {\bf v}, {\bf v}^*) = m - \dim G + \dim \stab(\beta),
\end{equation}
where $m = \dim \Rep(Q, {\bf v})_*^{LT}$, $G$ is the group, $\beta
= \Phi(A) - \alpha$ at a critical point $A \in \Rep(Q, {\bf
v})_\mu$ and the stabilizer is with respect to the adjoint action of
$G$ on $\g$. Therefore, it remains
to show that $\dim G - \dim \stab(\beta) = \dim
(\mathfrak{g}_\C)_*^{LT}$. This follows from observing that
$\stab(\beta)$ is the block-diagonal part of $\mathfrak{g}$ (with
respect to the H-N filtration $*$ associated to $A$), $\dim (\mathfrak{g}_\C)_*^{LT} =
\dim (\mathfrak{g}_\C)_*^{UT}$ where $(\g_\C)_*^{UT}$ denotes the
upper-triangular part of $\g_\C$ with respect to $*$, and that $\dim
(\mathfrak{g}_\C)_*^{LT} + 2 \dim \stab(\beta) + \dim
(\mathfrak{g}_\C)_*^{UT} = \dim \mathfrak{g}_\C = 2 \dim G$.

By using the complex group action to get close to $C_{{\bf v}^*}$, we may use
Proposition~\ref{prop:critical-local-co-ordinates} to also give local
coordinates near any point in $\Rep(Q, {\bf v})_{*,{\bf v^*}}$, not just
those which are critical.

\begin{proposition}\label{prop:local-coord-on-stratum}
Let \(Q = ({\mathcal I}, E)\) be a quiver and 
${\bf v} \in \Z^{\mathcal I}_{\geq 0}$ a dimension vector.
Let $\mu$ be a non-minimal Harder-Narasimhan type of $\Rep(Q, {\bf
  v})$.  Suppose $A \in \Rep(Q, {\bf v})_{*,{\bf v^*}}$ and let $A_\infty$ denote
  the limit point of the negative gradient flow with respect to
  $\|\Phi - \alpha\|^2$ with initial condition $A$. Then 
there exists $\varepsilon > 0$ and \(g \in G_\C\) such that for any $B \in
  \Rep(Q, {\bf v})_{{\bf v}^*}$ with $\| B - A \| < \varepsilon$,
  there exist unique elements $u \in (\ker \rho_{A_\infty}^\C)^\perp$ and
  $\delta a \in \ker (\rho_{A_\infty}^\C)^* \cap T_A \Rep(Q, {\bf v})_{*,{\bf v^*}}$ such that
  $B = g^{-1} \exp(u) \cdot (A_\infty + \delta a)$.
\end{proposition}

\begin{proof}
By Proposition~\ref{prop:critical-local-co-ordinates}, there exists an
\(\varepsilon_1 > 0\) such that the open set \(\{ B \in \Rep(Q, {\bf
  v}) \, \mid \, \| B - A_\infty \| < \varepsilon_1 \}\) in $\Rep(Q, {\bf
  v})_{\mu, {\bf v}^*}$ is a local coordinate chart for $\Rep(Q, {\bf v})$ centered
at $A_\infty$. Since
the finite-time negative gradient flow of \(\|\Phi - \alpha\|^2\) is
contained in a $G_\C$-orbit, from results of Section~\ref{sec:gradientflow}
there exists \(g \in G_\C\) such that \(\|g\cdot A - A_\infty\| <
\frac{1}{2} \varepsilon_1.\) Then by the choice of $\varepsilon_1$
above, the open set \(\{ B \in \Rep(Q, {\bf v})_{{\bf v}^*} \, \mid \, \| B -
g \cdot A \| < \frac{1}{2} \varepsilon_1 \}\) in $\Rep(Q, {\bf v})_{{\bf v}^*}$ is a
coordinate neighborhood centered around $g \cdot A$, with local
coordinates given by Proposition~\ref{prop:critical-local-co-ordinates}. The element
$g^{-1}$ induces a diffeomorphism which translates this coordinate
neighborhood to an open neighborhood around $B$. In particular,
there exists some $\varepsilon>0$ such that the conditions in the
Proposition are satisfied. 
\end{proof}

The next result shows that the H-N strata 
have a fibre bundle structure. We first need the following. 

\begin{lemma}\label{lem:strata-group-action}
Let $Q = ({\mathcal I}, E)$ be a quiver 
and \({\bf v} \in \Z^{\mathcal I}_{\geq 0}\) a dimension vector.
Let ${\bf v^*}$ be a H-N type, and $*$ a filtration of type ${\bf v^*}$. Then 
$\Rep(Q, {\bf v})_{{\bf v}^*} = G_\C \cdot \Rep(Q, {\bf v})_{*,{\bf v^*}}$. 
\end{lemma}

\begin{proof}
Since the action of $G_\C$ preserves Harder-Narasimhan types, 
$G_\C \cdot \Rep(Q, {\bf v})_{*,{\bf v^*}} \subseteq \Rep(Q, {\bf
v})_{{\bf v}^*}$. Conversely, any filtration of type ${\bf v}^*$ is equivalent to the
given filtration $*$ by some element of $G_\C$, and therefore any
representation of type ${\bf v}^*$ can be mapped by $G_\C$ to a representation that
preserves the filtration $*$. Hence $\Rep(Q,
{\bf v})_{{\bf v}^*} \subseteq G_\C \cdot \Rep(Q, {\bf v})_{*,{\bf v^*}}$.
\end{proof}

From the lemma above we can conclude that each Harder-Narasimhan
stratum fibers over a homogeneous space of $G$. Let $G_{\C, *},
G_*$ denote the subgroups of $G_\C$ and $G$, respectively, which
preserve the given fixed filtration $*$. 

\begin{proposition}\label{prop:strata-fibre-bundle}
Let \(Q = ({\mathcal I}, E)\) be a quiver and
${\bf v} \in \Z^{\mathcal I}_{\geq 0}$ a dimension vector.
Let ${\bf v}^*$ be a H-N type. Then 
$$
\Rep(Q, {\bf v})_{{\bf v}^*} \cong  G_\C \times_{G_{\C,*}} \Rep(Q, {\bf v})_{*,{\bf v^*}} \cong G \times_{G_*} \Rep(Q, {\bf v})_{*,{\bf v^*}}.
$$
\end{proposition}

\begin{proof}
Given Lemma~\ref{lem:strata-group-action}, to prove the first equality it
only remains to show that the map $\varphi : G_\C \times_{G_{\C,*}}
\Rep(Q, {\bf v})_{*,{\bf v^*}} \rightarrow \Rep(Q, {\bf v})_{{\bf v}^*}$ given by $[g, A]
\mapsto g \cdot A$ is a local diffeomorphism. If $A$ is a critical
point, Proposition \ref{prop:critical-local-co-ordinates} shows that
$d \varphi_{[\id, A]}$ is surjective. Since $G_{\C,*}$ is by definition the
maximal subgroup of $G_\C$ that preserves $\Rep(Q, {\bf v})_{*,{\bf v^*}}$,
this means  $d\varphi_{[\id, A]}$ is injective also. Therefore $\varphi$ is a local
diffeomorphism at $[\id, A]$.  By Lemma \ref{lem:strata-group-action}, the
action of $G_\C$ translates this local diffeomorphism over the whole
stratum $\Rep(Q, {\bf v})_{{\bf v}^*}$.
The second equality follows from the isomorphism of groups $G_\C = G
\times_{G_*} G_{\C,*}$ \cite[Theorem 2.16]{Das92}. 
\end{proof}

As a result we have a simple description of the cohomology of each stratum. Given a H-N type ${\bf v^*} = ({\bf v}_1^*, \ldots, {\bf v}_L^*)$ and a stability parameter $\alpha$, define $\alpha_i = \left. \alpha  \right|_{\Vect(Q, {\bf v}_i)} - \mu_\alpha(Q, {\bf v}_i)$. This is the stability parameter associated to the quotient representation in $\Rep(Q, {\bf v}_i)$ that corresponds to the $i^{th}$ term in the graded object of a representation in $\Rep(Q, {\bf v})_{\bf v^*}$, i.e. the parameter is shifted by an amount corresponding to the $\alpha$-slope of the subrepresentation. (See also \eqref{eqn:quotient-parameter} and Proposition \ref{prop:momentmapdetermined} for similar constructions).

\begin{proposition}\label{prop:cohomology-stratum}
Let ${\bf v} \in \Z^{\mathcal I}_{\geq 0}$ be a dimension vector, and let ${\bf v^*} = ({\bf v}_1, \ldots, {\bf v}_L)$ be a Harder-Narasimhan type. Then 
\begin{equation}
H_G^* \left( \Rep(Q,{\bf v})_{\bf v^*} \right) \cong \bigotimes_{i=1}^L H_{G_{{\bf v}_i}}^* \left(\Rep(Q, {\bf v}_i)^{\alpha_i-ss} \right) ,
\end{equation}
where $G_{\bf v_i}$ is the group associated to the quiver $Q$ with dimension vector ${\bf v_i}$, and $\alpha_i$ is the stability parameter defined above.
\end{proposition}

\begin{proof}
The previous proposition shows that $\Rep(Q, {\bf v})_{{\bf v}^*} \cong G \times_{G_*} \Rep(Q, {\bf v})_{*,{\bf v^*}}$. Scaling the extension class of the filtration (i.e. the off-diagonal terms) gives the following deformation retract
\begin{equation*}
\Rep(Q, {\bf v})_{*,{\bf v^*}} \simeq \prod_{i=1}^L \Rep(Q, {\bf v}_i)^{\alpha_i-ss} ,
\end{equation*}
which is equivariant with respect to $G_*$. Therefore 
\begin{equation*}
\Rep(Q, {\bf v})_{{\bf v}^*} \simeq G \times_{G_*} \prod_{i=1}^L \Rep(Q, {\bf v}_i)^{\alpha_i-ss}
\end{equation*}
(where we consider the product as a block diagonal representation with the standard action of $G$ and $G_*$), and the result follows.
\end{proof}

We conclude with a description of an open neighborhood of each H-N
stratum as a disk bundle. 

\begin{proposition}\label{prop:kahler-normal-bundle}
Let \(Q = ({\mathcal I}, E)\) be a quiver and
 ${\bf v} \in \Z^{\mathcal I}_{\geq 0}$ a dimension vector. 
Let ${\bf v}^*$ be a filtration type with corresponding H-N type $\mu$. 
Then there exists an
open neighbourhood $U_{{\bf v}^*}$ of the H-N stratum $\Rep(Q, {\bf v})_{{\bf v}^*}$
in $\Rep(Q, {\bf v})$ that is homeomorphic to a disk bundle of rank 
$d(Q, {\bf v}, {\bf v}^*)$ over $\Rep(Q, {\bf v})_{{\bf v}^*}$.
\end{proposition}

\begin{proof}

Proposition~\ref{prop:critical-local-co-ordinates} shows that near
each \(A \in \Rep(Q, {\bf v})_{*,{\bf v^*}}\) which is critical for $\|\Phi
- \alpha\|^2$, the stratum $\Rep(Q, {\bf v})_{{\bf v}^*}$ is locally a
manifold of codimension $d(Q, {\bf v}, {\bf v}^*)$. This codimension is
independent of choice of critical representation $C_{{\bf v}^*} \cap \Rep(Q,
{\bf v})_{*,{\bf v^*}}$. Propostion~\ref{prop:local-coord-on-stratum} shows
that $\Rep(Q, {\bf v})_{{\bf v}^*}$ is locally a manifold of the same
codimension $d(Q, {\bf v}, {\bf v}^*)$ at any point in $\Rep(Q, {\bf
  v})_{*,{\bf v^*}}$. By Lemma~\ref{lem:strata-group-action}, 
local coordinates near points in $\Rep(Q, {\bf v})_{*,{\bf v^*}}$ can be
translated by the $G_\C$-action to any point in $\Rep(Q, {\bf
  v})_{{\bf v}^*}$; hence $\Rep(Q, {\bf v})_{{\bf v}^*}$ is a manifold of $\Rep(Q,
{\bf v})$ of constant codimension $d(Q, {\bf v}, {\bf v}^*)$. 
In particular, there exists a tubular neighborhood $U_{{\bf v}^*}$ of $\Rep(Q,
{\bf v})_{{\bf v}^*}$ in $\Rep(Q, {\bf v})$ satisfying the conditions of the
theorem. 
\end{proof}

\section{Applications}\label{sec:applications}

\subsection{Hyperk\"ahler quotients and Nakajima quiver varieties}\label{subsec:hK}

One of the main applications for the results of this manuscript is the study of the
topology of Nakajima quiver varieties. The purpose of this section is to describe some of these applications.

We first give details of the variant on the construction
used in the previous sections that yields the Nakajima quiver
varieties. The first step is to consider the holomorphic cotangent bundle $T^{*}\Rep(Q, {\mathbf v})$ of $\Rep(Q, {\mathbf v})$, with the identification 
\begin{equation}\label{eq:cotangent-double-arrows} 
T^{*}\Rep(Q, {\mathbf v}) \cong 
  \oplus_{a \in E}
\left( \Hom(V_{\out(a)}, V_{\inw(a)}) \oplus \Hom(V_{\inw(a)}, V_{\out(a)})\right).
\end{equation}
Here, for any two complex vector spaces $V$ and $V'$, we consider $\Hom(V',V)$
to be the complex dual of $\Hom(V,V')$ by the $\C$-linear pairing 
\[
(A,B) \in \Hom(V, V') \times \Hom(V', V) \mapsto \tr(BA) \in \C.
\]

\begin{remark}\label{remark:cotangent-double-arrows} The
  identification~\eqref{eq:cotangent-double-arrows} makes it evident
  that $T^*\Rep(Q, {\mathbf v})$ may be identified with
  $\Rep(\overline{Q}, {\mathbf v})$ where $\overline{Q}$ is quiver
  obtained from $Q$ by setting $\overline{\mathcal I} = {\mathcal I}$
  and, for the edges $\overline{E}$, adding for each $a \in E$ in $Q$
  an extra arrow $\overline{a}$ with the reverse orientiation.
\end{remark}

We equip $T^{*}\Rep(Q, {\mathbf v})$ with a K\"ahler structure $\omega_{\R}$ given by the standard structure on each $\Hom(-, -)$ in the summand, as for $\Rep(Q, {\bf v})$. This has a corresponding real moment map $\Phi_{\R}$ derived as before. However, being a holomorphic cotangent bundle, $T^{*}\Rep(Q, {\mathbf v})$ in addition has a canonical holomorphic symplectic structure $\omega_{\C}$ given by, for \((\delta A_{i}, \delta B_{i}) = ((\delta A_i)_{a}, (\delta B_{i})_{a})_{a \in E}\) for \(i = 1, 2,\) 
\begin{equation}\label{eq:def-TRepQv-holom-sympl-form}
\omega_{\C}\left((\delta A_1, \delta B_1), (\delta A_2, \delta B_2)\right)  = 
\sum_{a \in E} \tr((\delta A_1)_a(\delta B_2)_a - (\delta
  A_2)_a(\delta B_1)_a)  \in \C.
\end{equation}
The complex group \(G_{\C} = \prod_{\ell \in {\mathcal I}} GL(V_{\ell})\) also acts naturally by conjugation on $T^{*}\Rep(Q, {\mathbf v})$. Using the complex bilinear pairing
\[ 
\langle u, v \rangle = - \tr(u v)
\]
for \(u, v \in \gl(V_{\ell}),\) we may identify $\gl(V_{\ell})$ with its complex dual $\gl(V_{\ell})^{*}$. It is then straightforward to compute the holomorphic moment map
\(\Phi_{\C}: T^{*}\Rep(Q, {\mathbf v}) \to \g_{\C}^{*} \cong \g_{\C}\) for the $G_{\C}$-action with respect to $\omega_{\C}$. We have for $(A, B) \in T^{*}\Rep(Q, {\mathbf v})$ that 
\[
\Phi_{\C}(A,B)_{\ell} = \sum_{a: \inw(a) = \ell} A_{a} B_{a} - \sum_{a': \out(a') = \ell} B_{a'} A_{a'}, 
\]
which again we may write (using the same simplified notation as in~\eqref{eq:simplified-Kahler-mom-map-eq}) as  
\begin{equation}\label{eq:PhiC-simplified}
\Phi_{\C}(A,B) = [ A, B], 
\end{equation}
where the summation over arrows \(a \in E\) is understood. 

Suppose given a central parameter $\alpha$ and assume that $G$ acts freely on the intersection \(\Phi_{\R}^{-1}(\alpha) \cap \Phi_{\C}^{-1}(0).\) Then we define the \emph{Nakajima quiver variety} associated to $(Q, {\mathbf v}, \alpha)$ as 
\begin{equation}\label{eq:MalphaQv-definition} 
M_{\alpha}(Q, {\mathbf v}) := \Phi_{\R}^{-1}(\alpha) \cap \Phi_{\C}^{-1}(0) / G.
\end{equation}
This is a special case of a {\em hyperk\"ahler quotient}.

In the construction of Nakajima quiver varieties
recounted above, 
we take a level set not only of the real moment map $\Phi =
\Phi_\R$ but also of the holomorphic moment map $\Phi_\C$.
In order to analyze the topology of the Nakajima quiver variety via
equivariant Morse theory, we propose to take the following approach:
restrict first to
the level set $\PhiC^{-1}(0)$, and then use the Morse theory of
$f = \|\PhiR - \alpha\|^2$ on $\PhiC^{-1}(0)$ to analyze the
$G$-equivariant topology of $\PhiR^{-1}(\alpha) \cap
\PhiC^{-1}(0)$. 

The first few steps of such an analysis follow by a straightforward
application of our results in Sections~\ref{sec:gradientflow}
and~\ref{sec:HN}, as we now explain.  For instance, it is not
difficult to see that the negative gradient flow of $\|\Phi - \alpha\|^2$ of
Section~\ref{sec:gradientflow}, defined on all of $T^*\Rep(Q, {\mathbf
  v})$, in fact
preserves $\PhiC^{-1}(0)$. This follows from the fact that the
negative gradient flow preserves closed $G_{\C}$-invariant subsets. We
have the following simple corollary of
Theorem~\ref{thm:gradflowconvergence}.

\begin{proposition}\label{prop:closed-invariant-preserved}
Let $(V, G, \Phi, \alpha)$ be as in the statement of Theorem~\ref{thm:gradflowconvergence}, and let $\gamma(x, t): V \times \R \to V$ denote the negative gradient flow of \(f = \|\Phi - \alpha\|^{2}.\) Let \(Y \subseteq V\) be any closed $G_{\C}$-invariant subset of $V$. Then the negative gradient flow $\gamma(x, t)$ preserves $Y$, i.e. for any initial condition \(y_{0} \in Y, \gamma(y_{0}, t) \in Y\) for all \(t \in \R,\) and the limit point \(y_{\infty} := \lim_{t \rightarrow \infty} \gamma(y_{0}, t)\) is also contained in $Y$. 
\end{proposition}

\begin{proof} 
  As we already saw in the proof of Lemma~\ref{lemma:exists-alltime},
  equation~\eqref{eq:gradient-flow-def} implies that the negative
  gradient vector field is always contained in the image of the
  infinitesmal $G_{\C}$ action on $V$. Therefore, the finite-time
  gradient flow is contained in a $G_{\C}$-orbit, and hence in $Y$,
  since $Y$ is $G_{\C}$-invariant. Moreover, since $Y$ is closed, the
  limit \(y_{\infty} = \lim_{t \rightarrow \infty} \gamma(y_{0}, t)\)
  is also contained in $Y$, as desired.
\end{proof}

In the case of $T^*\Rep(Q, {\mathbf v})$, recall from
Remark~\ref{remark:cotangent-double-arrows} that we may also view
$T^{*}\Rep(Q, {\mathbf v})$ as itself a representation space
$\Rep(\overline{Q}, {\mathbf v})$, where $\overline{Q}$ is the quiver $Q$ with edges
``doubled''. The real moment map $\PhiR$ on $T^{*}\Rep(Q, {\mathbf
  v})$ is the usual K\"ahler moment map for the linear $G$-action on
$\Rep(\overline{Q}, {\mathbf v})$, so the results in
Section~\ref{sec:gradientflow} 
apply. From the Proposition above, we conclude that the negative
gradient flow of $f = \|\PhiR - \alpha\|^2$ preserves
$\PhiC^{-1}(0)$.

\begin{corollary}\label{cor:PhiC-level-preserved} 
Let $Q = ({\mathcal
    I}, E)$ be a quiver and ${\mathbf v}
  \in \Z_{\geq 0}^{{\mathcal I}}$ a dimension vector. Let 
$$
  \PhiR: T^{*}\Rep(Q, {\mathbf v}) \to \g^{*} \cong \g \cong \prod_{\ell \in {\mathcal I}}
  \u(V_{\ell})
  $$ and \(\PhiC: T^{*}\Rep(Q, {\mathbf v}) \to \g_{\C}^{*}
  \cong \g_{\C} \cong \prod_{\ell \in {\mathcal I}} \gl(V_{\ell})\) be
  the real and holomorphic moment maps, respectively, for the action
  of \(G = \prod_{\ell \in {\mathcal I}} U(V_{\ell})\) on
  $T^{*}\Rep(Q, {\mathbf v})$. Let \(\gamma(x, t): T^{*}\Rep(Q,
  {\mathbf v}) \times \R \to T^{*}\Rep(Q, {\mathbf v})\) denote the
  negative gradient flow on $T^{*}\Rep(Q, {\mathbf v})$ with respect
  to $\|\PhiR - \alpha\|^{2}$ for a choice of parameter \(\alpha \in
  Z(\g).\) Then the negative gradient flow $\gamma(x,t)$ preserves
  $\PhiC^{-1}(0)$. \end{corollary}

\begin{proof} 
The result follows immediately from Proposition~\ref{prop:closed-invariant-preserved} upon observing that $\PhiC^{-1}(0)$, by the definition of $\PhiC$ in~\eqref{eq:PhiC-simplified}, is a $G_{\C}$-invariant subset of $T^{*}\Rep(Q, {\mathbf v})$, and also that $\PhiC$ is continuous, so $\PhiC^{-1}(0)$ is closed. 
\end{proof}

Since the negative gradient flow $\gamma(x,t)$ restricts to $\PhiC^{-1}(0)$, the Morse 
strata $S_{\nu, {\bf v}^*}$ on $T^*\Rep(Q, {\bf v})$ also restrict to $\PhiC^{-1}(0)$, and we have the
following. 

\begin{definition}\label{def:Morse-strata-Nakajima}
  Let $Q = ({\mathcal I}, E)$, ${\mathbf v} \in \Z^{\mathcal I}_{\geq 0}$, $G$, $\PhiR$, and
  $\PhiC$ be as in Corollary~\ref{cor:PhiC-level-preserved}. Then the
  {\em analytic (or Morse-theoretic) stratum of $\PhiC^{-1}(0)$ of critical type ${\bf v}^*$} is
  defined as 
\[
Z_{{\bf v}^*} := \left\{ y \in \PhiC^{-1}(0) \, \mid \, \lim_{t \rightarrow
    \infty} \gamma(y,t) \in C_{{\bf v}^*} \right\}
\]
and is equal to $S_{{\bf v}^*} \cap \PhiC^{-1}(0)$. We call the decomposition 
\[
\PhiC^{-1}(0) = \bigcup_\nu Z_{{\bf v}^*}
\]
the {\em analytic (or Morse-theoretic) stratification of
  $\PhiC^{-1}(0)$.} 
\end{definition}

Now recall that in 
Section~\ref{sec:HN}, we saw that the Harder-Narasimhan
stratification on $T^*\Rep(Q, {\bf v})$ (viewed as $\Rep(\overline{Q},
{\mathbf v})$ as in Remark~\ref{remark:cotangent-double-arrows}) is identified with its Morse
stratification. 
We may also define a Harder-Narasimhan stratification on $\PhiCzero$ by
restriction, as follows. 
Let $\Phi_{\C}^{-1}(0)_{{\bf v}^*} := \Phi_\C^{-1}(0) \cap
T^*\Rep(Q, {\mathbf v})_{{\bf v}^*}$. Then we call the decomposition 
\begin{equation}\label{eq:def-HN-PhiClevel}
\PhiC^{-1}(0) = \bigcup_{{\bf v}^*} \PhiC^{-1}(0)_{{\bf v}^*}
\end{equation}
the {\em H-N stratification of $\PhiC^{-1}(0)$}. The closure property~\eqref{eq:HNclosures-partialorder} with respect to the partial ordering is still satisfied, i.e. 
\begin{equation}
\overline{\PhiC^{-1}(0)_{{\bf v}^*}} \subseteq \bigcup_{{\bf w^*} \geq {\bf v^*}} \PhiC^{-1}(0)_{{\bf w}^*}. 
\end{equation}
Moreover, the result above
that the Harder-Narasimhan and Morse stratifications are
identical on $T^*\Rep(Q, {\mathbf v})$ immediately implies that the
induced Harder-Narasimhan and Morse stratifications on $\PhiCzero$ are
also equivalent.

As mentioned above, 
our proposed approach to study Nakajima quiver varieties 
is to analyze the Morse theory of $f = \|\PhiR-\alpha\|^2$ on
$\PhiCzero$, however such an analysis is made
non-trivial by the fact that $\PhiC^{-1}(0)$ is usually a {\em singular}
space. In particular, the analogue of
Proposition~\ref{prop:kahler-normal-bundle} is no longer true on
$\PhiC^{-1}(0)$, i.e. open neighborhoods of the strata
$\PhiC^{-1}(0)_{{\bf v}^*}$ in
$\PhiC^{-1}(0)$ no longer have a description as a (constant-rank) disk
bundle over the strata.

Our next result concerns a convenient choice of stability
parameter $\alpha$, with respect to which computations with
Harder-Narasimhan strata of $\PhiCzero$ can be simplified.  More
specifically, it 
is often useful to choose a stability parameter $\alpha$ with respect
to which the H-N $\alpha$-length of the H-N strata are always less
than or equal to $2$. This greatly simplifies, for example, explicit
computations of Morse indices at the critical sets. Moreover, we may
additionally choose $\alpha$ such that on $\Rep(Q, {\bf v})$,
$\alpha$-semistability is equivalent to $\alpha$-stability.
We have the following.

\begin{definition}
A triple $(Q, {\bf v}, \alpha)$, consisting of a quiver $Q$,
associated dimension vector ${\bf v}$, and a choice of stability
parameter $\alpha \in Z(\mathfrak{g})$, is called {\em $2$-filtered}
if for any representation $A \in \Rep(Q, {\bf v})$ its associated
canonical H-N filtration has H-N $\alpha$-length less than
or equal to $2$.
\end{definition}

The following proposition shows that for a large class of quivers $(Q,
{\bf v})$, there exists a choice of stability parameter $\alpha$ such that $(Q,
{\bf v}, \alpha)$ is $2$-filtered.

\begin{proposition}\label{prop:twofiltered}
Let \(Q = ({\mathcal I}, E)\) be a quiver and \({\bf v} \in
\Z^{\mathcal I}_{\geq 0}\) a dimension vector. Suppose that there
exists \(\ell \in {\mathcal I}\) with \(v_\ell = \dim(V_\ell) = 1.\)
Then there exists a choice of stability parameter \(\alpha \in Z(\g)\)
such that
\begin{itemize} 
\item $(Q, {\mathbf v}, \alpha)$ is $2$-filtered, and 
\item $\Rep(Q, {\bf v})^{\alpha-ss} = \Rep(Q, {\bf v})^{\alpha-st}$. 
\end{itemize} 

\end{proposition}

\begin{remark}
We have 
restricted attention in our exposition to the unframed quotients
$\Rep(Q, {\mathbf v}) \mod G$ because any framed quotient $\Rep(Q,
{\mathbf v}, {\mathbf w}) \mod G$ for \(Q = ({\mathcal I}, E)\) can be
realized, by the construction by Crawley-Boevey in \cite[p. 261]{CraBoe01} (see also Remark
\ref{remark:framed-vs-unframed}), as an unframed quotient associated
to a different quiver $Q' = ({\mathcal I}', E')$. 
The construction
involves an addition of a ``vertex at infinity'',
i.e. \({\mathcal I}' = {\mathcal I} \cup \{\infty\},\) with \(v_\infty
= \dim_\C(V_\infty) = 1.\) Hence the hypothesis of
Proposition~\ref{prop:twofiltered} is always satisfied for an unframed
quiver representation space associated via this construction to a framed one.
\end{remark}

\begin{proof}[Proof of Proposition \ref{prop:twofiltered}]
Let \({\mathcal I} = \{1, 2, \ldots, n, \infty\}\) denote the
vertices of the quiver, where \(\infty \in {\mathcal I}\) denotes a
vertex with
\(v_\infty = \dim_{\C}(V_{\infty}) = 1.\)

The proof is by explicit construction. 
Let \(\bar{\alpha} \in i\R\)
be a pure imaginary parameter with \(i \bar{\alpha} < 0.\) 
Define $\alpha$ by 
\begin{equation}
\alpha_{\ell} = \left\{ 
\begin{tabular}{l}
\(\bar{\alpha},\) \quad \quad for \(1 \leq \ell \leq n,\) \\
\(- \bar{\alpha} \left( \sum_{s=1}^n v_{s} \right),\) for \(\ell =
\infty.\) \\
\end{tabular}
\right.
\end{equation} 
Let \(A \in \Rep(Q,{\bf v}).\) 
It is straightforward from the definition of
$\alpha$-slope that the maximal $\alpha$-semistable subrepresentation
$A' \in \Rep(Q, {\bf v'})$ with associated hermitian vector spaces
$\{V'_\ell \subseteq V_\ell\}_{\ell \in {\mathcal I}}$ is the maximal
subrepresentation of $A$ with $V'_\infty = V_\infty$. Now consider the
quotient representation $A/A'$, with associated dimension vector
${\mathbf v''}$. Then \(v''_\infty = 0.\) Since the
$\alpha$-parameter is constant on all other vertices $\ell \neq
\infty$, the $\alpha$-slope of any subrepresentation of $A/A'$ is
equal to the $\alpha$-slope $\mu_\alpha(Q, {\mathbf v''})$ of
$A/A'$. This implies $A/A'$ is already $\alpha$-semistable, and we
conclude that $A$ has Harder-Narasimhan $\alpha$-length less than or
equal to $2$, as desired.

To prove the last claim, we must show that if \(A \in \Rep(Q, {\mathbf
  v})\) is $\alpha$-semistable, then there are no proper
subrepresentations \(A' \in \Rep(Q, {\mathbf v'})\) of $A$ with
\(\mu_\alpha(Q, {\mathbf v'}) = \mu_\alpha(Q, {\mathbf v}).\) We first
observe from the definition of $\alpha$ that \(\mu_\alpha(Q, {\mathbf
  v}) = 0.\) We take cases: if \(\dim_\C(V'_\infty) = v'_\infty = 0,\)
then \(\mu_\alpha(Q, {\mathbf v'}) = i \overline{\alpha} < 0,\) hence
\(\mu_\alpha(Q, {\mathbf v'}) < \mu_\alpha(Q, {\mathbf v}).\) On the
other hand, if \(\dim_\C(V'_\infty) = v'_\infty = 1,\) and if $A'$ is
a proper subrepresentation, then there exists \(\ell \in {\mathcal
  I}\) with \(v'_\ell < v_\ell,\) so in particular \(\sum_{\ell=1}^n
v'_\ell < \sum_{\ell=1}^n v_\ell.\) From this it follows that
\(\mu_\alpha(Q, {\mathbf v'}) < \mu_\alpha(Q, {\mathbf v}).\) Hence
$\alpha$-semistability implies $\alpha$-stability. The reverse
implication follows from the definitions, so we conclude \(\Rep(Q,
{\mathbf v})^{\alpha-ss} = \Rep(Q, {\mathbf v})^{\alpha-st},\) as
desired. 
\end{proof}

The last result in this section is that when the H-N stratification is
$2$-filtered, the fibers of the restriction to $\PhiCzero$ of the
projection of tubular neighborhoods to H-N strata are
well-behaved at a critical point. 
Here we follow the notation of Sections~\ref{sec:preliminaries}
and~\ref{sec:HN}. 
Let \(\alpha\) be a central parameter such
that $(\overline{Q}, {\bf v}, \alpha)$ is $2$-filtered, as above. Let ${\bf v}^*$ be a H-N type of $T^*\Rep(Q, {\bf v})$ with respect to $\alpha$, and consider the corresponding H-N stratum $\PhiC^{-1}(0)_{{\bf v}^*}$
of $\PhiC^{-1}(0)$. In order to understand the topology of the H-N
stratification of $\PhiC^{-1}(0)$, we must analyze the open
neighborhoods $U_{{\bf v}^*} \cap \PhiC^{-1}(0)$ of $\PhiC^{-1}(0)_{{\bf v}^*}$, where
the $U_{{\bf v}^*}$ are the open sets in $\Rep(\overline{Q}, {\bf v})$
constructed in Proposition~\ref{prop:kahler-normal-bundle}. 
Let \(\pi_{{\bf v}^*}: U_{{\bf v}^*} \to S_{{\bf v}^*}\) denote the orthogonal projection of
the tubular neighborhood to the stratum $S_{{\bf v}^*}$. The result below
is that although $\PhiC^{-1}(0)$ is singular, if
$(\overline{Q}, {\bf v}, \alpha)$ is $2$-filtered, then at any
critical \(A \in C_{{\bf v}^*} \cap \PhiC^{-1}(0),\) the fiber 
\(\pi^{-1}_{{\bf v}^*}(A) \cap \PhiC^{-1}(0)\) is a vector space.

\begin{theorem}\label{theorem:PhiC-level-linear} 
Let \(Q = ({\mathcal I}, E)\) be a quiver and \({\bf v} \in \Z^{\mathcal I}_{\geq 0}\)
  a dimension vector.  Let \(\alpha \in Z(\g)\) be a 
 central parameter such that $(\bar{Q}, {\bf v}, \alpha)$
  is $2$-filtered. Let \(f = \|\Phi_{\R} - \alpha\|^{2}\) be the
  norm-square of the real moment map on $T^{*}\Rep(Q, {\bf v}) =
  \Rep(\bar{Q}, {\bf v})$, and let $(A, B) \in \Crit(f) \cap
  \PhiC^{-1}(0)_{{\bf v}^*} \subseteq T^{*}\Rep(Q, {\bf v})$.
  Then locally near $(A, B)$, \begin{equation}
    \label{eq:PhiC-level-linearization} \PhiC^{-1}(0) \cap \pi_{{\bf v}^*}^{-1}(A, B) = \ker (d\PhiC)_{(A,B)} \cap \pi_{{\bf v}^*}^{-1}(A, B), \end{equation} where we consider both
  $\pi_{{\bf v}^*}^{-1}(A,B)$ and $\ker (d\PhiC)_{(A,B)}$ as
  affine subspaces of $T^{*}\Rep(Q, {\bf v})$ going through $(A,B)$.
\end{theorem}

\begin{proof} 
Let
\((A+\delta A, B + \delta B) \in T^*\Rep(Q,{\bf v})\) where \((\delta
A, \delta B) \in T_{(A,B)}(T^*\Rep(Q,{\bf v})) \cong T^*\Rep(Q,{\bf
v}).\) Then from the formula for $\PhiC$ we see that \((A+\delta A, B + \delta B)\) is contained in
\(\ker(d\PhiC)_{(A,B)}\) exactly if for all \(\ell \in {\mathcal I},\)
\[
\sum_{\inw(a)=\ell} \left(A_a \delta B_a + \delta A_a B_a\right) - 
\sum_{\out(a')=\ell} \left( \delta B_{a'} A_{a'} + B_{a'} \delta
A_{a'} \right) = 0.
\]
On the other hand, \((A+\delta A, B + \delta B)\) is in
$\PhiC^{-1}(0)$ if and only if for all \(\ell \in {\mathcal I}\) 
\[
\sum_{\inw(a)=\ell} (A_a\delta B_a +\delta A_aB_a+\delta A_a\delta B_a) - \sum_{\out(a')=\ell} (B_{a'}\delta A_{a'}+
\delta B_{a'}A_{a'}+\delta B_{a'}\delta A_{a'}) = 0,
\]
where we have used that \(\PhiC(A,B)=0.\) Hence to prove~\eqref{eq:PhiC-level-linearization} near $(A,B)$, it suffices to show that 
\begin{equation}\label{eq:levelset-linearization}
\sum_{\inw(a)=\ell} \delta A_a\delta B_a - \sum_{\out(a')=\ell} \delta B_{a'}\delta A_{a'} = 0
\end{equation}
for all $\ell$ if $(\delta A,\delta B) \in \pi_{{\bf v}^*}^{-1}(A,B)$.  

Let $*$ denote the H-N filtration of $(A,B)$. By assumption, the Harder-Narasimhan filtration of $A$ has H-N $\alpha$-length at most $2$. If the H-N $\alpha$-length is $1$ then the representation is $\alpha$-semistable, and there is nothing to prove. If the H-N $\alpha$-length is $2$, let $V_{\ell} = (V_{\ell})_{1} \oplus (V_{\ell})_{2}$ denote the splitting of $A$ corresponding to the H-N filtration for all \(\ell \in {\mathcal I}.\) Then $\Rep(\bar{Q}, {\bf v})_{*}^{LT}$, written with respect to a basis compatible with $*$, consists of homomorphisms in $\Hom((V_{\out(a)})_{1}, (V_{\inw(a)})_{2})$ for each \(a \in \bar{E}.\) In particular, since this holds for each \(\delta A_{a}, \delta B_{a},\) we may conclude that each summand in~\eqref{eq:levelset-linearization} is separately equal to zero, so the sum is also equal to zero.  
\end{proof}

\begin{remark} 
Although the fibers are vector spaces, as 
just shown, the 
dimension of these fibers may jump in rank along the critical set $C_{{\bf v}^*} \cap
\PhiC^{-1}(0)$. This reflects the singularity of $\PhiCzero$. 
\end{remark}

\subsection{Kirwan surjectivity for representation varieties in cohomology and $K$-theory}\label{subsec:Kirwan}

The Morse theory on the spaces of representations of quivers $\Rep(Q,
{\bf v})$ developed in the previous sections allows us to immediately
conclude surjectivity results for both rational cohomology and
integral $K$-theory rings of the quotient moduli spaces of
representations $\Rep(Q, {\bf v}) \mod_{\alpha} G$.

We refer the reader to \cite{Kir84}, \cite[Section 3]{HarLan07} for a
more detailed account of what follows. The original work of Kirwan
proves surjectivity in rational cohomology in the following general
setting. For $G$ a compact connected Lie group, suppose $(M, \omega)$
is a compact Hamiltonian $G$-space with a moment map \(\Phi: M \to
\g^{*}.\) Assuming $0 \in \g^{*}$ is a regular value of $\Phi$, Kirwan
gives a Morse-theoretic argument using the negative gradient flow of
the norm-square $\|\Phi\|^{2}$ in order to prove that the
$G$-equivariant inclusion \(\iota: \Phi^{-1}(0) \into M\) induces a
ring homomorphism \(\kappa: H^{*}_{G}(M; \Q) \to
H^{*}_{G}(\Phi^{-1}(0); \Q) \cong H^{*}(M \mod G; \Q)\) which is a
surjection. This argument is inductive on the Morse strata
\(S_{\beta} \subseteq M,\) defined as the set of points which limits
to a component $C_{\beta}$ of the critical set $\Crit(\|\Phi\|^{2})$
under the negative gradient flow of $f = \|\Phi\|^{2}$. Here, the
limit point always exists for any initial condition due to the
compactness of the original space $M$.

The surjectivity of the restriction along the inclusion map in the base case follows from
the definition of $f$, since the minimal critical set \(C_{0} =
f^{-1}(0)\) is precisely the level set \(\Phi^{-1}(0).\) The inductive
step uses the long exact sequence of the pair $(M_{\leq \beta}, M_{<
\beta})$ for \(M_{\leq \beta} := \sqcup_{\gamma \leq \beta}
S_{\gamma}, M_{< \beta} := \sqcup_{\gamma < \beta} S_{\gamma}\) for an
appropriate partial ordering on the indexing set of components of
\(\Crit(\|\Phi\|^{2} = \{ C_{\beta}\}\) (see e.g. \cite[Section
3]{HarLan07}). The long exact sequence splits into short exact
sequences, and hence yields surjectivity at each step, by an analysis
of the $G$-equivariant negative normal bundles of $M_{< \beta}$ in
$M_{\leq \beta}$ and an application of the Atiyah-Bott lemma
\cite[Proposition 13.4]{AtiBot83} in rational cohomology. Moreover,
this Morse-theoretic argument of Kirwan, together with an integral
topological $K$-theoretic version of the Atiyah-Bott lemma \cite[Lemma
2.3]{HarLan07}, also implies that the Kirwan surjectivity statement
(in the same setting as above) holds also in integral $K$-theory
\cite[Theorem 3.1]{HarLan07}.

In Kirwan's original manuscript \cite{Kir84} she
assumes that the original symplectic manifold $(M, \omega)$ of
which we take the quotient is {\em compact}. In \cite{HarLan07} this
is slightly generalized to the setting of Hamiltonian $G$-spaces with
proper moment map $\Phi$. However, neither of these assumptions
necessarily holds in our situation, namely, in the K\"ahler quotient
construction of the moduli spaces of quiver representations, since the
original space $\Rep(Q, {\mathbf v})$ is the affine space of quiver
representations, and its moment map $\Phi$ is not
necessarily proper. On the other hand, Kirwan also comments in
\cite[Section 9]{Kir84} that her results and proofs generalize
immediately to any Hamiltonian action $(M, \omega,
\Phi: M \to \g^*)$ provided that one can prove explicitly, in the
given case at hand, that the negative gradient flow with respect to
the norm-square $\|\Phi - \alpha\|^2$ does indeed converge. Since this
is precisely what we proved in Section~\ref{sec:gradientflow}, we have the
following.

\begin{theorem}\label{theorem:Kirwan-linear-general} 
Let $V$ be a hermitian vector space and suppose that a compact
connected Lie group $G$ acts linearly on $V$ via an injective
homomorphism $G \to U(V)$. Let $\Phi: V \to \g^* \cong \g$ be a moment
map for this action and suppose that $G$ acts freely on
$\Phi^{-1}(\alpha)$ for $\alpha \in Z(\g)$. Then the $G$-equivariant
inclusion \(\iota: \Phi^{-1}(\alpha) \into V\) induces a ring
homomorphism in $G$-equivariant rational cohomology 
\[
\kappa: H^*_G(V; \Q) \to H^*_G(\Phi^{-1}(\alpha); \Q) \cong H^*(V
\mod_{\alpha} G;\Q)
\]
which is surjective. 
\end{theorem} 

In particular, in the special case which is our main focus in this
manuscript, we have the following corollary. 

\begin{corollary}\label{corollary:Kirwan-Kahler-quivers}
Let \(Q = ({\mathcal I}, E)\) be a quiver, and \({\bf v} \in
\Z^{\mathcal I}_{\geq 0}\) a dimension vector. Let $\Rep(Q, {\bf v})$
be its associated representation space, \(\Phi: \Rep(Q,
{\bf v}) \to \g^{*} \cong \g \cong \prod_{\ell \in {\mathcal I}}
\u(V_{\ell})\) the moment map for the standard Hamiltonian
action of \(G = \prod_{\ell \in {\mathcal I}} U(V_{\ell})\) on
$\Rep(Q, {\bf v})$, and \(\alpha \in Z(\g)\)
a central parameter such that $G$ acts freely on
$\Phi^{-1}(\alpha)$. Then the $G$-equivariant inclusion \(\iota:
\Phi^{-1}(\alpha) \into \Rep(Q, {\bf v})\) induces a ring homomorphism
in $G$-equivariant rational cohomology
\[
\kappa: H^{*}_{G}(\Rep(Q, {\bf v}); \Q) \to H^{*}_{G}(\Phi^{-1}(\alpha); \Q) \cong H^{*}(\Rep(Q, {\bf v}) \mod_{\alpha} G; \Q)
\]
which is surjective.
\end{corollary}

In the case of topological integral $K$-theory, we must restrict to
the case of quivers \(Q = ({\mathcal I}, E)\) such that the components
of the critical sets of $\|\Phi\|^{2}$ are compact; this is because
the $K$-theoretic Atiyah-Bott lemma of \cite{HarLan07} requires a
compact base for its equivariant bundles. It is known that this
condition is satisfied if, for instance, \(Q = ({\mathcal I}, E)\) has
no oriented cycles \cite{HauPro05}. Hence we also have the following
corollary.

\begin{theorem}\label{theorem:Kirwan-Kahler-quivers-Kthy}
Let \(Q = ({\mathcal I}, E), {\bf v} \in \Z^{\mathcal I}_{\geq 0},
 G = \prod_{\ell \in {\mathcal I}} U(V_{\ell}),
 \Phi: \Rep(Q, {\bf v}) \to \g^{*},\) and \(\alpha \in Z(\g)\) 
be as in Corollary~\ref{corollary:Kirwan-Kahler-quivers}. Assume
 further that \(Q = ({\mathcal I}, E)\) has no oriented cycles. Then
 the $G$-equivariant inclusion \(\iota: \Phi^{-1}(\alpha) \into
 \Rep(Q, {\bf v})\) induces a ring homomorphism in $G$-equivariant
 integral topological $K$-theory
\[
\kappa: K^{*}_{G}(\Rep(Q, {\bf v}); \Q) \to K^{*}_{G}(\Phi^{-1}(\alpha); \Q) \cong K^{*}(\Rep(Q, {\bf v}) \mod_{\alpha} G; \Q)
\]
which is surjective.
\end{theorem} 

\subsection{Comparison with Reineke's results}\label{sec:reineke}

In \cite{Rei03} Reineke describes a method for computing the Betti numbers of the quotient $\Rep(Q, {\bf v}) \mod_\alpha G$ for a choice of parameter $\alpha$ where $\Rep(Q, {\bf v})^{\alpha-ss} = \Rep(Q,{\bf v})^{\alpha-st}$ (we call these ``generic'' parameters). The purpose of this section is to show that our Morse-theoretic approach reproduces Reineke's results in the case of generic parameters and generalises them to the case where the parameter need not be generic. When the parameter is not generic, we compute the $G$-equivariant Betti numbers of $\Rep(Q, {\bf v})^{\alpha-ss}$ instead of the ordinary Betti numbers of $\Rep(Q, {\bf v}) \mod_\alpha G$.

Proposition \ref{prop:codimension} leads to the following Morse-theoretic formula for the $G$-equivariant Poincar\'e polynomial of $\Rep(Q, {\bf v})^{\alpha-ss}$.

\begin{equation}\label{eqn:Poincare-polynomial}
P_t^G(\Rep(Q, {\bf v})^{\alpha-ss}) = P_t(BG) - \sum_{\bf v^*} t^{2d(Q, {\bf v}, {\bf v^*})} P_t^G (\Rep(Q, {\bf v})_{\bf v^*}) 
\end{equation}

Note that by Proposition \ref{prop:cohomology-stratum}, we can inductively compute $P_t^G (\Rep(Q, {\bf v})_{\bf v^*})$. 

\begin{remark}
For convenience we use the group $G$ in \eqref{eqn:Poincare-polynomial} instead of the quotient $PG = G / S^1$, which acts freely on $\Rep(Q, {\bf v})^{\alpha-st}$ (where $S^1$ is the subgroup of scalar multiples of the identity). This leads to an extra factor of $P_t(BU(1))$ on both sides of \eqref{eqn:Poincare-polynomial}. For example, when $\alpha$ is generic then $P_t^G(\Rep(Q, {\bf v})^{\alpha-st}) = \frac{1}{1-t^2} P_t(\Rep(Q, {\bf v}) \mod_\alpha G) $. 
\end{remark}

Note that Reineke's formula for the Betti numbers in \cite{Rei03} is expressed differently to that in \eqref{eqn:Poincare-polynomial}. To see that they are equivalent when $\alpha$ is generic, one has to combine Proposition 4.8, Theorem 5.1, Corollary 6.2 and Theorem 6.7 of \cite{Rei03}.

\subsection{Equivariant Morse theory and equivariant Kirwan surjectivity}\label{subsec:eqvt-Morse} 

In the study of the topology of quotients via the Morse theory of the
moment map, it is often possible to make Kirwan surjectivity
statements onto the equivariant cohomology of the quotient with
respect to some residual group action, not just the ordinary
(non-equivariant) cohomology of the quotient. To prove such an
equivariant version of Kirwan surjectivity, however, it is of course
necessary to check that the relevant Morse-theoretic arguments may all
be made equivariant with respect to the appropriate extra symmetry.
This is the goal of this section.

First, we briefly recall the statement of equivariant Kirwan surjectivity. Suppose $(M, \omega)$ is a symplectic manifold, and further suppose that we have Hamiltonian actions of two compact connected Lie groups $G$ and $G'$ on $(M, \omega)$ with moment maps $\Phi_G$ and $\Phi_{G'}$ respectively. Assume that the actions of $G$ and ${G'}$ commute, and that $\Phi_G$ and $\Phi_{G'}$ are ${G'}$-invariant and $G$-invariant, respectively. Then there is a residual (Hamiltonian) ${G'}$-action on the $G$-symplectic quotient $M \mod G$, and we may ask the following question: is the ring map induced by the natural $G \times {G'}$-equivariant inclusion $\Phi_G^{-1}(0) \into M$, 
\begin{equation}\label{eq:eqvt-Kirwan}
\kappa_{G'}: H^*_{G \times {G'}}(M;\Q) \to H^*_{G \times {G'}}(\Phi_G^{-1}(0); \Q)
\end{equation}
surjective? Note that if $G$ acts freely on $\Phi_G^{-1}(0)$, then as usual, the target of~\eqref{eq:eqvt-Kirwan} is isomorphic to $H^*_{G'}(M \mod G; \Q)$. Hence $\kappa_{G'}$ is the ${G'}$-equivariant version of the usual Kirwan surjectivity question.

\begin{remark} 
In the case of Nakajima quiver varieties, there is a well-studied extra $S^1$-action commuting with the usual $G = \prod_{\ell \in {\mathcal I}} U(V_\ell)$-action on $T^*\Rep(Q, {\bf v})$ which acts by spinning only the fiber directions of the cotangent bundle with weight $1$. It is straightforward to check that this $S^1$-action and the given $G$-action satisfy all the hypotheses required for the question given in~\eqref{eq:eqvt-Kirwan} to make sense, so this is a specific example of the situation under discussion. 
\end{remark}

\begin{theorem}\label{theorem:eqvt-Morse}. 
Let \({G'} \subseteq U(\Rep(Q, {\bf v}))\) be a Lie subgroup such that \(G = \prod_{\ell \in {\mathcal I}} U(V_{\ell}) \into U(\Rep(Q, {\bf v}))\) and ${G'}$ commute, let $\Phi_{G}$ denote the usual induced $G$-moment map on $\Rep(Q, {\bf v})$, and let \(\alpha \in Z(\g).\) Then the inclusion $\iota: \Phi_{G}^{-1}(\alpha) \into \Rep(Q, {\bf v})$ induces a ring homomorphism 
\[
\iota^{*}: H^{*}_{G \times {G'}}(\Rep(Q, {\bf v}; \Q) \to H^{*}_{G \times {G'}}(\Phi_{G}^{-1}(\alpha); \Q) 
\]
which is a surjection. In particular, if $G$ acts freely on $\Phi_{G}^{-1}(\alpha)$, the composition of $\iota^{*}$ with the isomorphism \(H^{*}_{G \times {G'}}(\Phi_{G}^{-1}(\alpha); \Q) \cong H^{*}_{{G'}}(\Rep(Q, {\bf v}) \mod_{\alpha} G ; \Q)\) induces a surjection of rings 
\begin{equation}\label{eq:K-eqvt-Kirwan}
\kappa_{{G'}}: H^{*}_{G \times {G'}}(\Rep(Q, {\bf v}; \Q) \to H^{*}_{{G'}}(\Rep(Q, {\bf v}) \mod_{\alpha} G ; \Q).
\end{equation}
Moreover, if the quiver $Q$ has no oriented cycles, the inclusion $\iota$ induces a surjection 
\[
\iota^{*}: K^{*}_{G \times {G'}}(\Rep(Q, {\bf v}) \to K^{*}_{G \times {G'}}(\Phi_{G}^{-1}(0)) 
\]
and if $G$ acts freely on $\Phi_{G}^{-1}(0)$ then the composition of $\iota^{*}$ with the isomorphism \(K^{*}_{G \times {G'}}(\Phi_{G}^{-1}(0)) \cong K^{*}_{{G'}}(\Rep(Q, {\bf v}) \mod_{\alpha} G)\) induces a surjection of rings 
\begin{equation}\label{eq:K-eqvt-Kirwan-Kthy}
\kappa_{{G'}}: K^{*}_{G \times {G'}}(\Rep(Q, {\bf v}) \to K^{*}_{{G'}}(\Rep(Q, {\bf v}) \mod_{\alpha} G).
\end{equation}
\end{theorem}

\begin{proof} 

From the arguments given in Section~\ref{subsec:Kirwan}, it is evident that it suffices to show that all the steps in the Morse theory of Section~\ref{subsec:Kirwan} may be made ${G'}$-equivariant. 
We begin by showing that $\Phi_{G}$ is ${G'}$-invariant, i.e. for all \(k \in {G'}, u \in \g, A \in \Rep(Q, {\bf v}),\) we have 
\[
\langle \Phi_{G}(k \cdot A), u \rangle_{\g} = \langle \Phi_{G}(A), u \rangle_{\g}.
\]
By definition \(\langle \Phi_{G}(A), u \rangle_{\g} = \langle \Phi(A), \iota_{\g}(u) \rangle_{\u(\Rep(Q, {\bf v})}\) where \(\iota_{\g}: \g \into \u(\Rep(Q, {\bf v}))\) is the natural inclusion, so we may compute
\begin{align}
\begin{split}
\langle \Phi_{G}(k \cdot A), u \rangle_{\g} & =  \langle \Ad_{k} \Phi(A), \iota_{\g}(u) \rangle_{\u(\Rep(Q, {\bf v}))} \\
& = \langle \Phi(A), \Ad_{k^{-1}} \iota_{\g}(u) \rangle_{\u(\Rep(Q, {\bf v}))} \\
& = \langle \Phi(A), \iota_{\g}(u) \rangle_{\u(\Rep(Q, {\bf v}))} \\
& = \langle \Phi_{G}(A), u \rangle_{\g},\\
\end{split}
\end{align}
for all \(k \in {G'}, u \in \g, A \in \Rep(Q, {\bf v}),\) as desired, where in the second-to-last equality we use that $G$ and ${G'}$ commute. In particular, the norm-square $\|\Phi_{G} - \alpha\|^{2}$ for any \(\alpha \in Z(\g)\) is also ${G'}$-invariant, since the metric on $\g$ is induced from that on $\u(\Rep(Q, {\bf v}))$. Moreover, since \({G'} \subseteq U(\Rep(Q, {\bf v})),\) by definition it preserves the metric on $\Rep(Q, {\bf v})$, so the negative gradient vector field of the function $\|\Phi_{G} - \alpha\|^{2}$ is ${G'}$-invariant, implying that the associated flow is ${G'}$-equivariant and the associated Morse strata are ${G'}$-invariant.

To complete the argument, it suffices to note that the $G \times
{G'}$-invariance of the metric implies that the negative normal bundles
at the critical sets are $G \times {G'}$-equivariant bundles, and that if
there exists an \(S^{1} \subseteq G\) satisfying the hypotheses of the
$G$-equivariant Atiyah-Bott lemmas, then the same \(S^{1} \subseteq G
\subseteq G \times {G'}\) satisfies the hypotheses of the \(G \times
{G'}\)-equivariant Atiyah-Bott lemma. This is true in either rational
cohomology \cite[Proposition 13.4]{AtiBot83} or in integral
topological $K$-theory \cite[Lemma 2.3]{HarLan07}. The result
follows.
\end{proof}

%\bibliographystyle{plain}
%\bibliography{ref}

\def\cprime{$'$}

\end{document}